\documentclass[12pt]{report}   

\usepackage{setspace}
\usepackage{buthesis}

\usepackage{amsmath,amssymb,graphicx,amsfonts,amsthm,amsbsy,array,graphicx}
\usepackage{latexsym,amscd}
\usepackage{verbatim}
\usepackage{epsfig}
\usepackage{wrapfig}
\usepackage{changebar}

\usepackage{makeidx}
\makeindex

\usepackage{bm}

\usepackage{booktabs}

\usepackage{mathtools}
\mathtoolsset{showonlyrefs}

\usepackage{enumerate}

\usepackage{hyperref}
\usepackage[usenames, dvipsnames]{xcolor}

\definecolor{darkblue}{rgb}{0.0,0.0,0.3}
\hypersetup{colorlinks,breaklinks,
  linkcolor=darkblue,urlcolor=darkblue,
anchorcolor=darkblue,citecolor=darkblue}

\usepackage[comma,numbers,square,sort&compress]{natbib}

\allowdisplaybreaks

\makeatletter
\newcommand*{\myfnsymbolsingle}[1]{%
  \ensuremath{%
    \ifcase#1
    \or 
      \dagger
    \or 
      \ddagger
    \or 
      \dagger\dagger
    \else 
      \@ctrerr
    \fi
  }%
}
\makeatother

\usepackage[Bjarne]{ThesisFncychap}
\include{BjarneThesisTitles}

\newtheorem{theorem}{Theorem}[section]
\newtheorem*{theorem*}{Theorem}
\newtheorem{lemma}[theorem]{Lemma}
\newtheorem*{lemma*}{Lemma}
\newtheorem{proposition}[theorem]{Proposition}

\newtheorem*{proposition*}{Proposition}
\newtheorem{corollary}[theorem]{Corollary}
\newtheorem*{corollary*}{Corollary}
\newtheorem{claim}[theorem]{Claim}
\newtheorem{conjecture}[theorem]{Conjecture}

\theoremstyle{definition}
\newtheorem{remark}[theorem]{Remark}
\newtheorem*{remark*}{Remark}

\makeatletter\@addtoreset{equation}{section}\makeatother

\renewcommand{\Im}{\operatorname{Im}}
\renewcommand{\Re}{\operatorname{Re}}
\DeclareMathOperator{\SL}{SL}
\DeclareMathOperator{\GL}{GL}

\DeclareMathOperator*{\Res}{Res}
\DeclareMathOperator*{\const}{const}
\DeclareMathOperator*{\Vol}{Vol}
\newcommand{\discrete}{\operatorname{discrete}}
\newcommand{\kron}[2]{\left(\frac{#1}{#2}\right)}

\usepackage[titles]{tocloft}

\dissertation 

\title{On Some Variants of the Gauss Circle Problem}
\author{David Lowry-Duda}
\degree{Doctor of Philosophy}
\department{the Department of Mathematics}
\previousdegrees{%
		B.S.\ in Applied Mathematics,
      Georgia Institute of Technology, Atlanta, GA, 2011 \\
    B.S.\ in International Affairs and Modern Languages,
      Georgia Institute of Technology, Atlanta, GA, 2011 \\
    M.Sc.\ in Mathematics, Brown University, Providence, RI, 2015
  }
\thesismonth{May} \thesisyear{2017}

\begin{document}


\begin{preliminaries}
\maketitle

\copyrightpage

\begin{abstract}
  
The Gauss Circle Problem concerns finding asymptotics for the number of lattice point
lying inside a circle in terms of the radius of the circle.
The heuristic that the number of points is very nearly the area of the circle is
surprisingly accurate.
This seemingly simple problem has prompted new ideas in many areas of number theory and
mathematics, and it is now recognized as one instance of a general phenomenon.
In this work, we describe two variants of the Gauss Circle problem that exhibit similar
characteristics.

The first variant concerns sums of Fourier coefficients of $\GL(2)$ cusp forms.
These sums behave very similarly to the error term in the Gauss Circle problem.
Normalized correctly, it is conjectured that the two satisfy essentially the same
asymptotics.
We introduce new Dirichlet series with coefficients that are squares of partial sums of
Fourier coefficients of cusp forms.
We study the meromorphic properties of these Dirichlet series and use these series to give
new perspectives on the mean square of the size of sums of these Fourier coefficients.
These results are compatible with current conjectures.

The second variant concerns the number of lattice points of bounded size on one-sheeted
hyperboloids.
This problem is very similar to counting the number of lattice points
within a spheres of various dimensions, except with the additional constraint of lying on
a hyperboloid.
It turns out that this problem is equivalent to estimating sums of the shape
$\sum r_d(n^2 + h)$, where $r_d(m)$ is the number of representations of $m$ as a sum of
$d$ squares.
We prove improved smoothed and sharp estimates of the second moment of these sums,
yielding improved estimates of the number of lattice points.

In both variants, the problems are related to modular forms and, in particular, to shifted
convolution sums associated to these modular forms.
We introduce new techniques and new Dirichlet series which are quite general.
At the end of this work, we describe further extensions and questions for further
investigation.


  \thispagestyle{empty}
\end{abstract}

\begin{dedication}
  
To a string of educators with big hearts, nurturing a childish curiosity in all things:

{\centering
Julie Gray,\\
Dr.\ Marge Counts,\\
Richard O'Brien,\\
Dr.\ Bob Covel,\\
Dr.\ Steve Clark,\\
Dr.\ Matt Baker.\\}

\end{dedication}

\begin{acknowledgments}

Before all others, I thank my extended mathematical family for making me believe in the
family metaphor.
Tom, Chan, Mehmet, and Li-Mei created such a fun and friendly atmosphere that I couldn't
help but fall in, immediately at home.
Thank you to my advisor Jeff, whose advice and worldview has had profound positive effects
on my experience as a graduate student.

Much of the work in this thesis was either directly or indirectly inspired by
collaboration with my mathematical family.
Chapter~\ref{c:sums} was born out of a question asked by Jeff Hoffstein at his 61st
birthday conference concerning the sign changes of sums of coefficients of cuspforms,
which I thought about immediately afterwards with Tom Hulse, Chan Ieong Kuan, and Alex
Walker.
This eventually led to both the ideas of Chapter~\ref{c:sums} and the papers briefly
described in Chapter~\ref{c:sums_apps}.
The main question of Chapter~\ref{c:hyperboloid} was brought to my attention by Chan Ieong
Kuan and Mehmet K{\i}ral several years ago.
During our first (somewhat misguided) attempts, I learned a lot about mathematics and
collaboration, and I owe them much.

Thank you to Paul and Miles, for the countless hours both in and outside of work together,
and for talking endlessly about the Mean Value Theorem with me.
I would also like to thank all my other collaborators, for helping make math fun.

Thank you to Jill, Joe, Mike, Reinier, and Dinakar for serving on my committees.

Thank you in particular to my little brother Alex for the companionship and collaboration,
and to my soon-to-be little sister-in-law-in-math Sara for tolerating our many months
living together.

Finally, I want to thank Joanna, my parents, and my family for their love and support.
They don't always understand what I think about, but they care anyway.

I love you all.

\noindent
\begin{tabular*}{\textwidth}{c}
  \midrule
\end{tabular*}
This material is based upon work supported by the National Science Foundation Graduate
Research Fellowship Program under Grant No. DGE 0228243.

\end{acknowledgments}

\begin{spacing}{1}
  \tableofcontents
  \clearpage{\pagestyle{empty}\cleardoublepage}

  \footnotesize
  \fontsize{11.5pt}{12.5pt}\selectfont

  \normalsize
\end{spacing}

\end{preliminaries}

\pagestyle{myheadings}



\chapter{Introduction}\label{c:introduction}

\hfill
\begin{minipage}{4in}
  \small
  \singlespacing{}
  Charles Darwin had a theory that once in a while one should perform a damn-fool
  experiment.
  It almost always fails, but when it does come off is terrific.\\

  Darwin played the trombone to his tulips.
  The result of this particular experiment was negative.
  \vspace{.3ex}
  \begin{flushright}
    {Littlewood, \textit{A Mathematician's Miscellany}}
  \end{flushright}
\end{minipage}

In this thesis, we study and introduce new methods to study problems that are closely
related to the Gauss Circle problem.
In this introductory chapter, we motivate and explain the Gauss Circle problem and how it
relates to the other problems described in later chapters.
In Section~\ref{sec:intro:GaussCircleExplanation}, we give a historical overview of the
classical Gauss Circle and Dirichlet Hyperbola problems.

\section{Gauss Circle Problem}\label{sec:intro:GaussCircleExplanation}

The Gauss Circle problem is the following seemingly innocent question:
\index{Gauss Circle problem}
\begin{quote}
  How many integer lattice points lie inside or on the circle of radius $\sqrt R$ centered
  at the origin? That is, how many points $(x,y) \in \mathbb{Z}^2$ satisfy $x^2 + y^2 \leq
  R$?
\end{quote}
We will use $S_2(R)$ to denote the number of lattice points inside or on the circle of
radius $\sqrt R$ centered at the origin.
Intuitively, it is clear that $S_2(R)$ should be approximately the area of the circle,
$\Vol B(\sqrt R) = \pi R$.
This can be made rigorous by thinking of each lattice point as being the center of a $1
\times 1$ square in the plane and counting those squares fully contained within the circle
and those squares lying on the boundary of the circle.
Using this line of thought, Gauss proved (in 1798) that
\begin{equation}
  \lvert S_2(R) - \Vol B(\sqrt R) \rvert \ll \sqrt R,
\end{equation}
so that the \emph{lattice point discrepancy}\index{lattice point discrepancy} between the
number of lattice points within the circle and the area of the circle is bounded by some
constant times the circumference.

It may appear intuitive that this is the best one could hope for.\footnote{Indeed, I am
  not currently aware of an intuitive heuristic that explains \emph{why} the actual error
term is so much better.}
No further progress was made towards Gauss Circle problem until 1906, when
Sierpi\'{n}ski~\cite{sierpinski1906pewnem} showed that
\begin{equation}
  \lvert S_2(R) - \Vol B(\sqrt R) \rvert \ll R^{\frac{1}{3}},
\end{equation}
a significant improvement over the bound from Gauss.
This is already remarkable --- somehow lattice points in the plane are distributed in the
plane in a way that a growing circle includes and omits points in a way that partially
offsets each other.

It is natural to ask: What is the correct exponent of growth on the error term?
Hardy, Littlewood, and Cram\'{e}r~\cite{cramer1922} proved that, \emph{on average}, the
correct exponent is $\frac{1}{4}$ when they proved
\begin{equation}\label{eq:intro:hardy_littlewood_meansquare}
  \int_0^X \lvert S_2(r) - \Vol B(\sqrt r) \rvert^2 dr = c X^{\frac{3}{2}} +
  O(X^{\frac{5}{4} + \epsilon})
\end{equation}
for some constant $c$.
At approximately the same time, Hardy~\cite{Hardy1917} showed that
\begin{equation}
  \lvert S_2(R) - \Vol B(\sqrt R) \rvert = \Omega(R^{\frac{1}{4}}),
\end{equation}
so that $\frac{1}{4}$ is both a lower bound and the average bound.

At the same time (and often, even within the same works, such as in Hardy's work on the
average and minimum values of the error in~\cite{Hardy1917}), mathematicians were studying
Dirichlet's Divisor problem.
Let $d(n)$ denote the number of positive divisors of $n$.
Then Dirichlet's Divisor problem is to determine the average size of $d(n)$ on integers
$n$ up to $R$.\index{Dirichlet Divisor problem}\index{Dirichlet Hyperbola problem | \\see
{Dirichlet Divisor problem} }
Noting that
\begin{equation}
  \sum_{n \leq R} d(n) =
  \sum_{n \leq R} \sum_{d \mid n} 1 =
  \sum_{d \leq R} \Big\lfloor \frac{R}{d} \Big\rfloor,
\end{equation}
we see that this is equivalent to counting the number of positive integer lattice points
under the hyperbola $xy = R$.
For this reason, some refer to this as Dirichlet's Hyperbola problem.

As noted already by Hardy~\cite{Hardy1917}, it's known that
\begin{equation}
  \sum_{n \leq R} d(n) = cR \log R + c' R + O(R^{\frac{1}{2}}),
\end{equation}
and that
\begin{equation}
  \int_1^X \Big\lvert \sum_{n \leq r} d(n) - c r \log r - c' r \Big\rvert^2 dr = c'' X^{\frac{3}{2}} +
  O(X^{\frac{5}{4} + \epsilon}),
\end{equation}
analogous to the Gauss Circle problem.
Attempts to further understand the error terms in the Gauss Circle and Dirichlet Hyperbola
problems have indicated that there is a strong connection between the two, and often an
improvement to one of the problems yields an improvement to the other.

It is interesting to note that much of the early work of Hardy and Littlewood on
the Gauss Circle and Dirichlet Hyperbola problems occurred from 1914 to 1919, which are
the years when Ramanujan studied and worked with them at Cambridge.
Of particular import is a specific identity inspired by Ramanujan (as noted by Hardy
in~\cite{Hardy1959onRamanujan}) that is now sometimes called ``Hardy's
Identity''\footnote{This is another instance of Stigler's Law of Eponymy.},
which states that
\begin{equation}\label{eq:intro:HardysIdentity}
  S_2(R) - \Vol B(\sqrt R) =  \sqrt{R} \sum_{n \geq 1} \frac{r_2(n)}{n^{1/2}} J_1(2\pi
  \sqrt{n R}),
\end{equation}
in which $r_2(n)$ is the number of ways of writing $n$ as a sum of $2$ squares and $J_\nu$ is
the ordinary Bessel function
\begin{equation}
  J_\nu(z) := \sum_{n \geq 0} \frac{(-1)^n}{\Gamma(n+1) \Gamma(\nu + n + 1)}
  (z/2)^{\nu + 2n}.
\end{equation}
Ivi\'c has noted in~\cite{ivic2004lattice} and~\cite{ivic1996} that  almost all
significant progress towards both the Gauss Circle and Dirichlet Hyperbola problems have
come from identities and approaches similar to~\eqref{eq:intro:HardysIdentity}, though
sometimes obscured through technical details.

\subsection{An Early Connection to Modular Forms}

One of the topics that Ramanujan devoted himself to was what we now call the ``Ramanujan
tau function,''\index{Ramanujan $\tau$}
which can be defined by matching coefficients in
\begin{equation}
  \sum_{n \geq 1} \tau(n) q^n = q \prod_{n \geq 1} (1 - q^n)^{24}.
\end{equation}
Ramanujan noticed or conjectured that this function satisfies many nice properties,
such as being multiplicative.
The individual $\tau(n)$ satisfy the bound
\begin{equation}
  \tau(n) \ll n^{\frac{11}{2} + \epsilon},
\end{equation}
and numerical experimentation might lead one to conjecture that
\begin{equation}
  \sum_{n \leq R} \tau(n) \ll R^{\frac{11}{2} + \frac{1}{4} + \epsilon}.
\end{equation}
Just like the Gauss Circle and Dirichlet Hyperbola problem, it appears that summing over
$R$ many terms contributes only $\frac{1}{4}$ to the exponent in the size.
Further, Chandrasekharan and Narasimhan~\cite{chandrasekharan1962functional,
chandrasekharan1964mean} showed a result analogous to the average estimate of Hardy and
Littlewood, proving that
\begin{equation}
  \int_0^X \big\lvert \sum_{n \leq r} \tau(n) \big\rvert^2 dr = c X^{11 + \frac{3}{2}} + O(X^{12 +
  \epsilon}).
\end{equation}
This also indicates that the average additional exponent is $\frac{1}{4}$.

We now recognize that this is another analogy to the Gauss Circle problem, except that in
this case there is no main term.
This analogy readily generalizes.

The Ramanujan $\tau$ function appears as the coefficients of the unique weight $12$
holomorphic cusp form on $\SL(2, \mathbb{Z})$, usually written
\begin{equation}
  \Delta(z) = \sum_{n \geq 1} \tau(n) e(nz),
\end{equation}
where here and throughout this thesis, $e(nz) = e^{2\pi i n z}$.
Generally, we can consider a holomorphic cusp form $f$ on a congruence subgroup $\Gamma
\subseteq \SL(2, \mathbb{Z})$ of weight $k$, with Fourier expansion
\begin{equation}
  f(z) = \sum_{n \geq 1} a(n) e(nz).
\end{equation}
Let $S_f(R)$ denote the partial sum of the first $R$ Fourier coefficients,
\begin{equation}
  S_f(R) := \sum_{n \leq R} a(n).
\end{equation}
Then Chandrasekharan and Narasimhan also show that
\begin{equation}
  \int_0^X \lvert S_f(r) \rvert^2 dr = c X^{k + \frac{3}{2}} + O(X^{k + \epsilon}),
\end{equation}
indicating that \emph{on average}, the partial sums satisfy
\begin{equation}
  S_f(R) \ll R^{\frac{k-1}{2} + \frac{1}{4} + \epsilon}.
\end{equation}
In other words, partial sums of coefficients of cusp forms appear to satisfy a Gauss
Circle problem type growth bound.

Further, the Gauss Circle problem can be restated as a problem estimating
\begin{equation}
  S_2(R) = \sum_{n \leq R} r_2(n),
\end{equation}
where $r_2(n)$ denotes the number of representation of $n$ as a sum of two squares as
above.
The coefficients $r_2(n)$ appear as the coefficients of a the modular form $\theta^2(z)$,
so the analogy between $S_f$ and $S_2$ is very strong.
However $\theta^2$ is not cuspidal, so there are some differences.

In Chapter~\ref{c:sums}, we consider this ``Cusp Form Analogy'' and study $S_f(n)$.
We introduce new techniques that are fundamentally different than most techniques employed
in the past.
Of particular interest is the introduction of Dirichlet series of the form
\begin{equation}
  \sum_{n \geq 1} \frac{S_f(n)^2}{n^s},
\end{equation}
including their meromorphic continuation and analysis.
In Chapter~\ref{c:sums_apps}, we give an overview of the completed and currently planned
applications of the analysis and techniques for studying $S_f(n)$ in Chapter~\ref{c:sums}.
This includes an overview of several recent papers of the author and his collaborators
where these techniques have been used successfully.

\subsection{Further Generalizations of the Gauss Circle Problem}

A very natural generalization of the Gauss Circle problem is to higher dimensions.
We will call the following the Gauss $d$-Sphere problem:

\begin{quote}
  How many integer lattice points lie inside or on the $d$-dimensional sphere of radius
  $\sqrt R$ centered at the origin? That is, how many points $\bm x \in \mathbb{Z}^d$
  satisfy $\| \bm x \|^2 \leq R$?
\end{quote}\index{Gauss $d$-Sphere problem}

We use $S_d(R)$ to denote the number of lattice points inside or on the $d$-sphere of
radius $\sqrt R$, as occurs in the Gauss $d$-Sphere problem.
As with the Gauss Circle problem, it is intuitively clear that $S_d(R) \approx \Vol
B_d(\sqrt R)$, where we use $B_d(\sqrt R)$ to denote a $d$-dimensional ball of radius
$\sqrt R$ centered at the origin.
Just as with the Gauss Circle problem, the true goal of the Gauss $d$-Sphere problem
is to understand the lattice point discrepancy $S_d(R) - \Vol B_d(\sqrt R)$.

The Gauss $d$-Sphere problem is most mysterious for low dimensions, but perhaps the
dimension $3$ Gauss Sphere problem is the most mysterious.
For an excellent survey on the status of this problem, see~\cite{ivic2004lattice}.

The Gauss $d$-Sphere problem is not considered in detail in this thesis, but in
Chapter~\ref{c:sums_apps} a new approach on aspects of the Gauss $d$-Sphere problem is
outlined which builds on the techniques of Chapter~\ref{c:sums}.
This is a project under current investigation by the author and his collaborators.

Closely related to the Gauss $d$-Sphere problem is the problem of determining the number
of lattice points that lie within $B_d(\sqrt R)$ \textbf{and} which lie on a one-sheeted
hyperboloid\index{one-sheeted hyperboloid problem}
\index{H@$\mathcal{H}_{d,h}$}
\begin{equation}
  \mathcal{H}_{d,h} := X_1^2 + \cdots + X_{d-1}^2 = X_d^2 + h
\end{equation}
for some $h \in \mathbb{Z}_{\geq 1}$.
We use $N_{d,h}(R)$ to denote the number of lattice points within $B_d(\sqrt R)$ and on
$\mathcal{H}_{d,h}$.\index{N@$N_{d,h}$}
This is a problem within the larger class of constrained lattice counting problems.

The dimension $3$ one-sheeted hyperboloid problem is the first nontrivial dimension, and
just as with the standard Gauss $d$-Sphere problem, the dimension $3$ hyperboloid is the
most enigmatic.
Very little is currently known.
A significant reason for the mystery comes from the fact that this is too small of a
dimension to apply the Circle Method of Hardy and Littlewood, and there doesn't seem to be
an immediate analogue of Hardy's Identify~\eqref{eq:intro:HardysIdentity}.
In fact, it's not quite clear what the correct separation between the main term and error
term should be.
The best known result is the recent result of Oh and Shah~\cite{ohshah2014}, which uses
ergodic methods to prove that
\begin{equation}\label{eq:intro:oh_shah}
  N_{3,h}(R) = c R^{\frac{1}{2}} \log R + O\big( R^{\frac{1}{2}} (\log R)^{\frac{3}{4}}
  \big)
\end{equation}
when $h$ is a positive square.

From the perspective of modular forms, the underlying modular object is the modular form
$\theta^{d-1}(z) \overline{\theta(z)}$.
But in contrast to the variants of the Gauss Circle problem discussed above, $N_{d,h}(R)$
is encoded within the $h$th Fourier coefficient of $\theta^{d-1}\overline{\theta}$, which
affects the shape of the analysis significantly.

In Chapter~\ref{c:hyperboloid}, we consider the problem of estimating the number
$N_{d,h}(R)$ of points on one-sheeted hyperboloids.
Along the way, we study the Dirichlet series
\begin{equation}
  \sum_{n \geq 1} \frac{r_{d-1}(n^2 + h)}{(2n^2 + h)^s}
\end{equation}
and its meromorphic properties, and use it to get improved estimates for $N_{d,h}(R)$.

\section{Outline and Statements of Main Results}

\subsubsection*{Chapter 2.}

Chapter~\ref{c:background} consists of background information used in later sections.
In particular, a variety of properties concerning Eisenstein series are discussed and
referenced from the literature.
The second half of Chapter~\ref{c:background} concerns a complete description of three
Mellin integral transforms.
It is shown that two smooth integral transforms can be used in tandem to establish a sharp
cutoff result, which is employed in Chapter~\ref{c:hyperboloid} and in the applications
described in Chapter~\ref{c:sums_apps}.

\subsubsection*{Chapters 3 and 4.}

Chapter~\ref{c:sums} is closely related to the journal article~\cite{hkldw}, which was
published in 2017.
In this Chapter, we introduce and study the meromorphic properties of the Dirichlet series
\begin{equation}\label{eq:intro:sums_dirichlet_series}
  \sum_{n \geq 1} \frac{S_f(n)}{n^s}, \quad \sum_{n \geq 1} \frac{S_f(n)
  \overline{S_g(n)}}{n^s}, \quad \text{and} \quad \sum_{n \geq 1}
  \frac{S_f(n)S_g(n)}{n^s},
\end{equation}
where $f$ and $g$ are weight $k$ cuspforms on a congruence subgroup $\Gamma \subseteq
\SL(2, \mathbb{Z})$ and where $S_f(n)$ and $S_g(n)$ denote the partial sums of the first
$n$ Fourier coefficients of $f$ and $g$, respectively.

The first major result is in Theorem~\ref{thm:Wsfgmero} and
Corollary~\ref{cor:DsSfSg_has_meromorphic}, which show that the series
in~\eqref{eq:intro:sums_dirichlet_series} have meromorphic continuation to the plane with
understandable analytic properties.

As a first application of this meromorphic continuation, we prove the major analytic
result of the chapter in Theorem~\ref{thm:second_moment_Sf}, which states that
\begin{equation}
  \sum_{n \geq 1} \frac{S_f(n) \overline{S_g(n)}}{n^{k-1}} e^{-n/X} = C X^{\frac{3}{2}} +
  O(X^{\frac{1}{2} + \epsilon}),
\end{equation}
and related results, for an explicit constant $C$.
In terms of the analogy to the Gauss Circle problem, this is a smoothed second moment and
is comparable in nature to the mean-square moment bound of Hardy and
Littlewood~\eqref{eq:intro:hardy_littlewood_meansquare}.
As the Dirichlet series in~\eqref{eq:intro:sums_dirichlet_series} are new, this smoothed
result is a new type of result.

Chapter~\ref{c:sums_apps} is a description of further work using the Dirichlet
series~\eqref{eq:intro:sums_dirichlet_series} and its meromorphic continuation.
This includes the work in the papers~\cite{hkldwShort} and~\cite{hkldwSigns}.
Further applications are described, as well as preliminary results from the author and his
collaborators on this topic.

\subsubsection*{Chapters 5 and 6}

Chapter~\ref{c:hyperboloid} focuses on studying the One-Sheeted Hyperboloid problem.
The Dirichlet series
\begin{equation}
  \sum_{n \in \mathbb{Z}} \frac{r_{d-1}(n^2 + h)}{(2n^2 + h)^s}
\end{equation}
is studied, and the first major result of the Chapter is in
Theorem~\ref{thm:hyperboloid:mero_summary}, which states that this Dirichlet series has
meromorphic continuation to the plane with understandable analytic behavior.

Using this meromorphic continuation, three applications are given.
Firstly, there is Theorem~\ref{thm:hyperboloid:smooth_full}, which gives a smoothed
analogue of the number of lattice points $N_{d,h}(R)$ on the one-sheeted hyperboloid
$\mathcal{H}_{d,h}$ and within $B_d(\sqrt R)$, including many second order main terms.
This Theorem can be interpreted as a long-interval smoothed average, with many secondary
main terms.

Secondly, there is Theorem~\ref{thm:hyp:concentrating_theorem_full}, which proves a
short-interval sharp average result.
But the most important result of the chapter is Theorem~\ref{thm:hyp:sharp_theorem_full},
which improves the state of the art estimate of Oh and Shah in~\eqref{eq:intro:oh_shah} by
showing that
\begin{equation}
  N_{3,h}(R) = C' R^{\frac{1}{2}} \log R + C R^{\frac{1}{2}} + O(R^{\frac{1}{2} -
  \frac{1}{44} + \epsilon})
\end{equation}
when $h$ is a positive square.
When $h$ is not a square, the $R^{\frac{1}{2}} \log R$ term does not appear.
In fact, the theorem is very general and gives results for any dimension $d \geq 3$.

In Chapter~\ref{c:hyperboloid_apps}, we describe an application of the methods and
techniques of Chapter~\ref{c:hyperboloid} to asymptotics of sums of the form
\begin{equation}
  \sum_{n \leq X} d(n^2 + h),
\end{equation}
where $d(m)$ denotes the number of positive divisors of $m$.


\clearpage{\pagestyle{empty}\cleardoublepage}

\chapter{Background}\label{c:background}

In this chapter, we review and describe some topics that will be used heavily in later
chapters of this thesis.
This chapter should be used as a reference.
The topics covered are classical and well-understood, but included for the sake of
presenting a complete idea.

\section{Notation Reference}

\index{Landau!$\ll$}
\index{Landau!$\Omega$}
We use the basic notation of Landau, so that
\begin{equation}
  F(x) \ll G(x)
\end{equation}
means that there are constants $C$ and $X$ such that for all $x > X$, we have that
\begin{equation}
  \lvert F(x) \rvert \leq C G(x).
\end{equation}
We use $F(x) \ll G(x)$ and $F(x) = O(G(x))$ interchangeably.
We also use
\begin{equation}
  F(x) = o(G(x))
\end{equation}
to mean that for any $\epsilon > 0$, there exists an $X$ such that for all $x > X$, we
have that $F(x) \leq \epsilon G(x)$.
On the other hand,
\begin{equation}
  F(x) = \Omega(G(x))
\end{equation}
means that $F$ and $G$ do not satisfy $F(x) = o(G(x))$.
Stated differently, $F(x)$ is as least as large as $G(x)$ (up to a constant) infinitely
often.

We will use $r_d(n)$ to denote the number of representations of $n$ as a sum of $d$
squares, i.e.\
\index{r@$r_d(n)$}
\begin{equation}
  r_d(n) := \# \{ \bm x \in \mathbb{Z}^d : x_1^2 + \cdots + x_d^2 = n \}.
\end{equation}
We will denote partial sums of $r_d(n)$ by $S_d$, i.e.\
\index{S@$S_d(X)$}
\begin{equation}
  S_d(X) := \sum_{n \leq X} r_d(n).
\end{equation}

\index{B@$B_d(R)$}
We use $B(R)$ to denote the disk of radius $R$ centered at the origin.
It is traditional to talk about the \emph{Gauss Circle problem} instead of the \emph{Gauss
Disk problem}, and so we will frequently refer to $B(R)$ as a circle.
We use $B_d(R)$ to denote the dimension $d$ ball of radius $R$ centered at the origin, and
for similar conventional reasons we will frequently refer to $B_d(R)$ as a sphere.
Throughout this work, we will never distinguish between points on the surface of a sphere
and those points contained within the sphere, so we use these terms interchangeably.

With respect to classical modular forms, we typically use the conventions and notations
of~\cite{Bump98, Goldfeld2006automorphic}.

Typically, $f$ will be a weight $k$ holomorphic modular form on a congruence subgroup of
$\SL(2, \mathbb{Z})$.\index{modular form (definition)}
By this, we mean the following.
A matrix $\gamma = \left(\begin{smallmatrix} a&b \\ c&d \end{smallmatrix}\right) \in
\GL(2, \mathbb{R})$ acts on $z$ in the upper half plane $\mathcal{H}$ by
\begin{equation}
  \gamma z = \frac{az + b}{cz + d}.
\end{equation}
A holomorphic modular form of (integral) weight $k$ on a congruence group
$\Gamma \subseteq \SL(2, \mathbb{Z})$ is a holomorphic function
$f: \mathcal{H} \longrightarrow \mathbb{C}$, which satisfies
\begin{equation}\label{eq:back:full_integer_law}
  f(\gamma z) = (cz+d)^k f(z) \qquad \text{for } \gamma =
  \left(\begin{smallmatrix} a&b \\ c&d \end{smallmatrix}\right) \in \Gamma,
\end{equation}
and which is holomorphic at $\infty$.
The latter condition requires some additional clarification, and an excellent overview of
this concept is in the first chapter of Diamond's introductory
text~\cite{diamondfirstcourse}.
In the context of this thesis, this latter condition guarantees that $f$ has a Fourier
expansion of the form
\begin{equation}
  f(z) = \sum_{n \geq 0} a(n) e(nz),
\end{equation}
where $e(nz) = e^{2\pi i n z}$.\index{e@$e(nz)$}
If in addition $a(0) = 0$, so that the Fourier expansion does not have a constant term,
then we call $f$ a holomorphic cusp form.

Note that some authors adopt the convention of writing the Fourier coefficients as $a(n) =
A(n) n^{\frac{k-1}{2}}$, which is normalized so that $A(n) \approx 1$ on average.
We do not use normalized coefficients in this thesis.

When it is necessary to use a second holomorphic modular form, we will denote it by $g$
and its Fourier expansion by
\begin{equation}
  g(z) = \sum_{n \geq 1} b(n) e(nz).
\end{equation}
Partial sums of these coefficients are denoted by
\index{S@$S_f(X)$}
\begin{equation}
  S_f(X) := \sum_{n \leq X} a(n), \qquad S_g(X) := \sum_{n \leq X} b(n).
\end{equation}

In Chapter~\ref{c:hyperboloid}, we will use half-integral weight modular forms
extensively.
Let $j(\gamma, z)$ be defined as
\begin{equation}
  j(\gamma, z) = \varepsilon_d^{-1} \kron{c}{d} (cz + d)^{\frac{1}{2}}, \qquad \gamma \in
  \Gamma_0(4),
\end{equation}
where $\varepsilon_d = 1$ if $d \equiv 1 \bmod 4$ and $\varepsilon_d = i$ if $d \equiv 3
\bmod 4$.
This is the same transformation law that is satisfied by the Jacobi theta function
$\theta(z) = \sum_{n \in \mathbb{Z}} e^{2\pi i n^2 z}$, which is the prototypical
half-integral weight modular form.
Then a half-integral weight modular form of weight $k$ on $\Gamma_0(4)$ is defined in the
same was as a full-integral weight modular form, except that it satisfies the
transformation law
\begin{equation}
  f(\gamma z) = j(\gamma, z)^{2k} f(z) \qquad \text{for } \gamma \in \Gamma_0(4)
\end{equation}
instead of~\eqref{eq:back:full_integer_law}.
Note that we will use $k$ to denote a full integer or a half integer in reference to
``modular forms of weight $k$.''
This is a different convention than some authors, who use $k/2$ in discussion of
half-integral weight forms.

Given two modular forms $f$ and $g$ defined on a congruence subgroup $\Gamma \subseteq
\SL(2, \mathbb{Z})$, we will let $\langle f, g \rangle$ denote the Petersson inner
product
\index{Petersson inner product}
\begin{equation}
  \langle f, g \rangle := \iint_{\Gamma \backslash \mathcal{H}} f(z) \overline{g(z)}
  d\mu(z).
\end{equation}
In this expression, $\mathcal{H}$ denotes the complex upper half plane and $d\mu(z)$
denotes the Haar measure, which in this case is given by
\begin{equation}
  d\mu(z) = \frac{dx\, dy}{y^2}.
\end{equation}

We will use $\kron{a}{n}$ to denote the Jacobi symbol.
Note that this sometimes may look very similar to a fraction.
Generally, the Jacobi symbol involves only arithmetic data, while fractions will have
complex valued arguments.

We also roughly follow some conventions concerning variable names.

The primary indices of summation will be $m$ and $n$.
If there is a shifted summation, we will usually use $h$ to denote the shift.
We will sometimes use $\ell$ to denote a distinguished index of summation (or more
generally a distinguished variable in local discussion).
The one major exception to this convention is the index $j$, which we reserve for spectral
summations (and in particular the discrete spectrum) or as an index for residues and
residual terms appearing from spectral analysis.

Complex variables will be denoted by $z,s,w$, and $u$.
Almost always $z = x + iy$ will be in the upper half plane $\mathcal{H}$, and denotes the complex
variable of the underlying modular form.
The other variables, $s,w$, and $u$ denote generic complex variables and usually appear as
the variables of various Dirichlet series.
We often follow the odd but classical convention of mixing Roman and Greek characters, and
write $s = \sigma + it$.
Note that we sometimes use $t$ to denote a generic real variable in integrals.

\index{L@$L^{(q)}(s)$}
If $L(s)$ is an $L$-function, then $L^{(q)}(s)$ denotes that $L$ -function, but with the
Euler factors corresponding to primes $p$ dividing $q$ removed.

\index{a@$\mathfrak{a}$}
We use $\mathfrak{a}$ to denote a cusp of a modular curve.
In this thesis, this almost always means that $\mathfrak{a}$ is one of the three cusps of
$\Gamma_0(4) \backslash \mathcal{H}$.

\section{On Full and Half-Integral Weight Eisenstein Series}\label{sec:Eisen_summary}
\index{Eisenstein series!weight $k$}
\index{Eisenstein series!half-integral weight}
\index{E@$E_\mathfrak{a}^k(z,w)$| \\see {Eisenstein series!weight $k$} }

The primary goal of this section is to describe some characteristics of the weight $k$
Eisenstein series for $\Gamma_0(4)$, both for full-integral weight and half-integral
weight $k$.
These details are often considered standard exercises in the literature, and are usually
tedious to compute.
We will emphasize the important properties of the Eisenstein series necessary for the
argument in Chapter~\ref{c:hyperboloid} --- some with complete proofs, and some with
descriptions of the method of proof.

\index{Eisenstein series!weight $0$}
Selberg defined the weight $0$ classical real analytic Eisenstein series $E(z,w)$ on
$\SL(2, \mathbb{Z})$ by
\begin{equation}
  E(z,w) = \sum_{\gamma \in \Gamma_\infty \backslash \SL(2, \mathbb{Z})} \Im (\gamma z)^s.
\end{equation}
The weight $0$ Eisenstein series $E(z,w)$ is very classical and very well-understood, and
both~\cite[Chapter 3]{Goldfeld2006automorphic} and~\cite[Chapter 13]{Iwaniec97} provide an
excellent description of its properties.

Weight $k$ Eisenstein series are also very classical, but they appear less often and
with much less exposition in the literature.
Half-integral weight Eisenstein series occupy an even smaller role in the literature,
although in recent years the study of metaplectic forms and metaplectic Eisenstein series
has grown in popularity.

We will use $E_\mathfrak{a}^k(z,w)$ to denote the weight $k$ Eisenstein series associated
to the cusp $\mathfrak{a}$.
In this expression, $k$ can be either a half-integer or a full-integer.
In this thesis, we use the half-integral weight Eisenstein series on the congruence
subgroup $\Gamma_0(4)$.
The quotient $\Gamma_0(4) \backslash \mathcal{H}$ has three cusps, at $\infty, 0$, and
$\frac{1}{2}$.

We shall describe the three Eisenstein series $E^k_\infty(z,w), E^k_0(z,w)$, and
$E^k_{\frac{1}{2}}(z,w)$, both for full-integral weight and half-integral weight $k$.
These Eisenstein series are defined as
\begin{align}
  E_\infty^k(z,w) &= \sum_{\gamma \in \Gamma_\infty \backslash \Gamma_0(4)} \Im(\gamma
  z)^w J(\gamma, z)^{-2k}
  =
  \sum_{\substack{\pm (c,d) \in \mathbb{Z}^2 \\ (c,d) \neq (0,0) \\ \gcd(4c,d) = 1}}
  \frac{y^w \varepsilon_d^{2k} \kron{4c}{d}^{2k}}{\lvert 4cz + d \rvert^{2w-k}(4cz+d)^k}
  \\
  E_0^k(z,w) &= \Big( \frac{z}{\lvert z \rvert} \Big)^{-k} E_\infty^k(\tfrac{-1}{4z}, w)
  =
  \sum_{\substack{\pm(c,d) \in \mathbb{Z}^2 \\ (c,d) \neq (0,0) \\ \gcd(4c, d) = 1}}
  \frac{(y/4)^w \varepsilon_d^{2k} \kron{4c}{d}^{2k}}{\lvert -c + dz \rvert^{2w - k} (-c +
  dz)^k}
  \\
  E_{\frac{1}{2}}(z,w) &= \Big( \frac{2z+1}{\lvert 2z+1 \rvert} \Big)^{-k}
  E_\infty^k(\tfrac{z}{2z+1}, w) =
  \sum_{\substack{\pm(c,d) \in \mathbb{Z}^2 \\ (c,d) \neq (0,0) \\ \gcd(2c, d) = 1}}
  \frac{y^w \varepsilon_d^{2k} \kron{2c}{d}^{2k}}{\lvert 2cz+d \rvert^{2w -
  k}(2cz+d)^k}.
\end{align}
In these expressions, $\varepsilon_d$ denotes the sign of the Gauss sum associated to the
primitive real quadratic Dirichlet character $\chi_d$, and is given by
\index{epsilon@$\varepsilon_d$}
\begin{equation}
  \varepsilon_d = \begin{cases}
    1 & d \equiv 1 \bmod 4, \\
    i & d \equiv 3 \bmod 4.
  \end{cases}
\end{equation}
We use $J(\gamma, z)$ to denote the normalized half-integral weight cocycle
\index{J@$J(\gamma, z)$}
$J(\gamma, z) = j(\gamma, z)/\lvert j(\gamma, z) \rvert$,
where $j(\gamma, z)$ is the standard $\theta$ cocycle\index{j@$j(\gamma, z)$}
\begin{equation}
  j(\gamma, z) :=
  \theta(\gamma z)/\theta(z) = \varepsilon_d^{-1} \kron{c}{d} (cz + d)^{\frac{1}{2}}.
\end{equation}
Here, $\theta(z)$ denotes the Jacobi theta function,\index{theta@$\theta(z)$}
which is a modular form of weight $\frac{1}{2}$ on $\Gamma_0(4)$.
Thus this cocycle law holds for $\gamma = \big( \begin{smallmatrix} a&b\\c&d
\end{smallmatrix} \big) \in \Gamma_0(4)$.

Notice that when $k$ is a full integer, the weighting factor simplifies to
\begin{equation}
  \varepsilon_d^{2k} \kron{4c}{d}^{2k} = \kron{-1}{d}^k,
\end{equation}
and when $k$ is a half-integer, the weighting factor simplifies slightly to
\begin{equation}
  \varepsilon_d^{2k} \kron{4c}{d}^{2k} = \varepsilon_d^{2k} \kron{4c}{d},
\end{equation}
(where in both cases the $4$ is replaced by $2$ for $E_\frac{1}{2}^k$).
This is a restatement of the fact that when $k$ is a full integer, the transformation law
is full-integral with character $\kron{-1}{\cdot}^k$, while for half-integers the
transformation law is a power of a standard normalized theta cocycle.

When $k$ is a half-integer, these three Eisenstein series are essentially the same three
Eisenstein series that appear in~\cite{goldfeld1985eisenstein}, only with some notational
differences.
Firstly, in~\cite{goldfeld1985eisenstein}, half-integer weights are denoted by
$\frac{k}{2}$, while in this work we denote all weights (both integral and half-integral)
by $k$.
Secondly, we shift the spectral argument, replacing $w$ with $w - \frac{k}{2}$ as compared
to~\cite{goldfeld1985eisenstein}.
This has the effect of using the normalized cocycle $J(\gamma,z)$ instead of
$j(\gamma,z)$, and also normalizes the arguments of the $L$-functions that appear in the
coefficients of the Eisenstein series.

\index{rho@$\rho_\mathfrak{a}^k$}
Each Eisenstein series has a Fourier-Whittaker expansion of the form
\begin{equation}
  E_\mathfrak{a}^k(\sigma_\mathfrak{b} z,w) = \delta_{[\mathfrak{a} = \mathfrak{b}]}y^w +
  \rho_{\mathfrak{a}, \mathfrak{b}}^k(0, w) y^{1-w} + \sum_{h \neq 0}
  \rho_{\mathfrak{a}, \mathfrak{b}}^k(h, w) W_{\frac{|h|}{h}\frac{k}{2}, w - \frac{1}{2}}(4\pi \lvert h
  \rvert y) e^{2\pi i h x}
\end{equation}
for some coefficients $\rho_{\mathfrak{a}, \mathfrak{b}}^k(h,w)$, and where
$\sigma_\mathfrak{b} z$ amounts to expanding at the cusp $\mathfrak{b}$.
When $\mathfrak{b} = \infty$, we will omit $\mathfrak{b}$ from the notation in the
coefficients and write instead $\rho_{\mathfrak{a}}^k(h,w)$.
Here, $W_{\alpha, \nu}(y)$ is the $\GL(2)$ Whittaker function.
This is given by~\cite[3.6.3]{GoldfeldHundleyI}
\index{Whittaker function}
\index{W@$W_{\alpha, \nu}(y)$ | see {Whittaker function } }
\begin{equation}
  W_{\alpha, \nu}(y) = \frac{y^{\nu + \frac{1}{2}}
  e^{-\frac{y}{2}}}{\Gamma(\nu - \alpha + \frac{1}{2})}
  \int_0^\infty e^{-yt} t^{\nu - \alpha - \frac{1}{2}}
  (1+t)^{\nu + \alpha - \frac{1}{2}} dt,
\end{equation}
valid for $\Re(\nu - \alpha) > \frac{-1}{2}$
and $\lvert \text{arg}(y) \rvert < \pi$.

When $\alpha = 0$, the Whittaker function is just the $K$-Bessel function
\index{K@$K$-Bessel function, $K_\nu(y)$}
\begin{equation}
  W_{0, \nu}(y) = \big( \frac{y}{\pi}\big)^{\frac{1}{2}} K_\nu\big( \frac{y}{2} \big),
  \quad
  K_\nu(y) = \frac{1}{2} \int_0^\infty e^{-\frac{1}{2} y(u + u^{-1})}
  u^\nu \frac{du}{u}.
\end{equation}
Thus in the weight $0$ case, the Whittaker functions simplify to $K$-Bessel functions (as
noted and employed in~\cite{Goldfeld2006automorphic} for instance).
This additional difficulty in weight $k$ is perhaps one reason why most expository
accounts stop at the weight $0$ case.

According to the general theory of Selberg (and described in~\cite[Theorem
13.2]{Iwaniec97}), the potential poles of $E_\mathfrak{a}^k(z,w)$ for $w > \frac{1}{2}$
can be recognized from the poles of the constant coefficient $\rho_{\mathfrak{a}}^k(0,w)$.

Although the individual expansions vary, they are usually of a very similar form.
For full integral weight, the coefficients have the shape
\index{rho@$\rho_\mathfrak{a}^k$!\\full integral weight}
\begin{align}
  \rho_\mathfrak{a}^k(0, w) &= (*) \frac{L^k_\mathfrak{a}(2w-1)}{L^k_\mathfrak{a}(2w)}
  \frac{\Gamma(2w-1)}{\Gamma(w + \frac{k}{2}) \Gamma(w - \frac{k}{2})} \\
  \rho_\mathfrak{a}^k(h, w) &= (*) \frac{\lvert h \rvert^{w-1}}{L^k_\mathfrak{a}(2w)}
  \frac{1}{\Gamma(w + \frac{\lvert h \rvert}{h} \frac{k}{2})} D_\mathfrak{a}^k(h,w)
\end{align}
where $(*)$ is a (possibly zero) constant times a collection of powers of $2$ and $\pi$,
$L^k_\mathfrak{a}(s)$ is a $\GL(1)$ $L$-function (maybe missing some Euler factors), and
$D_\mathfrak{a}^k$ is a finite Dirichlet sum.
For $\Re w > \frac{1}{2}$, the only potential pole of $E_\mathfrak{a}^k(z,w)$ is at $w = 1$.
This is shown in general in~\cite{Iwaniec97}, and the calculations are very similar to
those in~\cite[Chapter 3]{Goldfeld2006automorphic}.
For completeness and as a unifying reference, we recompute and state the exact
coefficients of $E_\infty^k$ in \S\ref{ssec:full_integral_eisenstein}.

For half-integral weight, the coefficients have similar shape,
\index{rho@$\rho_\mathfrak{a}^k$!\\half-integral weight}
\begin{equation}
  \begin{split}\label{eq:Eisenstein_halfweight_generalshape_coeffs}
    \rho_\mathfrak{a}^k(0, w) &= (*) \frac{L^k_\mathfrak{a}(4w-2)}{L^k_\mathfrak{a}(4w-1)}
    \frac{\Gamma(2w-1)}{\Gamma(w + \frac{k}{2}) \Gamma(w - \frac{k}{2})} \\
    \rho_\mathfrak{a}^k(h, w) &= (*) \lvert h \rvert^{w-1} \frac{1}{\Gamma(w +
    \frac{\lvert h \rvert}{h} \frac{k}{2})} D_\mathfrak{a}^k(h,w),
  \end{split}
\end{equation}
except that now $D_\mathfrak{a}^k(h,w)$ is a complete Dirichlet series formed from Gauss
sums, and which factors as
\begin{equation}\label{eq:Eisenstein_halfweight_Dirichletseries_factors}
  D_\mathfrak{a}^k(h,w) = \frac{L(2w - \frac{1}{2}, \chi_{k,h})}{\zeta(4w-1)}
  \widetilde{D}_\mathfrak{a}(h,w),
\end{equation}
where $\widetilde{D}_\mathfrak{a}(h,w)$ is a finite Dirichlet polynomial and $\chi_{k,h}$
\index{chi@$\chi_{k,h}$}
is the real quadratic character associated to the extension $\mathbb{Q}(\sqrt{\mu_k h})$,
and where $\mu_k = (-1)^{k - \frac{1}{2}}$.
In other words, these are almost the same as the full integral weight Eisenstein series,
except with an additional $L$-function in the numerator, and with slightly different
arguments of the involved $L$-functions.
For $\Re w > \frac{1}{2}$, the only potential pole of $E_\mathfrak{a}^k(z,w)$ is at $w
= \frac{3}{4}$.
This is all described in~\cite{goldfeld1985eisenstein}, including the factorization
in~\eqref{eq:Eisenstein_halfweight_Dirichletseries_factors}, which is given
in~\cite[Corollary 1.3]{goldfeld1985eisenstein}.
For completeness and as a unifying reference, we recompute and state the shape of the
coefficients of $E_\infty^k$ in \S\ref{ssec:half_integral_eisenstein}, including a proof
of the factorization~\eqref{eq:Eisenstein_halfweight_Dirichletseries_factors}.

The rest of this section consists of a complete description of these coefficients.
However, the general shapes of the coefficients are sufficient for the rest of this
thesis.

\subsection{Full Integral Weight}\label{ssec:full_integral_eisenstein}

\begin{claim}
  For $k \geq 1$ a full integer, the Fourier-Whittaker coefficients of $E_\infty^k(z,w)$
  are given by \begin{align}
    \rho_\infty^k(0,w) &= \begin{cases}
    0 & \text{if } k \; \text{odd}  \\
    \displaystyle y^{1-w} 2\pi 4^{1-3w} \frac{\zeta(2w-1)}{\zeta^{(2)}(2w)}
    \frac{\Gamma(2w-1)}{\Gamma(w + \frac{k}{2}) \Gamma(w - \frac{k}{2})}
    & \text{if } k \; \text{even}
    \end{cases}
    \\
    \rho_\infty^k(h,w) &= \frac{\pi^w  \lvert h \rvert^{w-1} e^{-\frac{i \pi k}{2}}}
    {L\big( 2w, \kron{-4}{\cdot}^k\big) \Gamma(w + \frac{\lvert h \rvert}{h}
    \frac{k}{2})} D_\infty^k(h,w)
  \end{align}
  where
  \begin{equation}
    D_\infty^k(h,w) = \sum_{c \mid h} \frac{c}{(4c)^{2w}}\Big( e^{\frac{\pi i h}{2c}} +
    (-1)^k e^{\frac{3\pi i h}{2c}}\Big)
\end{equation}
  is a finite Dirichlet sum.
\end{claim}

This is a classical computation.
For completeness and ease of reference, we go through this computation completely.
The coefficients for other full integral weight Eisenstein coefficients are very similar,
and are achieved through (essentially) the same work.

\begin{proof}
  In the expression for $E_\infty^k(z,w)$, when $(c,d) = (0,1)$, there is the single term $y^w$.
  We compute the rest of the $0$th Fourier-Whittaker coefficient through
  \begin{align}
    \rho_\infty^k(0, w) &= \int_0^1 E_\infty^k(z,w) - y^w dx = \sum_{\substack{c > 0, d
    \in \mathbb{Z} \\ \gcd(4c,d) = 1}} \int_0^1 \frac{y^s \kron{-4}{d}^k}{\lvert 4cz + d
    \rvert^{2w - k} (4cz + d)^k}.
  \end{align}
  We write $\kron{-4}{d}$ in place of $\kron{-1}{d}$ to facilitate our next step, and as a
  way of reinforcing that $\gcd(d,4) = 1$.  Multiplying through by $L\big(2w,
  \kron{-4}{\cdot}^k\big)$ and distributing allows us to remove the $\gcd(4c, d) = 1$
  condition from the sum, so that
  \begin{align}
    \rho_\infty^k(0,w) &= \frac{y^w}{L\big(2w, \kron{-4}{\cdot}^k\big)} \sum_{c > 0}
    \sum_d \kron{-4}{d}^k \int_0^1  \frac{1}{\lvert 4cz+d \rvert^{2w-k}(4cz + d)^k} dx \\
    &= \frac{y^w}{L\big(2w, \kron{-4}{\cdot}^k\big)} \sum_{c > 0} \frac{1}{(4c)^{2w}} \sum_d
    \kron{-4}{d}^k \int_{\frac{d}{4c}}^{1 + \frac{d}{4c}}  \frac{1}{\lvert z
    \rvert^{2w-k}z^k} dx.
  \end{align}
  We can write $d = d' + 4cq$ for each $d' \bmod{4c}$ and $q \in \mathbb{Z}$ uniquely, and
  executing the resulting $q$ sum tiles the integral to the whole real line, giving
  \begin{equation}
    \rho_\infty^k(0,w) = \frac{y^w}{L\big(2w, \kron{-4}{\cdot}^k\big)} \sum_{c > 0}
    \frac{1}{(4c)^{2w}} \sum_{d \bmod 4c} \kron{-4}{d}^k \int_{-\infty}^{\infty}
    \frac{1}{\lvert z \rvert^{2w-k}z^k} dx.
  \end{equation}
  Notice that the $d$ sum is $0$ if $k$ is odd, and is $2c$ (the number of odd integers up
  to $4c$) when $k$ is even.  Further, when $k$ is even the $L$-function in the
  denominator simplifies to a zeta function missing its $2$-factor.

  It remains necessary to compute the integral.
  By~\cite[Section 13.7]{Iwaniec97}, we have the classical integral transform
  \begin{equation}\label{eq:IntegralIdentity_0th}
    \int_{-\infty}^\infty \frac{1}{\lvert z \rvert^{2w - k} z^k} = y^{1-2w} \pi 4^{1-w}
    \frac{\Gamma(2w-1)}{\Gamma(w + \frac{k}{2}) \Gamma(w - \frac{k}{2})}.
  \end{equation}
  Therefore, we have for $k$ even that
  \begin{equation}
    \rho_\infty^k(0,w) = y^{1-w} 2\pi 4^{1-3w} \frac{\zeta(2w-1)}{\zeta^{(2)}(2w)}
    \frac{\Gamma(2w-1)}{\Gamma(w + \frac{k}{2}) \Gamma(w - \frac{k}{2})}
  \end{equation}

  For the $h$th Fourier coefficient, the method begins similarly.
  Following the same initial steps, we compute
  \begin{align}
    \rho_\infty^k(h,w) &= \int_0^1 E_\infty^k(z,w) e^{-2\pi i h x} dx \\
    &=
    \frac{y^w}{L(2w, \kron{-4}{\cdot}^k)} \sum_{c > 0} \frac{1}{(4c)^{2w}} \sum_d
    \kron{-4}{d}^k e^{2\pi i \frac{hd}{4c}} \int_{\frac{d}{4c}}^{1 + \frac{d}{4c}}
    \frac{e^{-2\pi i h x}}{\lvert z \rvert^{2w - k} z^k} dx.
  \end{align}
  Writing $d = d' + 4cq$ as above again tiles out the integral, so that
  \begin{equation}
    \rho_\infty^k(h,w) = \frac{y^w}{L(2w, \kron{-4}{\cdot}^k)} \sum_{c > 0}
    \frac{1}{(4c)^{2w}} \sum_{d\bmod 4c} \kron{-4}{d}^k e^{2\pi i \frac{hd}{4c}}
    \int_{-\infty}^{\infty} \frac{e^{-2\pi i h x}}{\lvert z \rvert^{2w - k} z^k} dx.
  \end{equation}

  For $h \neq 0$, the integral can be evaluated as in~\cite[\S13.7]{Iwaniec97} or
  using~\cite[3.385.9]{GradshteynRyzhik07} to be
  \begin{equation}\label{eq:IntegralIdentity_hth}
    \int_{-\infty}^{\infty} \frac{e^{-2\pi i h x}}{\lvert z \rvert^{2w - k} z^k} dx =
    \frac{y^{-w} i^{-k} \pi^w \lvert h \rvert^{w-1}}{\Gamma(w + \frac{\lvert h \rvert}{h}
    \frac{k}{2})} W_{\frac{\lvert h \rvert}{h}\frac{k}{2}, \frac{1}{2} - w}(4\pi \lvert h
  \rvert y).
  \end{equation}
  To understand the arithmetic part,
  \begin{equation}
    \sum_{c > 0} \frac{1}{(4c)^{2w}} \sum_{d \bmod 4c} \kron{-4}{d}^k e^{\pi i \frac{hd}{2c}},
  \end{equation}
  note that each $d$ can be written uniquely in the form $d = d' + 4q$ for $0 \leq d' < 4$
  and $0 \leq q < c$.
  Then the arithmetic part is written
  \begin{equation}
    \sum_{c > 0} \frac{1}{(4c)^{2w}} \sum_{d=0}^3 \kron{-4}{d}^k e^{\pi i
  \frac{hd}{2c}} \sum_{q \bmod c} e^{2\pi i \frac{hq}{c}}.  \end{equation}
  The final sum over $q$ is
  \begin{equation}
    \sum_{q \bmod c} e^{2\pi i \frac{hq}{c}} = \begin{cases}
      0 & \text{if } c \nmid h \\
      c & \text{if } c \mid h.
    \end{cases}
  \end{equation}
  Therefore the arithmetic part simplifies down to
  \begin{equation}
    \sum_{c \mid h} \frac{c}{(4c)^{2w}} \big( e^{\pi i \frac{h}{2c}} + (-1)^k
  e^{\frac{3\pi i h}{2c}} \big).  \end{equation}
  Simplification completes the proof.
\end{proof}

\subsection{Half Integral Weight}\label{ssec:half_integral_eisenstein}

The pattern here is almost the exact same as with full-integral weight.
We will prove the expansion for $E_\infty^k(z,w)$ completely.
The primary difference is that there is now a Dirichlet series of Gauss sums, which can be
factored as a ratio of a Dirichlet $L$-function and a zeta function, up to a short
correction factor.

\begin{claim}
  For $k \geq \frac{1}{2}$ a half integer, the Fourier-Whittaker coefficients of $E_\infty^k(z,w)$ is given by
  \begin{align}
    \rho_\infty^k(0,w) &= \bigg( \frac{1+i^{2k}}{2} \bigg) \frac{\pi 4^{1-w}}{2^{4w - 1} - 1}
    \frac{\zeta(4w - 2)}{\zeta(4w - 1)}
    \frac{\Gamma(2w - 1)}{\Gamma(w + \frac{k}{2}) \Gamma(w - \frac{k}{2})}
    \\
    \rho_\infty^k(h,w) &= \frac{e^{\frac{-i\pi k}{2}} \pi^w \lvert h \rvert^{w-1}}
    {\Gamma(w + \frac{\lvert h \rvert}{h} \frac{k}{2})}
    \sum_{c > 0} \frac{g_h(4c)}{(4c^{2w})},
  \end{align}
  where\index{g@$g_h(c)$}
  \begin{equation}
    g_h(c) = \sum_{d \bmod c} \varepsilon_d^{2k} \kron{c}{d} e \big( \tfrac{hd}{c} \big)
  \end{equation}
  is a Gauss sum.
  The Dirichlet series associated to these Gauss sums can be written
  as\index{D@$\widetilde{D}_\infty^k(h,w)$}
  \begin{equation}
    \sum_{c \geq 1} \frac{g_h(4c)}{(4c)^{2w}}
    =
    \frac{L^{(2)} (2w - \frac{1}{2}, \chi_{k,h})}{\zeta^{(2h)}(4w-1)}
    \widetilde{D}_\infty^k(h,w),
  \end{equation}
  as proved in Proposition~\ref{prop:back:half_integral_L}.
\end{claim}

\begin{proof}
  As in the full-integral weight case, we compute the constant Fourier coefficient through
  \begin{equation}
    \rho_\infty^k(0,w) = \int_0^1 E_\infty^k(z,w) - y^w dx =
    \sum_{\substack{c > 0, d \in \mathbb{Z} \\ \gcd(4c, d) = 1}}
    \int_0^1 \frac{y^w \varepsilon_d^{2k} \kron{4c}{d}}
    {\lvert 4cz + d \rvert^{2w - k} (4cz + d)^k}.
  \end{equation}
  The only difference in comparison to the full-integral weight case is that the numerator
  has $\varepsilon_d^{2k} \kron{4c}{d}$.
  Notice that the character enforces $\gcd(4c, d) = 1$, so that this condition can be
  dropped from the summation indices.

  Factoring $(4c)^{-2w}$, writing $d = d' + 4cq$ for each $d' \bmod 4c$ and a unique
  $q \in \mathbb{Z}$, performing the change of variables $x \mapsto x - \frac{d}{4c}$, and
  tiling the integral functions exactly as in the full-integral proof, and leads to
  \begin{equation}
    y^w \sum_{c > 0} \frac{1}{(4c)^{2w}} \sum_{d \bmod 4c} \varepsilon_d^{2k} \kron{4c}{d}
    \int_{-\infty}^\infty \frac{1}{\lvert z \rvert^{2w-k} z^k} dx.
  \end{equation}
  Using the evaluation of this integral from~\eqref{eq:IntegralIdentity_0th}, we see that
  the $0$th coefficient can be written as
  \begin{equation}
    y^{1-w} \sum_{c > 0} \frac{1}{(4c)^{2w}} \sum_{d \bmod 4c} \varepsilon_d^{2k}
    \kron{4c}{d} \frac{\pi 4^{1-w} \Gamma(2s - 1)}{\Gamma(w + \frac{k}{2})
    \Gamma(w - \frac{k}{k})}.
  \end{equation}

  It is now necessary to understand the arithmetic part, given by
  \begin{equation}\label{eq:back:E_halfk_0_I}
    \sum_{c > 0} \frac{1}{(4c)^{2w}} \sum_{d \bmod 4c} \varepsilon_d^{2k} \kron{4c}{d}.
  \end{equation}
  Noting that we can rewrite $\varepsilon_d$ as
  \begin{equation}
    \varepsilon_d = \frac{1}{2}\big( \chi_{-1}^2(d) + \chi_{-1}(d) \big) +
    \frac{i}{2}\big( \chi_{-1}^2(d) - \chi_{-1}(d) \big),
  \end{equation}
  where $\chi_{-1}(d) = \kron{-1}{d}$ is the Legendre character associated to when $-1$ is
  a square, we see that the $d$ summation is trivial unless the character $\kron{4c}{d}$
  is trivial.
  (This follows from classical observation that summing a non-trivial character over a
  whole group gives zero).
  Thus the sum is nontrivial if and only if $c$ is a square, and when $c$ is a square we
  have that
  \begin{equation}
    \sum_{d \bmod 4c^2} \varepsilon_d^{2k}\kron{4c^2}{d}
    = \sum_{\substack{d \bmod 4c^2 \\ \gcd(d,4c) = 1}} \varepsilon_d^{2k}
    = \frac{1 + i^{2k}}{2} \varphi(4c^2).
  \end{equation}
  Therefore~\eqref{eq:back:E_halfk_0_I} can be written as
  \begin{equation}
     \frac{1+i^{2k}}{2} \sum_{c > 0} \frac{1}{(4c^2)^{2w}} \varphi(4c^2)
     = \frac{1+i^{2k}}{2} \frac{1}{2^{4w - 1} - 1} \frac{\zeta(4w - 2)}{\zeta(4w - 1)}.
  \end{equation}
  The last equality follows from comparing Euler products, correcting the missing
  $2$-factor, and simplifying.
  Combining the arithmetic part with the analytic part gives the $\rho_\infty^k(0,w)$.

  Computing the $h$th Fourier coefficient is very similar.
  After following the same initial steps to tile out the integral, we get the equality
  \begin{equation}
    \int_0^1 E_\infty^k(z,w) e^{-2\pi i h x} dx = \sum_{c > 0} \frac{1}{(4c)^{2w}}
    \sum_{d \bmod 4c} \varepsilon_d^{2k} \kron{4c}{d} e\big(\tfrac{hd}{4c}\big)
    \int_{-\infty}^\infty \frac{y^w e^{-2\pi i h x}}{\lvert z \rvert^{2w-k} z^k} dx.
  \end{equation}
  This integral was evaluated in~\eqref{eq:IntegralIdentity_hth}, so we see that this
  becomes
  \begin{equation}
     \frac{e^{-i\pi k/2} \pi^w \lvert h \rvert^{w-1}}
    {\Gamma(w + \frac{\lvert h \rvert}{h}\frac{k}{2})}
    W_{\frac{\lvert h \rvert}{h}\frac{k}{2}, \frac{1}{2} - w} (4 \pi \lvert h \rvert y)
    \sum_{c > 0} \frac{1}{(4c)^{2w}}
    \sum_{d \bmod 4c} \varepsilon_d^{2k} \kron{4c}{d} e\big(\tfrac{hd}{4c}\big).
  \end{equation}
  The arithmetic part is now a Dirichlet series of Gauss sums, and is not finite.
  Simplifiying completes the computation.
\end{proof}

\subsection*{On the $L$-series associated to half-integral weight Eisenstein series
coefficients}

We now consider the Dirichlet series
\begin{equation}
  \sum_{c \geq 1} \frac{g_h(4c)}{(4c)^{2w}}
\end{equation}
appearing in the $h$th Fourier coefficients of $E_\infty^k(z,w)$ when $k$ is a half
integer.
The goal of this section is to show the following proposition.

\begin{proposition}\label{prop:back:half_integral_L}
  \begin{equation}
    \sum_{c \geq 1} \frac{g_h(4c)}{(4c)^{2w}}
    =
    \frac{L^{(2)} (2w - \frac{1}{2}, \chi_{k,h})}{\zeta^{(2h)}(4w-1)}
    \widetilde{D}_\infty^k(h,w),
  \end{equation}
  where $\widetilde{D}_\infty^k(h,w)$ is a finite Dirichlet polynomial.
\end{proposition}

We prove this proposition through a series of lemmata.\footnote{Yes, that is a real word.
And it's a fantastic word.}
First, we first consider just an individual Gauss sum
$g_h(4c) = \sum_{d \bmod 4c} \varepsilon_d^{2k} e(\frac{hd}{4c}) \kron{4c}{d}$.
It is necessary to break each Gauss sum into two pieces.

\begin{lemma}\label{lem:back:gauss_two_pieces}
  \begin{equation}
    g_h(4c) = \sum_{d_2\bmod 2^\alpha} \varepsilon_{d_2c'}^{2k} (-1)^{\frac{c'-1}{2}
    \frac{d_2c'-1}{2}}\varepsilon_{c'}^{-1} \kron{2^\alpha}{d_2}
    \varepsilon_{c'}\sum_{d_1\bmod c'} \kron{d_1}{c'}  e\big( \tfrac{h d_1}{c'} \big).
  \end{equation}
\end{lemma}

\begin{proof}

Write $4c = 2^\alpha c'$ where $\gcd(c', 2) = 1$.
Note that we necessarily have $\alpha \geq 2$.
For any $d \bmod 4c$, we write $d = d_1 2^\alpha + d_2c'$ with $d_1$ varying mod $c'$ and
$d_2$ varying mod $2^\alpha$.
Using this, we can write
\begin{align}
    \sum_{d \bmod 4c} \varepsilon_d^{2k} e\big(\tfrac{hd}{4c}\big) \kron{4c}{d}
    &=
    \sum_{(d_1\bmod c')} \sum_{(d_2\bmod 2^\alpha)} \varepsilon_{d_2c'}^{2k}
    e\big( \tfrac{h (d_1 2^\alpha + d_2 c')}{2^\alpha c'} \big)
    \kron{2^\alpha c'}{d_1 2^\alpha + d_2c'}
    \\
    &=
    \sum_{d_2 \bmod 2^\alpha} \varepsilon_{d_2c'}^{2k}
    e\big(\tfrac{h d_2}{2^\alpha}\big)
    \kron{2^\alpha}{d_2c'} \sum_{d_1 \bmod c'}
  \kron{c'}{d_1 2^\alpha+ d_2c'} e\big( \tfrac{h d_1}{c'} \big).
\end{align}

Note that $d_2$ is necessarily odd.
By quadratic reciprocity, we can rewrite the inner sum as
\begin{align}
    \sum_{d_1 \bmod c'} \kron{c'}{d_1 2^\alpha+ d_2c'} e\big( \tfrac{h d_1}{c'} \big)
    &=
    \sum_{d_1 \bmod c'} \kron{d_1 2^\alpha + d_2 c'}{c'}
    (-1)^{\frac{c'-1}{2} \frac{d_2c'-1}{2}} e \big( \tfrac{h d_1}{c'} \big)
    \\
    &=
    \sum_{d_1 \bmod c'} \kron{d_1 2^\alpha}{c'} (-1)^{\frac{c'-1}{2}
    \frac{d_2c'-1}{2}} e\big( \tfrac{h d_1}{c'}\big).
\end{align}
Inserting this back into $g_h(4c)$ gives
\begin{align}
  g_h(4c)
  &=
  \sum_{d_2\bmod 2^\alpha} \varepsilon_{d_2c'}^{2k} (-1)^{\frac{c'-1}{2}
  \frac{d_2c'-1}{2}} e\big( \tfrac{h d_2}{2^\alpha} \big) \kron{2^\alpha}{d_2}
  \kron{2^\alpha}{c'}\sum_{d_1\bmod c'} \kron{2^\alpha}{c'}\kron{d_1}{c'}
  e\big( \tfrac{h d_1}{c'} \big) \\
  &=
  \sum_{d_2\bmod 2^\alpha} \varepsilon_{d_2c'}^{2k} (-1)^{\frac{c'-1}{2}
  \frac{d_2c'-1}{2}} e\big( \tfrac{h d_2}{2^\alpha} \big)
  \varepsilon_{c'}^{-1} \kron{2^\alpha}{d_2}
  \varepsilon_{c'}\sum_{d_1\bmod c'} \kron{d_1}{c'}  e\big( \tfrac{h d_1}{c'} \big).
\end{align}
Note that we have introduced $\varepsilon_c \varepsilon_c^{-1}$.
This ends the proof of the lemma.
\end{proof}

The $d_1$ sum is now completely decoupled from the $d_2$ sum, and they can be understood
separately.
We first understand the $d_1$ summation.

\begin{lemma}\label{lem:back:Hh_multiplicative}
  For odd $c$, define
  \begin{equation}
    H_h(c) := \varepsilon_c \sum_{d \bmod c} \kron{d}{c} e \big( \tfrac{h d}{c} \big),
  \end{equation}
  as occurs in the $d_1$ summation in $g_h(4c)$.
  Then $H_h(c)$ is multiplicative.
\end{lemma}

\begin{proof}

Let $n_1$ and $n_2$ be two odd relative prime integers.
By the Chinese Remainder Theorem, and congruence class $d \bmod n_1 n_2$ can be uniquely
written as $d = b_2 n_1 + b_1 n_2$ where $b_2$ is defined modulo $n_2$ and $b_1$ is
defined modulo $n_1$.
Then
\begin{align}
  H_h(n_1 n_2)
  &=
  \varepsilon_{n_1 n_2} \sum_{d \bmod n_1n_2} \kron{d}{n_1n_2}
  e\big(\tfrac{d h}{n_1 n_2} \big)
  \\
  &= \varepsilon_{n_1 n_2} \sum_{b_1 \bmod n_1} \sum_{b_2 \bmod n_2}
  \kron{b_2 n_1 +b_1 n_2}{n_1 n_2}  e \big( \tfrac{h b_2 n_1 +b_1 n_2}{n_1 n_2} \big)
  \\
  &= \varepsilon_{n_1 n_2}\kron{n_2}{n_1} \kron{n_1}{n_2}
  \sum_{b_1 \bmod n_1} \kron{b_1}{n_1} e\big( \tfrac{h b_1}{n_1} \big)
  \sum_{b_2 \bmod n_2} \kron{b_2}{n_2} e\big( \tfrac{h b_2}{n_2} \big)
  \\
  &= \varepsilon_{n_1 n_2} \varepsilon_{n_1}^{-1} \varepsilon_{n_2}^{-1}
  \kron{n_2}{n_1} \kron{n_1}{n_2} H_h(n_1) H_{h}(n_2).
\end{align}
Casework and quadratic reciprocity shows that
$\varepsilon_{n_1n_2} \varepsilon_{n_1^{-1}} \varepsilon_{n_2}^{-1}
\kron{n_2}{n_1}\kron{n_1}{n_2}= 1$,
so that $H_h(n_1 n_2) = H_h(n_1) H_h(n_2)$.
\end{proof}

In order to understand a Dirichlet series formed from $H_h(c')$, it will be sufficient to
understand $H_h(p^k)$ for odd primes $p$.
When $p \nmid h$, these are particularly easy to understand.

\begin{lemma}\label{lem:back:Hh_goodp_eval}
  Suppose $p$ is an odd prime and $p \nmid h$.
  Then
  \begin{equation}
    H_h(p^k) = \begin{cases}
      \kron{-h}{p} \sqrt p & k = 1 \\
      0 & k \geq 2.
    \end{cases}
  \end{equation}
\end{lemma}

\begin{proof}
  When $k = 1$, we have
  \begin{equation}
    H_h(p) = \varepsilon_p \sum_{d \bmod p} \kron{d}{p} e \big( \tfrac{dh}{p} \big) =
    \varepsilon_p^2 \kron{h}{p} \sqrt p = \kron{-h}{p} \sqrt p,
  \end{equation}
  as $H_h(p)$ is very nearly a standard Gauss Sum, as considered in the beginning
  of~\cite{Davenport1980}.

  For $k \geq 2$, there are two cases.
  If $k$ is even, then the sum is exactly a sum over the primitive $p^k$-roots of unity,
  and therefore is zero.
  If $k$ is odd, then writing $\eta = e^{2\pi i / p^k}$ as a primitive $p^k$-root of
  unity, we have
  \begin{align}
    H_h(p^k) &= \sum_{d\bmod p^k} \eta^d \kron d {p^k} = \sum_{d\bmod p^k} \eta^d \kron{d}{p} =
    \sum_{\substack{b \bmod p \\ c \bmod p^{k-1}}} \eta^{b+pc} \kron b p \\
    &= \sum_{b\bmod p} \eta^b
    \kron b p \sum_{c\bmod p^{k-1}} \eta^{pc} = 0,
  \end{align}
  as the inner sum is a sum over the primitive $p^{k-1}$ roots of unity.
\end{proof}

The case is substantially more complicated for primes $p$ dividing $h$.
For this thesis, we do not need to calculate these explicitly.
(Although we could, using these techniques).
It is sufficient to know that only finitely many contribute.

\begin{lemma}\label{lem:back:Hh_badp_eval}
  Suppose $p$ is an odd prime and $p \mid h$.
  Further, suppose $p^\ell \mid h$ but $p^{\ell + 1} \nmid h$.
  Then for $k \geq \ell + 2$, we have $H_h(p^k) = 0$.
\end{lemma}

\begin{proof}

  This is substantially similar to the case when $p \nmid h$, and the same proof carries
  over.
  When $k$ is even, $H_h(p^k)$ is a sum over primitive $p^{k-\ell}$ roots of unity.
  When $k$ is odd, $H_h(p^k)$ has an inner exponential sum over the $p^{k - \ell - 1}$
  roots of unity.
\end{proof}

\begin{remark}
  Although we do not compute it here, it is possible to compute the exact contribution
  from each factor $p$ dividing $h$.
  One complete reference can be found at the author's website~\cite{mixedmathGauss}.
  Some subcases are included in~\cite{goldfeld1985eisenstein}.
\end{remark}

We now seek to understand the $d_2$ summation in Lemma~\ref{lem:back:gauss_two_pieces}.

\begin{lemma}\label{lem:back:d2_sum_simplifies}
  We have
  \begin{equation}\label{eq:back:d2_sum_I}
    \varepsilon^{2k}_{d_2 c'} (-1)^{\frac{c'-1}{2} \frac{d_2 c' - 1}{2}}
    \varepsilon_{c'}^{-1}
    =
    \varepsilon_{d_2}^{2k} \kron{(-1)^{k + \frac{1}{2}}}{c'}.
  \end{equation}
  In particular,
  \begin{equation}
    \sum_{d_2 \bmod 2^\alpha}
    \varepsilon^{2k}_{d_2 c'} (-1)^{\frac{c'-1}{2} \frac{d_2 c' - 1}{2}}
    e\big( \tfrac{h d_2}{2^\alpha} \big)
    \varepsilon_{c'}^{-1} \kron{2^\alpha}{d_2}
    =
    \sum_{d_2 \bmod 2^\alpha}
    \varepsilon_{d_2}^{2k} \kron{(-1)^{k + \frac{1}{2}}}{c'} \kron{2^\alpha}{d_2}
    e\big( \tfrac{h d_2}{2^\alpha} \big).
  \end{equation}
\end{lemma}

\begin{proof}

  First note that as $\varepsilon_d^4 = 1$, we can reduce the analysis into two cases: when
  $2k \equiv 1 \bmod 4$ and when $2k \equiv 3 \bmod 4$.
  After this reduction, the equality in~\eqref{eq:back:d2_sum_I} is quickly verified by
  considering the possible values of $d_2$ and $c'$ modulo $4$ (recalling that both $d_2$
  and $c'$ are odd).
  The rest of the lemma follows immediately.
\end{proof}

In this simplified $d_2$ summation, there appears $\varepsilon_{d_2}^{2k}$, which acts a bit
like a character in $d_2$ modulo $4$, and $\kron{2^\alpha}{d_2}$, which acts a bit like a
character modulo $8$, and an unrestrained exponential.
Therefore we should expect that when $2^\alpha > 8$, or rather when $\alpha \geq 4$, then
the entire $d_2$ sum vanishes unless $h$ is highly divisible by $2$.

\begin{lemma}\label{lem:back:d2_sum_is_finite}
  Suppose $2^\ell \mid h$ and $2^{\ell + 1} \nmid h$.
  If $\alpha \geq \ell + 4$, then the $d_2$ summation vanishes, i.e.\
  \begin{equation}
    \sum_{d_2 \bmod 2^\alpha}
    \varepsilon_{d_2}^{2k} \kron{2^\alpha}{d_2} e\big( \tfrac{h d_2}{2^\alpha} \big) = 0.
  \end{equation}
\end{lemma}

\begin{proof}
  Without loss of generality, choose least non-negative representations for each $d_2 \bmod
  2^\alpha$.
  Write $d_2 = 8d' + d''$ where $0 \leq d' < 2^{\alpha - 3}$ and $0 \leq d'' < 8$.
  Then $\varepsilon_{8d' + d''} = \varepsilon_{d''}$ and
  $\kron{2^\alpha}{8d' + d''} = \kron{2^\alpha}{d''}$, so the $d'$ summation can be
  considered separately.
  This summation is
  \begin{equation}
    \sum_{d' = 0}^{2^{\alpha-3} - 1} e \bigg( \frac{h d'}{2^{\alpha - 3}} \bigg),
  \end{equation}
  which is $0$ unless $2^{\alpha - 3} \mid h$.
\end{proof}

As $c = 2^\alpha c'$, we see that the $d_2$ summation constrains the contribution from the
$2$-factor of $c$.
We are now finally ready to prove Proposition~\ref{prop:back:half_integral_L}.

\begin{proof}[Proof of Prop~\ref{prop:back:half_integral_L}]

  We try to understand
  \begin{align}
    &\sum_{c > 0} \frac{1}{(4c)^{2w}} \sum_{d \bmod 4c} \varepsilon_d^{2k}
    e \big( \tfrac{hd}{4c} \big) \kron{4c}{d}
    \\
    &= \sum_{\alpha \geq 2} \sum_{\substack{c \geq 1 \\ \gcd(2,c) = 1}}
    \frac{1}{(2^\alpha c)^{2w}} \kron{(-1)^{k + \frac{1}{2}}}{c} H_h(c)
    \sum_{d_2 \bmod 2^\alpha} \varepsilon_{d_2}^{2k}
    e \big( \tfrac{d_2 h}{2^\alpha} \big) \kron{2^\alpha}{d_2}
    \\
    &= \Bigg( \sum_{\alpha \geq 2} \sum_{d_2 \bmod 2^\alpha}
      \frac{1}{2^{2\alpha w}} \varepsilon_{d_2}^{2k} e \big( \tfrac{d_2 h}{2^\alpha} \big)
      \kron{2^\alpha}{d_2}
    \Bigg)
    \Bigg( \sum_{\substack{c \geq 1 \\ \gcd(c,2) = 1}}
      \frac{1}{c^{2w}}
      \Big(\frac{(-1)^{k + \frac{1}{2}}}{c}\Big) H_h(c)
    \Bigg). \label{eq:back:Lprop_proof_I}
  \end{align}
  We have used Lemma~\ref{lem:back:gauss_two_pieces} and
  Lemma~\ref{lem:back:d2_sum_simplifies} to split and simplify this expression.
  Let $\chi_k(c) := \kron{(-1)^{k + \frac{1}{2}}}{c}$ for ease of notation.
  As the summands over $c$ are multiplicative (by Lemma~\ref{lem:back:Hh_multiplicative}),
  the $c$ sum can be written (for $\Re w \gg 1$) as
  \begin{equation}
    \sum_{\substack{c \geq 1 \\ \gcd(c,2) = 1}}
    \frac{1}{c^{2w}} \chi_k(c) H_h(c)
    =
    \prod_{\substack{p \\ p \neq 2}} \Big( 1 + \frac{\chi_k(p) H_h(p)}{p^{2w}} +
    \frac{\chi_k(p^2) H_h(p^2)}{p^{4w}} + \cdots  \Big).
  \end{equation}

  For primes not dividing $h$, this expression simplifies significantly as $H_h(p^k) = 0$
  for $k \geq 2$ (as shown in Lemma~\ref{lem:back:Hh_goodp_eval}).
  The product over these primes is then
  \begin{equation}
    \prod_{\substack{p \\ p \neq 2 \\ p \nmid h}} \Big( 1 + \frac{\big(\frac{h (-1)^{k -
    \frac{1}{2}}}{p}\big)} {p^{2w - \frac{1}{2}}} \Big).
  \end{equation}
  On primes avoiding $h$ and $2$, it's quickly checked that this perfectly matches the
  Euler product for $L(2w-\frac{1}{2}, \chi_{k,h}) \zeta(4w-1)^{-1}$, where
  $\chi_{k,h}(\cdot) = \big( \frac{h (-1)^{k - \frac{1}{2}}}{\cdot} \big)$.
  Therefore we can write the product over those primes avoiding $h$ and $2$ as
  \begin{equation}
    \prod_{\substack{p \\ p \neq 2 \\ p \nmid h}} \Big( 1 + \frac{\big(\frac{h (-1)^{k -
    \frac{1}{2}}}{p}\big)} {p^{2w - \frac{1}{2}}} \Big)
    =
    \frac{L^{(2)} (2w - \frac{1}{2}, \chi_{k,h})}{\zeta^{(2h)}(4w-1)}.
  \end{equation}
  (Note that we are using the convention that $L^{(Q)}(s)$ denotes an $L$-function $L(s)$,
  but with the Euler factors for primes $p$ dividing $Q$ removed).

  We now consider the primes $p$ which do divide $h$ in the Euler product
  in~\eqref{eq:back:Lprop_proof_I}.
  By Lemma~\ref{lem:back:Hh_badp_eval}, we see that $H_h(p^k) = 0$ when $p^{k-1} \nmid h$.
  Therefore the product over primes dividing $h$ is a product of finitely many terms of
  finite length, and is thus just a Dirichlet polynomial.

  Similarly, by Lemma~\ref{lem:back:d2_sum_is_finite}, the sum over $d_2$ and $\alpha$
  in~\eqref{eq:back:Lprop_proof_I} is a finite sum whose length depends on the $2$-factor
  of $h$, and is also a Dirichlet polynomial.

  We group the product over primes dividing $h$ with the $d_2$ and $\alpha$ summation
  in~\eqref{eq:back:Lprop_proof_I}, which are both finite Dirichlet polynomials, into a
  single Dirichlet polynomial
  \begin{equation}
    \begin{split}
      \widetilde{D}_\infty^k (h,w) :=
      \sum_{\alpha \geq 2} \sum_{d_2 \bmod 2^\alpha}
        \frac{1}{2^{2\alpha w}} \varepsilon_{d_2}^{2k} e \big( \tfrac{d_2 h}{2^\alpha} \big)
        \kron{2^\alpha}{d_2}
      \times
      \prod_{\substack{p \\ p \mid h \\ p \neq 2}}
        \sum_{j \geq 0} \frac{\chi_k(p^j) H_h(p^j)}{p^{2jw}}.
    \end{split}
  \end{equation}
  Note carefully that although this is written as an infinite polynomial, it is a finite
  Dirichlet polynomial.

  Collecting these pieces together we have now shown that
  \begin{equation}
    \sum_{c > 0} \frac{1}{(4c)^{2w}} \sum_{d \bmod 4c} \varepsilon_d^{2k} e \big( \tfrac{hd}{4c} \big) \kron{4c}{d}
    =
    \frac{L^{(2)} (2w - \frac{1}{2}, \chi_{k,h})}{\zeta^{(2h)}(4w-1)}
    \widetilde{D}_\infty^k(h,w),
  \end{equation}
  as we set out to show.
\end{proof}

\begin{remark}
  If $h$ is squarefree, then it is possible to show that $\widetilde{D}_\infty^k(h,w)$ has
  the necessary Euler factors to ``fill in'' the $h$ factors of $\zeta^{(2h)}(4w - q)$
  (although not the $2$ factor), and the expression for $\widetilde{D}_\infty^k(h,w)$
  simplifies significantly.
  Thus the case when $h$ is squarefree is significantly simpler.
\end{remark}

\section{Cutoff Integrals and Their Properties}\label{sec:cutoff_integrals}

We recall the Mellin transform and inverse Mellin transform, and use these to construct
appropriate integral transforms to analyze properties of the coefficients of a Dirichlet
series.
In general, if
\index{Mellin transform}
\begin{equation}
  F(s) = \int_0^\infty f(x) x^s \frac{dx}{x},
\end{equation}
then $F(s)$ is the \emph{Mellin transorm} of $f(x)$.
Mellin transforms are deeply related to Laplace transforms and Fourier transforms, and
when $f$ and $F$ are sufficiently nice, there is an analogous inversion theorem giving
that
\begin{equation}
  f(x) = \frac{1}{2\pi i} \int_{(\sigma)} F(s) x^{-s} ds.
\end{equation}

\subsubsection*{Ces\`{a}ro Cutoff Transform}
\index{Integral transform!Ces\'{a}ro Weights}
\index{Integral transform!Riesz Means}

In this thesis, we reintroduce and use Ces\`{a}ro weights.
Note that these are sometimes referred to as ``Riesz Means.''
Given a positive integer $k$ and a Dirichlet series
\begin{equation}
  D(s) = \sum_{n \geq 1} \frac{a(n)}{n^s},
\end{equation}
we have the fundamental relationship
\begin{equation}\label{eq:cesaro_transform}
  \begin{split}
    \frac{1}{k!} \sum_{n \leq X} a(n) \Big(1 - \frac{n}{X}\Big)^k
    &= \frac{1}{2\pi i} \int_{(\sigma)} D(s) \frac{X^s}{s(s+1)\cdots(s+k)} ds
    \\
    &= \frac{1}{2\pi i} \int_{(\sigma)} D(s) \frac{X^s\Gamma(s)}{\Gamma(s+k+1)} ds,
  \end{split}
\end{equation}
where $\sigma$ is large enough that $D(s)$ and the integral absolutely converge.
The individual weights $(1 - \frac{n}{X})^k$ on each $a(n)$ are the $k$-Ces\`{a}ro
weights, and give access to smoothed asymptotics.

The relationship~\eqref{eq:cesaro_transform} follows from the following integral
equality.
\begin{lemma}
  For $\sigma > 0$, we have
  \begin{equation}
    \frac{1}{2\pi i} \int_{(\sigma)} \frac{Y^s}{s(s+1)\cdots(s+k)} ds = \begin{cases}
      \frac{1}{k!} \big(1 - \frac{1}{Y}\big)^k & \text{if } Y \geq 1, \\
      0 & \text{if }Y < 1.
    \end{cases}
  \end{equation}
\end{lemma}
\begin{proof}
  In the case when $Y \geq 1$, shifting the contour infinitely far to the left shows that
  the integral can be evaluated as
  \begin{equation}
    \sum_{j = 0}^k \Res_{s = -j} \bigg( \frac{Y^s}{s(s+1)\cdots(s+k)} \bigg) = \sum_{j =
    0}^k \frac{(-1)^j Y^{-j}}{j! (k-j)!} = \frac{1}{k!} \Big(1 - \frac{1}{Y}\Big)^k.
  \end{equation}
  Note that the last equality is an application of the binomial theorem.

  In the case when $Y < 1$, shifting the contour infinitely far to the right shows that
  the integral is $0$.
\end{proof}

To recover~\eqref{eq:cesaro_transform}, one expands $D(s)$ within the integral and applies
the lemma to each individual term with $Y = (X/n)$.

\subsubsection*{Exponentially Smoothed Integral}
\index{Integral transform!exponential smoothing}

The integral definition of the Gamma function
\begin{equation}
  \Gamma(s) = \int_0^\infty t^s e^{-t} \frac{dt}{t}
\end{equation}
is a Mellin integral, and gives the inverse Mellin integral
\begin{equation}
  e^{-x} = \frac{1}{2 \pi i} \int_{(\sigma)} x^{-s} \Gamma(s) ds.
\end{equation}
Applied to the Dirichlet series $D(s)$, we have
\begin{equation}
  \sum_{n \geq 1} a(n) e^{-n/X} = \frac{1}{2\pi i} \int_{(\sigma)} D(s) X^s \Gamma(s) ds.
\end{equation}

\subsubsection*{Concentrating Integral}
\index{Integral transform!concentrating}

We now produce an integral transform that has the effect of concentrating the mass of the
integral around the parameter $X$.
We claim that
\begin{equation}\label{eq:back:conc_I}
  \frac{1}{2\pi i} \int_{(2)} \exp\left(\frac{\pi s^2}{Y^2}\right) \frac{X^s}{Y}ds =
  \frac{1}{2\pi} \exp\left(-\frac{Y^2 \log^2 X}{4\pi}\right).
\end{equation}

\begin{proof}
  Write $X^s = e^{s\log X}$ and complete the square in the exponents.
  Since the integrand is entire and the integral is absolutely convergent, performing the
  change of variables $s \mapsto s-Y^2 \log X/2\pi$ and shifting the line of integration
  back to the imaginary axis yields
  \begin{equation*}
    \frac{1}{2\pi i} \exp\left( - \frac{Y^2 \log^2 X}{4\pi}\right) \int_{(0)} e^{\pi
    s^2/Y^2} \frac{ds}{Y}.
  \end{equation*}
  The change of variables $s \mapsto isY$ transforms the integral into the standard
  Gaussian, completing the proof.
\end{proof}

Applied to a Dirichlet series $D(s)$, we have
\begin{equation}
  \frac{1}{2\pi} \sum_{n \geq 1} a(n) \exp \bigg( - \frac{Y^2 \log^2 (X/n)}{4\pi} \bigg) =
  \frac{1}{2\pi i} \int_{(\sigma)} \exp \big( \frac{\pi s^2}{Y^s} \big) \frac{X^s}{Y} ds.
\end{equation}
Note in particular that when $\lvert n - X \rvert$ is large, there is significant
exponential decay.
Therefore this integral concentrates the mass of the expression very near $X$ (and in
particular in an interval of width $X/Y$ around $X$).

\subsubsection*{Cutoff Integrals from Smooth, Compactly Supported Functions}
\index{Integral transform!compactly supported transform}
\index{phi@$\Phi_Y(s) and \phi_Y(n)$}

It will also be useful to document a more general family of cutoff transforms.
For $X,Y > 0$, let $\phi_Y(X)$ denote a smooth non-negative function with maximum value $1$,
satisfying
\begin{enumerate}
  \item $\phi_Y(X) = 1$ for $X \leq 1$,
  \item $\phi_Y(X) = 0$ for $X \geq 1 + \frac{1}{Y}$.
\end{enumerate}
Let $\Phi_Y(s)$ denote the Mellin transform of $\phi_Y(x)$, given by
\begin{equation}
  \Phi_Y(s) = \int_0^\infty t^s \phi_Y(t) \frac{dt}{t},
\end{equation}
defined initially for $\Re s > 0$.
Repeated applications of integration by parts and differentiation under the integral shows
that $\Phi_Y(s)$ satisfies the following four properties:
\begin{enumerate}
  \item $\Phi_Y(s) = \frac{1}{s} + O_s(\frac{1}{Y})$,
  \item $\Phi'_Y(s) = -\frac{1}{s^2} + O_s(\frac{1}{Y})$
  \item $\Phi_Y(s) = -\frac{1}{s} \int_1^{1 + \frac{1}{Y}} \phi'_Y(t) t^s dt$,
  \item and for all positive integers $m$, for $s$ constrained within a vertical strip and
    $\lvert s-1 \rvert > \epsilon$, we have
    \begin{equation}
      \Phi_Y(s) \ll_\epsilon \frac{1}{Y} \Big( \frac{Y}{1 + \lvert s \rvert} \Big)^m.
    \end{equation}
\end{enumerate}
Further, the last property can be extended to real $m > 1$ through the
Phragm\'{e}n-Lindel\"{o}f principle.
The Mellin transform pair $\Phi_Y(s), \phi_Y(x)$ gives a general set of integral cutoff
relations,
\begin{equation}
  \sum_{n \leq X} a(n) + \sum_{X < n \leq X + X/Y} a(n) \phi_Y\Big(\frac{n}{X}\Big) =
  \frac{1}{2\pi i} \int_{(\sigma)} D(s) \Phi_Y(s) X^s ds.
\end{equation}


\clearpage{\pagestyle{empty}\cleardoublepage}

\chapter{On Dirichlet Series for Sums of Coefficients of Cusp Forms}\label{c:sums}

\section{Introduction}

Continuing with same notation as before, let $f$ be a holomorphic cusp form of positive
weight $k$ on a congruence subgroup $\Gamma \subseteq \SL(2, \mathbb{Z})$, where $k \in
\mathbb{Z} \cup ( \mathbb{Z} + \tfrac{1}{2})$.
Denote the Fourier expansion of $f$ at $\infty$ by
\begin{equation}
  f(z) = \sum_{n \geq 1} a(n) e(nz).
\end{equation}
The individual coefficients $a(n)$ have long been of interest since the coefficients
contain interesting arithmetic data.
For example, the major insight leading to the resolution of Fermat's Last Theorem involved
showing that for an elliptic curve $E$ there exists a corresponding modular form $f_E$
whose coefficients (at prime indices) satisfy
\begin{equation}
  a(p) = p + 1 - \# E(\mathbb{F}_p),
\end{equation}
or rather that the $a(p)$ counted the number of points on the elliptic curve over finite
fields.

The first cusp form to be studied in depth was the Delta Function (as described in
Chapter~\ref{c:introduction}), whose coefficients are the Ramanujan $\tau$ function,
\index{Ramanujan tau}
\begin{equation}
  \Delta(z) = \sum_{n \geq 1} \tau(n) e(nz).
\end{equation}
Ramanujan conjectured that the coefficients of $\Delta$ should satisfy the bound
\begin{equation}
  \lvert \tau(n) \rvert \ll d(n) n^{\frac{11}{2}},
\end{equation}
where $d(n)$ is the number of positive divisors of $n$.
This conjecture initiated an exploration that included a much wider set of objects than
Ramanujan could have dreamt of.

Ramanujan's Conjecture has been extended to include all modular and automorphic forms.
For a cusp form $f$ of weight $k$ in $\GL(2)$, the conjecture states that
\begin{equation}\label{eq:ramanujan_conjecture_an}
  a(n) \ll n^{\frac{k-1}{2} + \epsilon}.
\end{equation}
This is now known as the Ramanujan-Petersson conjecture, and it is now known to be true
for full integral weight $k$ holomorphic cusp forms on $\GL(2)$ as a consequence of
Deligne's proof of the Weil Conjecture~\cite{Deligne}.

It is an interesting coincidence that Hardy and Littlewood were investigating averaged
estimates for the Gauss Circle and Dirichlet Divisor problems when Ramanujan was arriving
in England, thinking about the Delta function.
It is easy to look at the first several $\tau$ coefficients and believe that the sign of
$\tau(n)$ changes approximately uniformly random in $n$.
Under this assumption, and assuming Ramanujan's Conjecture that $\tau(n) \ll
n^{\frac{11}{2} + \epsilon}$, it is very natural to conjecture that the summatory function
of $\tau(n)$ satisfies the square-root cancellation phenomenon,
\begin{equation}
  \sum_{n \leq X} \tau(n) \ll X^{\frac{11}{2} + \frac{1}{2} + \epsilon}.
\end{equation}
As described in Chapter~\ref{c:introduction}, this is analogous to the error terms $E(R)$
in the Gauss Circle or Dirichlet Divisor Problems.
We similarly expect the even better bound,
\begin{equation}
  \sum_{n \leq X} \tau(n) \ll X^{\frac{11}{2} + \frac{1}{4} + \epsilon}.
\end{equation}
Although there is a clear qualitative connection with the Circle and Divisor problems, it
seems unlikely that this common thread was recognized by Hardy, Littlewood, or Ramanujan
at the time.

For our general cusp form $f$ of weight $k$ in $\GL(2)$, we expect an analogous conjecture
to hold, which we refer to as the ``Classical Conjecture.''
\index{Classical Conjecture}

\begin{conjecture}[Classical Conjecture]
  Let $f(z) = \sum_{n \geq 1} a(n) e(nz)$ be a holomorphic cusp form of weight $k$ on
  $\GL(2)$, where $k \in \mathbb{Z}\cup(\mathbb{Z}+\frac{1}{2})$ and $k > 1$.
  Then\index{S@$S_f(n)$}
  \begin{equation}
    S_f(n) := \sum_{n \leq X} a(n) \ll X^{\frac{k-1}{2} + \frac{1}{4} + \epsilon}.
  \end{equation}
\end{conjecture}

In a seminal pair of works, Chandrasekharan and
Narasimhan~\cite{chandrasekharan1962functional, chandrasekharan1964mean} showed that the
Classical Conjecture is true \emph{on average} by showing that
\begin{equation*}
  \sum_{n \leq X} \lvert S_f(n) \rvert^2 = C X^{k-1 + \frac{3}{2}} + B(X)
\end{equation*}
where $B(X)$ is an error term satisfying
\begin{equation}
  B(X) = \begin{cases}
    O(X^k \log^2 X) \\
    \Omega(X^{k - \frac{1}{4}} \frac{(\log \log \log X)^2}{\log X}),
  \end{cases}
\end{equation}
and where $C$ is an explicitly known constant.

This should be thought of as a Classical Conjecture \emph{on average} due to the following
classical argument:
\begin{align}
  \left( \sum_{n \leq X} \lvert S_f(n) \rvert \right)^2 &= \left( \sum_{n \leq X} \lvert
  S_f(n) \rvert \cdot 1 \right)^2
  \\
  &\leq \sum_{n \leq X} \lvert S_f(n) \rvert^2 \sum_{m \leq X} 1 = X \sum_{n \leq X}
  \lvert S_f(n) \rvert^2
  \\
  &\ll X^{k-1 + \frac{5}{2}}.
\end{align}
The Cauchy-Schwarz-Bunyakovsky inequality is used to pass from the first line to the
second, and the bound of Chandrasekharan and Narasimhan is used to pass from the second to
the third.
Taking the square root of each side and dividing by $X$ gives
\begin{equation}
  \frac{1}{X} \sum_{n \leq X} \lvert S_f(n) \rvert \ll X^{\frac{k-1}{2} + \frac{1}{4}},
\end{equation}
which is precisely the statement that the Classical Conjecture holds \emph{on average}.
Their result is described more completely in \S\ref{ssec:CN}.

Building on this result, Hafner and Ivi\'c were able to show that for holomorphic cusp
forms of full integral weight on $\SL(2, \mathbb{Z})$, we know
\begin{equation}
  S_f(n) \ll X^{\frac{k-1}{2} + \frac{1}{3}}.
\end{equation}
The argument of Hafner and Ivi\'c only applies for holomorphic forms of full-integral
eight and of level one, but it is possible to provide some extension to their result using
their methodology.

In the rest of this chapter, we will examine a new method for examining the behavior of $S_f(n)$.
We will be able to study a slightly more general object.
Let $g = \sum_{n \geq 1} b(n) e(nz)$ be another modular form of weight $k$ and of the same
level as $f$.
The fundamental idea is to study the Dirichlet series
\index{Ds@$D(s, S_f \times S_f)$}
\begin{align}
  D(s, S_f) &= \sum_{n \geq 1} \frac{S_f(n)}{n^{s + \frac{k-1}{2}}} \\
  D(s, S_f \times S_g) &= \sum_{n \geq 1} \frac{S_f(n) S_g(n)}{n^{s + k-1}} \\
  D(s, S_f \times \overline{S_g}) &= \sum_{n \geq 1}
  \frac{S_f(n) \overline{S_g(n)}}{n^{s + k-1}}.
\end{align}
In the sequel, we show that these three Dirichlet series have meromorphic continuation to
$\mathbb{C}$.
In \S\ref{sec:secondmoment}, we show how to analyze these Dirichlet series
to prove results concerning average sizes of the partial sums $S_f(n)$.

\begin{remark}
  Note that the notation used in this thesis is different than the notation used in the
  series of papers~\cite{hkldw, hkldwShort, hkldwSigns}.
  In this thesis, we adopt the convention that has risen amidst the representation
  theoretic point of view on automorphic forms.
  Therefore $L(s, f \times f)$ in this thesis is the same as $L(s, f\times \overline{f})$
  in the papers.
  This difference is ultimately very minor, and does not change any aspect of the
  analysis.
\end{remark}

\section{Useful Tools and Notation Reference}

For ease of reference, we give a notational reference and a brief description of some of
the tools necessary for the analysis.

\subsection{The Rankin--Selberg $L$-function}\label{ssec:rankinselberg_lfunction}
\index{Rankin--Selberg $L$-function} 
\index{L@$L(s, f\times g)$}

The Rankin--Selberg convolution $L$-function is described in detail 
in~\cite{Goldfeld2006automorphic, Bump98}, but we summarize its construction and
properties.
Note that there is a choice of convention concerning notation for conjugation.
The more common convention is changing due to influence from more general lines of
inquiry.

Let $f(z) = \sum a(n) e(nz)$ and $g(z) = \sum b(n) e(nz)$ be modular forms of weight $k$
on a congruence subgroup $\Gamma \subseteq \SL(2, \mathbb{Z})$, where at least one is
cuspidal.
Let $\Gamma \backslash \mathcal{H}$ denote the upper half plane modulo the group action
due to $\Gamma$, and let $\langle f, g \rangle$ denote the Petersson inner product
\index{Petersson inner product}
\begin{equation}
  \langle f, g \rangle = \iint_{\Gamma \backslash \mathcal{H}} f(z) \overline{g(z)}
  \frac{dx dy}{y^2}.
\end{equation}
The Rankin--Selberg $L$-function is given by the Dirichlet series 
\begin{equation}
  L(s, f\times \overline{g}) := \zeta(2s) \sum_{n \geq 1} \frac{a(n) \overline{b(n)}}{n^{s
  + k - 1}},
\end{equation}
which is absolutely convergent for $\Re s > 1$.
This $L$-function has a meromorphic continuation to all $s \in \mathbb{C}$ via the
identity
\begin{equation}\label{eq:Lsfgbar_equals_eisenstein}
  L(s, f\times \overline{g}) := \frac{(4\pi)^{s + k - 1} \zeta(2s)}{\Gamma(s + k - 1)}
  \langle \Im(\cdot)^k f \overline{g}, E(\cdot, \overline{s})\rangle,
\end{equation}
where $E(z,s)$ is the real-analytic Eisenstein series\index{E@$E(z,s)$}
\begin{equation}\label{eq:eisenstein_def}
  E(z,s) = \sum_{\gamma \in \Gamma_\infty \backslash \Gamma} \Im(\gamma z)^s.
\end{equation}

If, in~\eqref{eq:Lsfgbar_equals_eisenstein}, we replace $f \overline{g}$ with $f T_{-1}g$,
where $T_{-1}$ is the Hecke operator giving the action
\begin{equation}
  T_{-1} F(x + iy) = F(-x + iy),
\end{equation}
then one gets a meromorphic continuation of
\begin{equation}
  L(s, f\times g) = \zeta(2s) \sum_{n \geq 1} \frac{a(n) b(n)}{n^{s + k - 1}}.
\end{equation}
More details on the $T_{-1}$ Hecke operator can be found in the discussion leading up to
Theorem~3.12.6 of~\cite{Goldfeld2006automorphic}.
As the meromorphic properties of both are determined by the zeta function, Gamma function,
and Eisenstein series, we see that complex analytic arguments are very similar on either
$L(s, f\times \overline{g})$ or $L(s, f\times g)$.
We will only carry out the argument for $L(s, f\times \overline{g})$, and
describe any changes necessary to perform the argument on the other.

These Rankin--Selberg $L$-functions are holomorphic except for, 
at most, a simple pole at $s = 1$ with residue proportional to $\langle f, g \rangle$.
When $\Gamma = \SL(2, \mathbb{Z})$, there is the functional equation
\begin{equation}
  (2\pi)^{-2s} \Gamma(s) \Gamma(s + k - 1) L(s, f\times \overline g) =: \Lambda(s, f\times
  \overline{g}) = \Lambda(1 - s, f\times \overline{g}),
\end{equation}
coming from the functional equation of the completed Eisenstein series
\begin{equation}
  \pi^{-s} \Gamma(s) \zeta(2s) E(z,s) =: E^*(z,s) = E^*(z, 1-s).
\end{equation}
There are analogous transformations for higher levels, but their formulation is a bit more
complicated due to the existence of other cusps.

\subsection{Selberg spectral expansion}\label{ssec:selberg-spectral}
\index{spectral expansion}
\index{Maass forms}
\index{mu@$\mu_j$ | see {Maass forms} }
\index{Selberg Eigenvalue Conjecture}

Let $L^2(\Gamma \backslash \mathcal{H})$ denote the space of square integrable functions
on $\Gamma \backslash \mathcal{H}$ with respect to the Petersson norm.
There is a complete orthonormal system for the residual and cuspidal spaces of $\Gamma
\backslash \mathcal{H}$, which we denote by $\{\mu_j(z): j \geq 0\}$, consisting of the
constant function $\mu_0(z)$ and infinitely many Maass cusp forms $\mu_j(z)$ for $j \geq
1$ with associated eigenvalues $\tfrac{1}{4} + t_j^2$ with respect to the hyperbolic
Laplacian.
Without loss of generality, we assume that the $\mu_j$ are also simultaneous
eigenfunctions of the standard Hecke operators, as well as the $T_{-1}$ operator.
Then for any $f \in L^2(\Gamma \backslash \mathcal{H})$, we have the \emph{spectral
decomposition} of $f$ given by
\begin{equation}
  f(z) = \sum_j \langle f, \mu_j \rangle \mu_j(z) + \sum_{\mathfrak{a}} \frac{1}{4\pi}
  \int_\mathbb{R} \langle f, E_\mathfrak{a}(\cdot, \tfrac{1}{2} + it)\rangle
  E_\mathfrak{a}(a, \tfrac{1}{2} + it) \; dt,
\end{equation}
where $\mathfrak{a}$ ranges of the cusps of $\Gamma \backslash \mathcal{H}$.  Throughout
this thesis, we will refer to the first sum as the \emph{discrete spectrum} and the sums
of integrals as the \emph{continuous spectrum}.
\index{discrete spectrum}
\index{continuous spectrum}
The spectral decomposition as presented here is a consequence of Selberg's Spectral
Theorem, as presented in Theorem~15.5 of~\cite{IwaniecKowalski04}.

To each Maass form $\mu_j$ is associated a type $\tfrac{1}{2} + it_j$, and these $it_j$
are expected to satisfy Selberg's Eigenvalue Conjecture, which says that all $t_j$ are
real.
In complete generality, it is known that $t_j$ is either purely real or purely imaginary.
Selberg's Eigenvalue Conjecture has been proved for many congruence subgroups, including
$\SL(2, \mathbb{Z})$, but it is not known in general.
\index{Selberg's Eigenvalue Conjecture}
We let $\theta = \sup_j \{ \lvert \Im(t_j) \rvert \}$ denote the best known progress
towards Selberg's Eigenvalue Conjecture for $\Gamma$.
\index{theta@$\theta$ in Selberg's Eigenvalue Conjecture}
The current best result for $\theta$ in all congruence subgroups that $\theta \leq
\tfrac{7}{64}$, due to Kim and Sarnak~\cite{KimSarnak03}.

\subsection{Decoupling integral transform}\label{ssec:mellinbarnes_decouple}
\index{Mellin Barnes transform}

We will use an integral analogue of the binomial theorem, originally considered by
Barnes~\cite{Barnes}, also presented in 6.422(3) of~\cite{GradshteynRyzhik07}.
\begin{lemma}[Barnes, 1908]
  If $0 > \gamma > - \Re s$ and $\lvert \arg t \rvert < \pi$, then
  \begin{equation}\label{eq:mellinbarnes_base}
    \frac{1}{2\pi i} \int_{(\gamma)} \Gamma(-z) \Gamma(s + z) t^z \; dz = \Gamma(s) (1 +
    t)^{-s}.
  \end{equation}
\end{lemma}
This is a corollary to an integral representation of the beta function,
\begin{equation}
  B(z, s) = \int_0^\infty \frac{x^z}{(1 + x)^{z + s}} \; \frac{dx}{x},
\end{equation}
which implies that
\begin{equation}\label{eq:mellinbarnes_beta_is_mellin_transform}
  B(z, s-z) = \int_0^\infty \frac{x^z}{(1+x)^{s}} \; \frac{dx}{x}.
\end{equation}
The right hand side is a Mellin transform,
so~\eqref{eq:mellinbarnes_beta_is_mellin_transform} indicates that $B(z, s-z)$ is the
Mellin transform of $(1+x)^{-s}$ (with auxiliary variable $z$).
Applying the Mellin Inversion Theorem (as shown in~\cite{titchmarshfourier}) and the
representation of the Beta function in terms of Gamma functions, $B(s,t) =
\Gamma(s)\Gamma(t) / \Gamma(s+t)$, we recover a proof of the Lemma.

We will apply this integral transform to decouple $m,n$ in $(m+n)^{-s} = m^{-s} (1 +
\tfrac{n}{m})^{-s}$.
It is easy to check that an application of the lemma (followed by a change of variables $z
\mapsto -z$) gives
\begin{equation}
  \frac{1}{(n+m)^s} = \frac{1}{2\pi i} \int_{(\gamma)}
  \frac{\Gamma(z)\Gamma(s-z)}{\Gamma(s)} \frac{1}{n^{s-z} m^z} \; dz.
\end{equation}

\subsection{Chandrasekharan and Narasimhan}\label{ssec:CN}
\index{Chandrasekharan and Narasimhan}

We will refer to a result of Chandrasekharan and Narasimhan through this thesis.
We combine~\cite[Theorem~4.1]{chandrasekharan1962functional}
and~\cite[Theorem~1]{chandrasekharan1964mean} to state the following
theorem of theirs.

\begin{theorem}[Chandrasekharan and Narasimhan, 1962 and 1964]
  Let $f$ and $g$ be objects with meromorphic Dirichlet series
  \begin{equation*}
    L(s,f) = \sum_{n \geq 1} \frac{a(n)}{n^s}, \qquad L(s,g)=\sum_{n\geq 1}
    \frac{b(n)}{n^s}.
  \end{equation*}
  Suppose $G(s) = Q^s\prod_{i=1}^\ell \Gamma(\alpha_i s + \beta_i)$ is a product of Gamma
  factors with $Q>0$ and $\alpha_i > 0$.
  Define $A = \sum_{i=1}^\ell \alpha_i$.
  Let $w$ and $w'$ be numbers such that $\sum_{n \leq X} \lvert b(n) \rvert^2 \ll X^{2w -
  1} \log^{w'} X$.
  Let
  \begin{equation*}
    Q(X) = \frac{1}{2\pi i} \int_{\mathcal{C}} \frac{L(s,f)}{s} X^s \ ds,
  \end{equation*}
  where $\mathcal{C}$ is a smooth closed contour enclosing all the singularities of the
  integrand.
  Let $q$ be the maximum of the real parts of the singularities of $L(s,f)$ and let $r$ be
  the maximum order of a pole of $L(s,f)$ with real part $q$.
  Suppose the functional equation
  \begin{equation}\label{eq:basic_feq}
    G(s)L(s,f) = \epsilon(f) G(\delta-s)L(\delta-s,g)
  \end{equation}
  is satisfied for some $\lvert \epsilon(f) \rvert = 1$ and $\delta >0$.
  Then we have that
  \begin{equation}
    \begin{split}\label{eq:CN_L1}
      S_f(X) =& \sum_{n\leq X} a(n) =
      Q(X) + O(X^{\frac{\delta}{2}-\frac{1}{4A}+2A(w-\frac{\delta}{2}-\frac{1}{4A})\eta +
      \epsilon}) \\
      & \quad + O(X^{q-\frac{1}{2A}-\eta}\log(X)^{r-1})+O\bigg( \sum_{X \leq n \leq X'}
    |a(n)| \bigg)
    \end{split}
  \end{equation}
  for any $\eta \geq 0$, and where $X' =X+O(X^{1-\frac{1}{2A}-\eta})$.
  If all $a(n) \geq 0$, the final $O$-error term above does not contribute.

  Suppose further that $A \geq 1$ and that $2w - \delta - \frac{1}{A} \leq 0$.
  Then
  \begin{equation*}
    \sum_{n \leq X} \lvert S_f(n) - Q(n) \rvert^2 = cX^{\delta - \frac{1}{2A} + 1} +
    O(X^{\delta} \log^{w' + 2} X)
  \end{equation*}
  for a constant $c$ that can be made explicit.
\end{theorem}

Thus from little more than a functional equation with understood Gamma factors, one can
produce nontrivial bounds on first and second moments.

\section{Meromorphic Continuation}
\index{D@$D(s, S_f \times S_f)$}

We will now produce the meromorphic continuations of $D(s, S_f)$ and $D(s, S_f \times
\overline{S_g})$.
The meromorphic continuation of $D(s, S_f \times S_g)$ follows from applying the exact
same methodology to $T_{-1} g$ in place of $g$, so we only write down the corresponding
results.
In this section, we explicitly show the results in the case when $\Gamma = \SL(2,
\mathbb{Z})$.
For higher level congruence subgroups, the same methodology will work.
We remark on this in \S\ref{sec:higherlevel}.

Throughout this section, let $f(z) = \sum a(n) e(nz)$ and $g(z) = \sum b(n) e(nz)$ be
weight $k$ cusp forms on $\SL(2, \mathbb{Z})$.
Define $S_f(n) := \sum_{m \leq n} a(m)$ to be the partial sum of the first $n$ Fourier
coefficients of $f$.
Define $S_g(n)$ similarly.

We first describe the meromorphic continuation of $D(s, S_f)$.
The main ideas of this continuation are very similar to those in the continuation of $D(s,
S_f \times \overline{S_g})$, but the details are much simpler.
We then produce the meromorphic continuation of $D(s, S_f \times \overline{S_g})$ in
\S\ref{ssec:mero_DsSfSg}.
We shall see that the main obstacle is understanding the shifted convolution sum
\begin{equation}
  \sum_{n,h \geq 1} \frac{a(n)\overline{b(n-h)}}{n^{s+k-1} h^w}.
\end{equation}
We then use these meromorphic continuations to understand $D(s, S_f \times
\overline{S_g})$ in the next section.

\subsection{Meromorphic continuation of $D(s, S_f)$}
\index{D@$D(s, S_f)$}

The meromorphic continuation of $D(s, S_f)$ is very simple.
The pattern of the proof is similar to the pattern necessary for $D(s, S_f \times
\overline{S_g})$, so it is useful to be very clear.
The proof proceeds in two steps:
\begin{enumerate}
  \item Decompose the Dirichlet series into sums of Dirichlet series that are easier to
    understand, and
  \item Understand the reduced Dirichlet series by relating them to $L$-functions.
\end{enumerate}

\begin{proposition}
  With $f$ and $S_f(n)$ as defined above, the Dirichlet series associated to $S_f(n)$
  decomposes into
  \begin{equation}\label{eq:DsSf_decomposition}
    D(s, S_f) := \sum_{n \geq 1} \frac{S_f(n)}{n^{s + \frac{k-1}{2}}} = L(s, f) +
    \frac{1}{2\pi i} \int_{(2)} L(s-z, f) \zeta(z) \frac{\Gamma(z) \Gamma(s +
    \frac{k-1}{2} - z)}{\Gamma(s + \frac{k-1}{2})} \; dz,
  \end{equation}
  valid for $\Re s > 3$.
  Here, $L(s, f)$ denotes the standard $L$-function associated to $f$, given by
  \begin{equation}
    L(s, f) := \sum_{n \geq 1} \frac{a(n)}{n^{s + \frac{k-1}{2}}},
  \end{equation}
  normalized to have functional equation of the form $s \mapsto 1-s$.
\end{proposition}

\begin{proof}
  We directly manipulate the Dirichlet series.
  \begin{equation}
    D(s, S_f) = \sum_{n \geq 1} \frac{S_f(n)}{n^{s + \frac{k-1}{2}}} = \sum_{\substack{n
    \geq 1 \\ m \geq 0}} \frac{a(n-m)}{n^{s + \frac{k-1}{2}}}.
  \end{equation}
  In this last equality, we adopt the convention that $a(n) = 0$ for $n \leq 0$ to
  simplify notation.
  Separate the $m = 0$ case and reindex the remaining sum with $n \mapsto n+m$ to get
  \begin{equation}
    \sum_{n \geq 1} \frac{a(n)}{n^{s + \frac{k-1}{2}}} + \sum_{m,n \geq 1}
    \frac{a(n)}{(n+m)^{s + \frac{k-1}{2}}}.
  \end{equation}

  The first sum is exactly $L(s, f)$.
  In the second sum, decouple $(n+m)^{-s}$ through the use of the Mellin-Barnes transform
  detailed in \S\ref{ssec:mellinbarnes_decouple}.
  For $ \gamma > 1$ and $\Re s$ sufficiently large, the $m$ sum can be collected into
  $\zeta(z)$ and the $n$ sum can be collected into $L(s + \frac{k-1}{2})$.
  Simplification completes the proof.
\end{proof}

The meromorphic continuation of the $L$-function $L(s,f)$ is well-understood.
Note that in
\begin{equation}
  \frac{1}{2\pi i} \int_{(2)} L(s - z, f) \zeta(z) \frac{\Gamma(z) \Gamma(s +
  \frac{k-1}{2} - z)}{\Gamma(s + \frac{k-1}{2})} \; dz,
\end{equation}
the integrand is meromorphic in both $s$ and $z$, and has exponential decay in vertical
strips in $\Im z$ for any individual $s$.
Therefore one can use the meromorphic continuation of $L(s,f)$ to understand that this
integral is meromorphic for all $s \in \mathbb{C}$.

Note that it is also possible to shift the line of $z$ integration arbitrarily far in the
negative direction, passing poles at $z = 1, 0, -1, \ldots$ and picking up their residues.
Shifting the line of $z$ integration to $\epsilon$ for a small $\epsilon > 0$ passes
exactly one pole, coming from $\zeta(z)$ at $z = 1$, with residue
\begin{equation}
  \Res_{z = 1} = L(s - 1, f) \frac{\Gamma(s + \frac{k-1}{2} - 1)}{\Gamma(s +
  \frac{k-1}{2})} = \frac{L(s-1, f)}{s + \frac{k-1}{2} - 1}.
\end{equation}
Therefore, we have the equality
\begin{equation}
  D(s, S_f) = L(s, f) + \frac{L(s-1, f)}{s + \frac{k-1}{2} - 1} + \frac{1}{2\pi i}
  \int_{(\epsilon)} L(s - z, f) \zeta(z) \frac{\Gamma(z)\Gamma(s + \frac{k-1}{2} -
  z)}{\Gamma(s + \frac{k-1}{2})} \; dz.
\end{equation}
The first term is analytic, the third term is analytic in $s$ for $\Re s > \epsilon -
\frac{k-1}{2}$, and the middle term appears to have a simple pole at $s = 1 -
\frac{k-1}{2}$.
However, from the functional equation equation of $L(s,f)$,
\begin{equation}
  \Lambda(s,f) := (2\pi)^{-(s + \tfrac{k-1}{2})} \Gamma(s + \tfrac{k-1}{2}) L(s,f)
  =
  \varepsilon \Lambda(1-s, f),
\end{equation}
we see that the residue $L(-\tfrac{k-1}{2})$ is a trivial zero of the $L$-function,
and $s = 1 - \frac{k-1}{2}$ is not a pole after all.

Further shifting the line of $z$ integration to $-M + \epsilon$ for small $\epsilon > 0$
passes $(M)$ further poles, coming from $\Gamma(z)$ at $z = -j$ for $0 \leq j < M$.
The $j$th pole, located at $z = -j, 0 \leq j < M$ has residue
\begin{equation}
  \Res_{z = -j} = L(s + j, f) \zeta(-j) \frac{\Gamma(s + \frac{k-1}{2} + j)}{\Gamma(s +
  \frac{k-1}{2})} \frac{(-1)^j}{j!}.
\end{equation}
In each residue, the $L$-function is analytic, and all apparent poles of the Gamma
function in the numerator are cancelled by poles of the Gamma function in the denominator.
Therefore each residue is analytic in $s$.

For any integer $M \geq 0$, we therefore have that
\begin{equation}
  D(s, S_f) = L(s, f) + \sum_{j = -1}^{M-1} \Res_j + \frac{1}{2\pi i} \int_{(-M +
  \epsilon)} L(s - z, f)\zeta(z) \frac{\Gamma(z) \Gamma(s + \frac{k-1}{2} - z)}{\Gamma(s +
\frac{k-1}{2})} dz.
\end{equation}
The $L$-function and residues are analytic in $s$.
The integral term is analytic in $s$ for $\Re s > -\frac{k-1}{2} - M + \epsilon$.
Since $M$ is arbitrary, we see that $D(s, S_f)$ is actually analytic for all $s \in
\mathbb{C}$.
We record this observation as a corollary to the decomposition of $D(s, S_f)$.

\begin{corollary}
  The Dirichlet series $D(s, S_f)$ has analytic continuation to the whole complex plane,
  given by~\eqref{eq:DsSf_decomposition}.
\end{corollary}

\begin{remark}
  I note that if $f$ is not cuspidal, then the decomposition and most of the analysis of
  $D(s, S_f)$ carries over verbatim, with one key difference: the value
  $L(-\tfrac{k-1}{2})$ is no longer a trivial zero.
  Therefore it is possible to show in general that $D(s, S_f)$ is meromorphic in the plane
  with at most several simple poles with residues given by special values of $L(s, f)$.

  This indicates a very strong parallel between the properties of $L(s, f)$ and
  $D(s, S_f)$.
  But it should be noted that $D(s, S_f)$ does not have a simple functional equation or an
  Euler product.
\end{remark}

\subsection{Meromorphic continuation of $D(s, S_f \times \overline{S_g})$}
\label{ssec:mero_DsSfSg}

The meromorphic continuation of $D(s, S_f \times \overline{S_g})$ is a bit involved, but
the approach is very similar to the approach for $D(s, S_f)$.
We proceed in three steps:
\begin{enumerate}
  \item Decompose $D(s, S_f \times \overline{S_g})$ into sums of Dirichlet series that are
    easier to understand,
  \item Related the reduced Dirichlet series to $L$-functions and convolution sums, and
  \item Combine the analytic properties of the $L$-functions and convolution sums.
\end{enumerate}
The first two steps are fundamentally the same as in$D(s, S_f)$, except that convolution
sums are necessary in the analysis.
As the meromorphic properties of convolution sums are significantly more delicate,
ascertaining the final analytic properties will take much more work.
The final step is deferred to \S\ref{sec:analyticbehavior}.

\index{Ds@$D(s, S_f \times S_f)$}
\index{Ws@$W(s; f, f)$}
\begin{proposition}\label{prop:SfSg_decomposition}
  With $f, g, S_f(n)$, and $S_g(n)$ as defined above, the Dirichlet series associated to
  $S_f(n)\overline{S_g(n)}$ decomposes into
  \begin{equation}
    \begin{split}
      D(s, &S_f \times \overline{S_g}) := \sum_{n \geq 1} \frac{S_f(n)
      \overline{S_g(n)}}{n^{s + k - 1}} \\
      &= W(s; f,\overline{g}) + \frac{1}{2\pi i} \int_{(\gamma)} W(s - z; f,\overline{g})
      \zeta(z) \frac{\Gamma(z) \Gamma(s - z + k - 1)}{\Gamma(s + k - 1)} \; dz,
    \end{split}
  \end{equation}
  for $1 < \gamma < \Re(s-1)$.
  Here, $W(s; f, \overline{g})$ denotes
  \begin{equation}
    W(s; f,\overline{g}) := \frac{L(s, f\times \overline{g})}{\zeta(2s)} + Z(s, 0, f\times
    \overline{g}),
  \end{equation}
  $L(s, f\times \overline{g})$ denotes the Rankin--Selberg $L$-function 
  as in~\ref{ssec:rankinselberg_lfunction},
  and $Z(s, w, f\times \overline{g})$ denotes the symmetrized
  shifted convolution sum\index{Zs@$Z(s, w, f\times f)$}
  \begin{equation}
    Z(s,w,f\times \overline{g}) := \sum_{n,h \geq 1} \frac{a(n)\overline{b(n-h)} +
    a(n-h)\overline{b(n)}}{n^{s + k - 1} h^w}.
  \end{equation}
\end{proposition}

\begin{proof}
  Expand and recollect the partial sums $S_f$ and $S_g$.
  \begin{equation}
    D(s, S_f \times \overline{S_g}) = \sum_{n \geq 1} \frac{S_f(n) \overline{S_g(n)}}{n^{s
    + k - 1}} = \sum_{n \geq 1} \frac{1}{n^{s + k - 1}} \sum_{m = 1}^n a(m) \sum_{h = 1}^n
    \overline{b(h)}.
  \end{equation}
  Separate the sums over $m$ and $h$ into the cases where $m=h, m>h$, and $m<h$.
  We again adopt the convention that $a(n) = 0$ for $n \leq 0$ to simplify notation.
  Reorder the sums, summing down from $n$ instead of up to $n$, giving
  \begin{equation}
    \sum_{n \geq 1} \frac{1}{n^{s + k - 1}} \Bigg( \sum_{h = m > 0} + \sum_{h > m \geq 0}
    + \sum_{m > h \geq 0}\Bigg) a(n-m) \overline{b(n-h)}.
  \end{equation}
  In the first sum, take $h = m$.
  In the second sum, when $h > m$, we let $h = m + \ell$ and then sum over $m$ and $\ell$.
  Similarly in the third sum, when $m > h$, we let $m = h + \ell$.
  Together, this yields
  \begin{equation}
    \begin{split}
      \sum_{n \geq 1} &\frac{1}{n^{s + k - 1}} \bigg( \sum_{m \geq 0} a(n-m)
      \overline{b(n-m)} \\
      &+ \sum_{\substack{\ell \geq 1 \\ m \geq 0}} a(n-m) \overline{b(n-m-\ell)} +
    \sum_{\substack{\ell \geq 1 \\ m \geq 0}} a(n-m-\ell)\overline{b(n-m)}\bigg).
    \end{split}
  \end{equation}

  In each sum, the cases when $m = 0$ are distinguished.
  Altogether, these contribute
  \begin{equation}
    W(s; f,\overline{g}) = \frac{L(s, f\times \overline{g})}{\zeta(2s)} + Z(s, 0, f\times
    \overline{g}).
  \end{equation}
  Within $W(s; f, \overline{g})$, one should think of $L(s, f\times
  \overline{g})\zeta(2s)^{-1}$ as the diagonal part of the double summation, while
  $Z(s, 0, f\times \overline{g})$ contains the off-diagonal, written as the sum of the
  above-diagonal and below-diagonal parts of the double summation.

  Reindexing by changing $n \mapsto n + m$, the remaining sums with $m \geq 1$ can be rewritten as
  \begin{equation}
    \sum_{m,n \geq 1} \frac{1}{(n+m)^{s + k - 1}}\bigg( a(n) \overline{b(n)} + \sum_{\ell
    \geq 1} a(n) \overline{b(n - \ell)} + \sum_{\ell \geq 1} a(n - \ell)
  \overline{b(n)}\bigg).
  \end{equation}
  Decouple $(n+m)^{-(s + k - 1)}$ through the use of the Mellin-Barnes transform detailed
  in~\S\ref{ssec:mellinbarnes_decouple}.
  Restricting to $\gamma > 1$ and $\Re s$ sufficiently large, we can collect the $m$ sum
  into $\zeta(z)$ and the $n$ sum can be colleced into $W(s; f,g)$.
  Simplification completes the proof.
\end{proof}

As in the case of $D(s, S_f)$, the individual pieces $L(s, f\times \overline{g})
\zeta(2s)^{-1}$ and $Z(s, w, f\times \overline{g})$ are known to have meromorphic
continuation to the complex plane.
The Rankin--Selberg $L$-function is classical, 
and its meromorphic continuation is explained in~\S\ref{ssec:rankinselberg_lfunction}.
The meromorphic properties of $Z(s,w, f\times \overline{g})$ are extensively studied
in~\cite{HoffsteinHulse13}.
One should expect to be able to perform an analysis similar to the analysis for $D(s,
S_f)$ to study $D(s, S_f \times \overline{S_g})$, perhaps by shifting the line of $z$
integration and analyzing residue terms.

However, the shifted sums $Z(s, 0, f\times \overline{g})$ show miraculous cancellation
with the diagonal $L(s, f\times\overline{g}) \zeta(2s)^{-1}$ that does not occur in $Z(s,
w, f\times \overline{g})$ for general $w$.
We group $W(s; f,\overline{g})$ into a single object and study its analytic properties
in the next section.

\section{Analytic Behavior of $D(s, S_f \times \overline{S_g})$}\label{sec:analyticbehavior}
\index{Zs@$Z(s, w, f\times f)$}

In this section, we will understand the meromorphic continuation of $D(s, S_f \times
\overline{S_g})$ by studying the analytic properties of $W(s; f, \overline{g})$.
Although we work in level $1$, the methodology in this section generalizes to arbitrary
level and half-integral weight forms.
Therefore, we will use $\theta = \sup_j \{ \lvert \Im(t_j) \rvert \}$ to denote progress
towards Selberg's Eigenvalue Conjecture (as described in~\S\ref{ssec:selberg-spectral})
even though it is known that $\theta = 0$ in the level $1$ case.

We first produce a spectral expansion for the off-diagonal component, the symmetrized
shifted double Dirichlet series
\begin{equation}
  Z(s,w, f\times \overline{g}) := \sum_{m \geq 1} \sum_{\ell \geq 1} \frac{a(m)
  \overline{b(m - \ell)} + a(m - \ell)\overline{b(m)}}{m^{s + k - 1} \ell^w}.
\end{equation}
We then use this to understand the analytic behavior of $W(s; f, \overline{g})$ and, from
this, the analytic behavior of $D(s, S_f \times \overline{S_g})$.

\subsubsection{Spectral expansion}

For each integer $h \geq 1$, define the weight zero Poincar\'{e} series on $\Gamma$,
\index{Poincar\'{e} series}
\begin{equation}
  P_h(z,s):=\sum_{\gamma \in\Gamma_\infty \backslash \Gamma} \Im(\gamma z)^s e\left(h
  \gamma z\right),
\end{equation}
defined initially for $\Re(s)$ sufficiently large, but with meromorphic continuation to
all $s\in\mathbb{C}$.

Recall $T_{-1}$ from~\S\ref{ssec:rankinselberg_lfunction}.
Let $\mathcal{V}_{f,\overline{g}}(z) :=y^k (f \overline{g}+T_{-1}(f\overline{g}))$.
Note that $\mathcal{V}_{f, \overline{g}}(z) \in L^2(\Gamma \backslash \mathcal{H})$, so
the Petersson inner product $\langle \mathcal{V}_{f, \overline{g}}, P_h(\cdot,
\overline{s}) \rangle$ converges.
By expanding this inner product, we get
\begin{equation}
  \langle \mathcal{V}_{f,\overline{g}} ,P_h(\cdot, \overline{s}) \rangle
  =\frac{\Gamma\left({s}+k -1\right)}{\left(4\pi \right)^{{s}+ k -1}}
  D_{f,\overline{g}}({s};h),
\end{equation}
where we define $D_{f, \overline{g}}$ to have analogous notation as
in~\cite{HoffsteinHulse13},
\index{D@$D_{f,f}(s;h)$}
\begin{equation}
  D_{f,g}(s;h) := \sum_{n \geq
  1}\frac{a(n)\overline{b(n-h)}+a(n-h)\overline{b(n)}}{n^{s+k-1}},
\end{equation}
which converges absolutely for $\Re s$ sufficiently positive.
Dividing by $h^w$ and summing over $h \geq 1$ recovers $Z(s, w, f\times \overline{g} )$,
\begin{equation}\label{eq:Zsw_poincare}
  Z(s, w, f\times \overline{g} ) := \sum_{n, h \geq 1} \frac{D_{f, \overline{g}
  }(s;h)}{h^w} = \frac{(4\pi)^{s+k-1}}{\Gamma(s+k-1)} \sum_{h \geq 1} \frac{\langle
  \mathcal{V}_{f,\overline{g} }, P_h \rangle}{h^w},
\end{equation}
for $\Re s$ and $\Re w$ sufficiently positive.

We will obtain a meromorphic continuation of $Z(s, w, f\times \overline{g} )$ by using the
spectral expansion of the Poincar\'e series and substituting it
into~\eqref{eq:Zsw_poincare}.
Let $\{\mu_j\}$ be an orthonormal basis of Maass eigenforms with associated types
$\frac{1}{2} + it_j$ for $L^2(\Gamma \backslash \mathcal{H})$ as in
\S\ref{ssec:selberg-spectral}, each with Fourier expansion
\begin{equation}
  \mu_j(z)=\sum_{n \neq 0} \rho_j(n)y^{\frac{1}{2}}K_{it_j}(2\pi \vert n \vert y)
  e^{2\pi i n x}.
\end{equation}

The spectral expansion of the Poincar\'{e} series is given by
\begin{align}\label{eq:Pspectral}
  \begin{split}
    P_h(z,s)&=\sum_j \langle P_h(\cdot,s),\mu_j \rangle \mu_j(z) \\
            &\quad + \frac{1}{4\pi}\int_{-\infty}^\infty\langle
    P_h(\cdot,s),E(\cdot,\tfrac{1}{2}+it)\rangle E(z,\tfrac{1}{2}+it)\,dt.
  \end{split}
\end{align}
We shall refer to the above sum and integral as the discrete and continuous spectrum,
respectively, similar to the convention in~\S\ref{ssec:selberg-spectral}.

The inner product of $\mu_j$ against the Poincar\'{e} series gives
\begin{equation}\label{eq:mujP}
  \langle P_h(\cdot,s),\mu_j \rangle = \frac{\overline{\rho_j(h)}\sqrt{\pi}}{(4\pi
  h)^{s-\frac{1}{2}}}
  \frac{\Gamma(s-\frac{1}{2}+it_j)\Gamma(s-\frac{1}{2}-it_j)}{\Gamma(s)}.
\end{equation}
\begin{remark}
  In the computation of this inner product and the inner product of the Eisenstein series
  against the Poincar\'e series, we use formula~\cite[$\S$6.621(3)]{GradshteynRyzhik07} to
  evaluate the final integrals.
\end{remark}

%

Let $E(z,w)$ be the Eisenstein series on $\SL(2, \mathbb{Z})$, given by
\begin{equation}
  E(z,w) = \sum_{\gamma \in \Gamma_\infty \backslash \SL(2, \mathbb{Z})} (\Im \gamma z)^s.
\end{equation}
Then $E(z,w)$ has Fourier expansion (as in~\cite[Chapter 3]{Goldfeld2006automorphic})
\begin{align}\label{eq:sums:E_Fourier}
  E(z,w)& =y^w + \phi(w)y^{1-w}  \\
  &\quad +  \frac{2\pi^w \sqrt{y} }{\Gamma(w)\zeta(2w)}\sum_{m \neq 0} \vert m
  \vert^{w-\frac{1}{2}}\sigma_{1-2w}(\vert m \vert) K_{w-\frac{1}{2}}(2\pi \vert m \vert
  y)e^{2\pi i m x}, \notag
\end{align}
where
\begin{equation}
  \phi(w) = \sqrt{\pi} \frac{\Gamma(w - \tfrac{1}{2})\zeta(2w - 1)}{\Gamma(w)\zeta(2w)}.
\end{equation}
The Petersson inner product of the Poincar\'{e} series ($h \geq 1$) against the Eisenstein
series $E(z,w)$ is given by
\begin{align}\label{eq:PhE}
  \left\langle P_h(\cdot,s),E(\cdot,w)\right\rangle =\frac{2\pi^{\overline{w}+\frac{1}{2}}
  h^{\overline{w}-\frac{1}{2}}\sigma_{1-2\overline{w}}(h)}{\zeta(2\overline{w})(4\pi
  h)^{s-\frac{1}{2}}}\frac{\Gamma(s+\overline{w}-1)\Gamma(s-\overline{w})}{\Gamma(\overline{w})\Gamma(s)},
\end{align}
provided that $\Re s >\frac{1}{2}+\vert \Re w-\frac{1}{2}\vert$.
For $t$ real, $w=\frac{1}{2}+it$, and $\Re s>\frac{1}{2}$, we can
specialize~\eqref{eq:PhE} to
\begin{equation}\label{eq:PhEspecialized}
  \langle P_h(\cdot,s),E(\cdot,\tfrac{1}{2}+it)\rangle=\frac{2\sqrt{\pi}
  \sigma_{2it}(h)}{\Gamma(s)(4\pi
  h)^{s-\frac{1}{2}}}\frac{\Gamma(s-\frac{1}{2}+it)\Gamma(s-\frac{1}{2}-it)}{h^{it}\zeta^*(1-2it)},
\end{equation}
in which $\zeta^*(2s):=\pi^{-s}\Gamma(s)\zeta(2s)$ denotes the completed zeta function.

Now that we have computed the inner products of the Eisenstein series and Maass forms with
the Poincar\'e series, we are ready to analyze the spectral expansion.
After substituting~\eqref{eq:mujP} into the discrete part of~\eqref{eq:Pspectral}, the
discrete spectrum takes the form
\begin{equation}
  \frac{\sqrt{\pi}}{(4\pi h)^{s-\frac{1}{2}}\Gamma(s)}\sum_j \overline{\rho_j(h)}
  \Gamma(s-\tfrac{1}{2}+it_j)\Gamma(s-\tfrac{1}{2}-it_j) \mu_j(z)
\end{equation}
and is analytic in $s$ in the right half-plane $\Re s> \frac{1}{2}+\theta$.
After inserting~\eqref{eq:PhEspecialized}, the continuous spectrum takes the form
\begin{equation}
  \frac{\sqrt{\pi}}{2\pi(4\pi h)^{s-\frac{1}{2}}}\int_{-\infty}^\infty
  \frac{\sigma_{2it}(h)}{h^{it}}\frac{\Gamma(s-\frac{1}{2}+it)
  \Gamma(s-\frac{1}{2}-it)}{\zeta^*(1-2it)\Gamma(s)}E(z,\tfrac{1}{2}+it)\,dt,
\end{equation}
which is analytic in $s$ for $\Re s > \frac{1}{2}$ and has apparent poles when $\Re
s=\frac{1}{2}$.

Substiting this spectral expansion into~\eqref{eq:Zsw_poincare} and summing over $h \geq
1$ recovers an expression for all of $Z(s, w, f\times \overline{g})$.
Recognizing the Dirichlet series (as described in~\cite{Goldfeld2006automorphic}, for
instance)
\begin{align}
  \sum_{h \geq 1} \frac{\rho_j(h)}{h^{s + w - \frac{1}{2}}} &= L(s + w - \tfrac{1}{2},
  \mu_j) \\
  \sum \frac{\sigma_{1-2w}(h)}{h^{s + \frac{1}{2} - w}} &= L(s, E(\cdot, w)) = \zeta(s + w
  - \tfrac{1}{2})\zeta(s - w + \tfrac{1}{2}),
\end{align}
we are able to execute the $h$ sum completely.
Assembling it all together, we have proved the following proposition.

\begin{proposition}\label{prop:spectralexpansionfull}
  For $f,g$ weight $k$ forms on $\SL_2(\mathbb{Z})$, the shifted convolution sum $Z(s, w,
  f\times \overline{g} )$ can be expressed as
\begin{align}
  Z(s&, w, f\times \overline{g} ) := \sum_{m=1}^\infty \frac{a(m)
  \overline{b(m-h)}+a(m-h)\overline{b(m)} }{m^{s+k -1}h^w} \nonumber \\
  &= \frac{(4\pi )^k}{2} \sum_j\rho_j(1) G(s, i t_j) L(s + w -\tfrac{1}{2},\mu_j)\langle
\mathcal{V}_{f,\overline{g} },\mu_j \rangle \label{line:1spectralexp} \\
  &\quad+\frac{(4\pi)^{k}}{4\pi i}\int_{(0)} G(s, z) \mathcal{Z}(s,w,z) \langle
  \mathcal{V}_{f,\overline{g} },E(\cdot,\tfrac{1}{2}-\overline{z})\rangle \,dz,
  \label{line:2spectralexp}
\end{align}
when $\Re (s+w)>\frac{3}{2}$, where $G(s, z)$ and $\mathcal{Z} (s,w,z)$ are the collected
$\Gamma$ and $\zeta$ factors of the discrete and continuous spectra,
\begin{align*}
  G(s, z) &= \frac{\Gamma(s - \tfrac{1}{2} + z)\Gamma(s - \tfrac{1}{2} -
  z)}{\Gamma(s)\Gamma(s+k-1)} \\
  \mathcal{Z}(s,w,z) &= \frac{\zeta(s + w -\frac{1}{2} + z)\zeta(s + w -\frac{1}{2} -
  z)}{\zeta^*(1+2z)}.
\end{align*}
\end{proposition}

\begin{remark}\label{rem:extraremark}
Let's verify that this spectral expansion converges.
Recall Stirling's approximation: for $x,y \in \mathbb{R}$,
\index{Stirling's approximation}
\begin{equation}
  \gamma(x+iy) \sim (1+|y|)^{x-\frac{1}{2}}e^{-\frac{\pi}{2}|y|}
\end{equation}
as $y \to \pm\infty$ with $x$ bounded.
For vertical strips in $s$ and $z$,
\begin{equation}
  G(s,z) \sim P(s,z) e^{-\frac{\pi}{2}(2\max(|s|,|z|)-2|s|)},
\end{equation}
where $P(s,z)$ has at most polynomial growth in $s$ and $z$.
When $k$ is a full-integer, Watson's triple product formula (given in Theorem 3
of~\cite{watson2008rankin}) shows that
\begin{equation}
  \rho_j(1)\langle f\overline{g} \Im(\cdot)^k,\overline{\mu_j}\rangle, \quad \text{and}
  \quad  \rho_j(1)\langle T_{-1}(f\overline{g}) \Im(\cdot)^k,\overline{\mu_j}\rangle
\end{equation}
has at most polynomial growth in $|t_j|$.
When $k$ is a half-integer, K\i{}ral's bound (given in Proposition~13
of~\cite{mehmet2015}) also proves polynomial growth, albeit of a higher degree.
Through direct computation with the associated Rankin--Selberg $L$-function, 
the same can be said about
\begin{equation}
  \langle \mathcal{V}_{f,\overline{g} }, E(\cdot,\tfrac{1}{2}+z)\rangle / \zeta^*(1+2z).
\end{equation}
Both~\eqref{line:1spectralexp} and~\eqref{line:2spectralexp} converge uniformly on
vertical strips in $t_j$ and have at most polynomial growth in $s$.
\end{remark}

We will now specialize to $w = 0$ and analyze the meromorphic properties of $Z(s, 0,
f\times \overline{g})$.
This very naturally breaks into two parts: the contribution from the discrete spectrum (in
line~\eqref{line:1spectralexp}) and the contribution from the continuous spectrum (in
line~\eqref{line:2spectralexp}).

\subsubsection{Meromorphic continuation of $Z(s, 0, f\times \overline{g})$: discrete spectrum}

Examination of line~\eqref{line:1spectralexp}, the contribution from the discrete
spectrum, reveals that the poles come only from $G(s, it_j)$.
There are apparent poles when $s = \tfrac{1}{2} \pm it_j - n$ for $n \in \mathbb{Z}_{\geq
0}$.
Interestingly, the first set of apparent poles at $s = \frac{1}{2} \pm it_j$ do not
actually occur.

\begin{lemma}\label{lem:Litj_equals_zero}
  For even Maass forms $\mu_j$, we have $L(-2n \pm it_j, \mu_j) = 0$ for $n \in
  \mathbb{Z}_{\geq 0}$.
\end{lemma}

\begin{proof}
  The completed $L$-function associated to a Maass form $\mu_j$ is given by
  \begin{equation} \label{eq:feq}
    \Lambda_j(s) = \pi^{-s} \Gamma\left( \tfrac{s + \epsilon + it_j}{2}
    \right)\Gamma\left( \tfrac{s + \epsilon - it_j}{2} \right) L(s, \mu_j) = (-1)^\epsilon
    \Lambda_j(1-s),
  \end{equation}
  as in~\cite[Sec 3.13]{Goldfeld2006automorphic}, where $\epsilon = 0$ if the Maass form
  $\mu_j$ is even and $1$ if it is odd.
  In the case of an even Maass form, the functional equation is of shape
  \begin{equation}
    \Lambda_j(s) = \pi^{-s} \Gamma\left( \tfrac{s + it_j}{2} \right)\Gamma\left( \tfrac{s
    - it_j}{2} \right) L(s, \mu_j) = \Lambda_j(1-s).
  \end{equation}
  The completed $L$-function is entire.
  Thus $L(-2n \pm it_j, \mu_j)$ must be trivial zeroes to cancel the apparent poles at $s
  =-2n \pm it_j$ from the Gamma functions.
\end{proof}

It turns out that there are no poles appearing from odd Maass forms due to the symmetry of
the above-diagonal and below-diagonal terms in $\mathcal{V}_{f, \overline{g}}$.

\begin{lemma}\label{lem:oddorthogonaltoeven}
  Suppose $f$ and $g$ are weight $k$ cusp forms, as above. For odd Maass forms $\mu_j$, we
  have $\langle \mathcal{V}_{f,\overline{g}}, \mu_j \rangle = 0$.
\end{lemma}

\begin{proof}
  Recall that $\mathcal{V}_{f,g} :=y^k(f\overline{g}+T_{-1}(f\overline{g}))$, so clearly
  $T_{-1}\mathcal{V}_{f,\overline{g} }=\mathcal{V}_{f,\overline{g} }$.
  Recall also that $T_{-1} \mu_j = -\mu_j$ for odd Maass forms $\mu_j$ (in fact, this is
  the defining characteristic of an odd Maass form).
  Since $T_{-1}$ is a self-adjoint operator with respect to the Petersson inner product we
  have that
  \begin{equation}
    \langle \mathcal{V}_{f,\overline{g} },\mu_j \rangle
    = \langle T_{-1}\mathcal{V}_{f,\overline{g} },\mu_j \rangle
    = \langle \mathcal{V}_{f,\overline{g} },T_{-1}\mu_j \rangle
    = - \langle \mathcal{V}_{f,\overline{g} },\mu_j \rangle.
  \end{equation}
  Thus $\langle \mathcal{V}_{f,\overline{g} },\mu_j \rangle  =0$.
\end{proof}

In the special case when $f = g$, it is possible to show that odd Maass forms $\mu_j$ do
not contribute poles in \emph{either} the above-diagonal or below-diagonal terms in
$\mathcal{V}_{f, \overline{f}}$.
We record this observation as a corollary to the above lines of thought, even though it is
not necessary for the applications in this thesis.

\begin{corollary*}
  Suppose $f$ is a full-integral weight cuspidal Hecke eigenform, not necessarily with
  real coefficients.
  Then for odd Maass forms $\mu_j$, we have $\langle \lvert f \rvert^2 \Im(\cdot)^k, \mu_j
  \rangle = 0$.
  Similarly, we have $\langle f^2 \Im(\cdot)^k, \mu_j \rangle = 0$.
\end{corollary*}

\begin{proof}
  We sketch the proof. From Watson's well-known triple product
  formula~\cite{watson2008rankin}, we have
  \begin{equation}
    \langle \lvert f \rvert^2 \Im(\cdot)^k, \mu_j \rangle^2 \sim \frac{L(\frac{1}{2},
    f\times \overline{f} \times \mu_j)}{L(1, f, \text{Ad}) L(1, \overline{f}, \text{Ad})
    L(1, \mu_j, \text{Ad})}
  \end{equation}
  up to multiplication by a nonzero constant coming from the missing Gamma factors.
  The $L$-functions in the denominator are all nonzero, and the numerator factors as
  \begin{equation}
    L(\tfrac{1}{2}, f \times \overline{f} \times \mu_j) = L(\tfrac{1}{2}, \text{Ad}^2 f
    \times \mu_j) L(\tfrac{1}{2}, \mu_j).
  \end{equation}
  Since $\mu_j$ is odd, $L(\tfrac{1}{2}, \mu_j) = 0$ by the functional equation for odd
  Maass forms, given in~\eqref{eq:feq}.

  Applying Watson's triple product to $\langle f^2 \Im(\cdot)^k, \mu_j \rangle$ yields the
  numerator
  \begin{equation}
    L(\tfrac{1}{2}, \text{Sym}^2 f \times \mu_j) L(\tfrac{1}{2}, \mu_j),
  \end{equation}
  which is zero for the same reason.
\end{proof}

Lemma~\ref{lem:oddorthogonaltoeven} guarantees that the only Maass forms appearing in
line~\eqref{line:1spectralexp} are even. The first set of apparent poles from even Maass
forms appear at $s = \frac{1}{2} \pm it_j$ and occur as simple poles of the Gamma
functions in the numerator of $G(s, t_j)$. They come multiplied by the value of $L(it_j,
\mu_j)$, which by Lemma~\ref{lem:Litj_equals_zero} is zero.

In summary, $D(s, S_f \times S_g)$ has no poles at $s = \frac{1}{2} \pm it_j$.
The next set of apparent poles are at $s = -\frac{1}{2} \pm it_j$, appearing at the next
set of simple poles of the Gamma functions in the numerator.
Unlike the previous poles, these do not coincide with trivial zeroes of the $L$-function.
We have poles of the discrete spectrum at $s = -\frac{1}{2} \pm it_j$.

\subsubsection{Meromorphic continuation of $Z(s, 0, f\times \overline{g})$: continuous spectrum}

Examination of the line~\eqref{line:2spectralexp}, the contribution from the continuous
spectrum, is substantially more involved than the discrete spectrum.
It is here where the most remarkable cancellation occurs.
For ease of reference, we repeat this line, the part we call the continuous spectrum:
\begin{equation}\label{eq:line2-repeat}
  \frac{(4\pi)^{k}}{4\pi i}\int_{(0)} G(s, z) \mathcal{Z}(s,w,z) \langle
  \mathcal{V}_{f,\overline{g} },E(\cdot,\tfrac{1}{2}-\overline{z})\rangle \,dz
  \tag{\ref{line:2spectralexp}}
\end{equation}\noeqref{eq:line2-repeat} 
where $G(s, z)$ and $\mathcal{Z} (s,w,z)$ are the collected $\Gamma$ and $\zeta$ factors
\begin{align}
  G(s, z) = \frac{\Gamma(s - \tfrac{1}{2} + z)\Gamma(s - \tfrac{1}{2} -
  z)}{\Gamma(s)\Gamma(s+k-1)},
  \; \mathcal{Z}(s,w,z) = \frac{\zeta(s + w -\frac{1}{2} + z)\zeta(s + w -\frac{1}{2} -
  z)}{\zeta^*(1+2z)}.
\end{align}
The rightmost pole seems to occur from the pair of zeta functions in the numerator,
occurring when $s + w - \frac{1}{2} \pm z = 1$. We must disentangle $s$ and $w$ from $z$
in order to understand these poles.

Line~\eqref{line:2spectralexp} is analytic for $\Re (s+w) > \frac{3}{2}$, $\Re s >
\tfrac{1}{2}$.
As we will shortly set $w = 0$, we treat the boundary $\Re (s+w) > \frac{3}{2}$.
For $s$ with $\Re s \in (\frac{3}{2} - \Re w, \frac{3}{2} - \Re w + \epsilon)$ for some
very small $\epsilon$, we want to shift the contour of integration, avoiding poles coming
from the $\zeta^*(1 -2z)$ appearing in the denominator of the Fourier expansion of
$E(\cdot,\frac{1}{2}+\overline{z})$ (described in~\eqref{eq:sums:E_Fourier}).

We shift the $z$-contour to the right while staying within the zero-free region of
$\zeta(1 - 2z)$.
By an abuse of notation, we denote this shift here by $\Re z = \epsilon$ and let
$\epsilon$ in this context actually refer to the real value of the $z$-contour at the
relevant imaginary value.
This argument can be made completely rigorous, cf.~\cite[p. 481-483]{HoffsteinHulse13}.

We perform this shift in order to guarantee that the two poles in $z$ coming from $\zeta(s
+ w - \tfrac{1}{2} \pm z)$, occurring at $\pm z = \tfrac{3}{2} - s - w$, have different
real parts; simultaneously, we pass the pole with more positive real part, occurring at
$z = s + w - \frac{3}{2}$ from $\zeta(s + w - \tfrac{1}{2} - z)$.
By the residue theorem,
\begin{align}\label{eq:FEfirstpart}
  &\frac{(4\pi)^k}{4\pi i} \int_{(0)}G(s,z) \mathcal{Z}(s,w,z) \langle
  \mathcal{V}_{f,\overline{g} }, E\rangle  \ dz \\
  &= \frac{(4\pi)^k}{4\pi i}\int_{(\epsilon)} G \mathcal{Z} \langle
  \mathcal{V}_{f,\overline{g} }, E \rangle \  dz - \frac{(4\pi)^k}{2}\Res_{z = s + w -
  \frac{3}{2}} G \mathcal{Z} \langle \mathcal{V}_{f,\overline{g} }, E\rangle \nonumber,
\end{align}
where the above residue is found to be
\begin{equation}\label{eq:FEfirstresidue}
  -\frac{\zeta(2s + 2w - 2)\Gamma(2s + w -
  2)\Gamma(1-w)}{\zeta^*(2s+2w-2)\Gamma(s)\Gamma(s + k - 1)}\langle
\mathcal{V}_{f,\overline{g} }, E(\cdot, 2-\overline{s}-\overline{w})\rangle.
\end{equation}
The residue is analytic in $s$ for $\Re s \in (1 - \Re w, \tfrac{3}{2} - \Re w +
\epsilon)$, and has an easily understood meromorphic continuation to the whole plane.
Notice also that the shifted contour integral has no poles in $s$ for $\Re s \in
(\tfrac{3}{2} - \Re w - \epsilon, \tfrac{3}{2} - \Re w + \epsilon)$, so we have found an
analytic (not meromorphic!) continuation in $s$ of Line~\eqref{line:2spectralexp} past the
first apparent pole along $\Re s = \frac{3}{2} - \Re w$.

For $s$ with $\Re s \in (\frac{3}{2} - \Re w - \epsilon, \frac{3}{2} - \Re w)$, we shift
the contour of integration back to $\Re z = 0$.
Since this passes a pole, we pick up a residue.
But notice that this is the residue at the \emph{other} pole, $\tfrac{3}{2} - s - w$,
coming from $\zeta(s + w - \tfrac{1}{2} + z)$,
\begin{align}\label{eq:FEsecondpart}
  &\frac{(4\pi)^k}{4\pi i} \int_{(\epsilon)}G(s, w, z) \mathcal{Z}(s,w,z) \langle
  \mathcal{V}_{f,\overline{g} }, E\rangle dz  \\
  =&\frac{(4\pi)^k}{4\pi i}\int_{(0)} G \mathcal{Z} \langle \mathcal{V}_{f,\overline{g} },
  E \rangle dz + \frac{(4\pi)^k}{2}\Res_{z = \frac{3}{2} - s - w} G \mathcal{Z}  \langle
  \mathcal{V}_{f,\overline{g} }, E\rangle \nonumber.
\end{align}
By using the functional equations of the Eisenstein series and zeta functions, one can
check that
\begin{equation}
  \Res_{z = \frac{3}{2} - s - w} G \mathcal{Z} \langle \mathcal{V}_{f,\overline{g} },
  E\rangle = - \Res_{z = s + w -  \frac{3}{2}} G \mathcal{Z}
  \langle\mathcal{V}_{f,\overline{g} }, E\rangle,
\end{equation}
so the two residues combine together, and have well-understood meromorphic continuations
to the whole plane.
The shifted integral has clear meromorphic continuation until the next apparent poles at
$\Re s = \tfrac{1}{2}$ coming from the Gamma functions $\Gamma(s - \tfrac{1}{2} \pm z)$ in
the integrand.
Thus~\eqref{line:2spectralexp}, originally defined for $\Re s > \frac{3}{2} - \Re w$, has
meromorphic continuation for $\frac{1}{2} < \Re s < \frac{3}{2} - \Re w$ given by
\begin{equation}
  \frac{(4\pi)^k}{4\pi i}\int_{(0)}G \mathcal{Z} \langle \mathcal{V}_{f,\overline{g} },
  E\rangle dz + (4\pi)^k \Res_{z = \frac{3}{2} - s - w} G \mathcal{Z} \langle
\mathcal{V}_{f,\overline{g} }, E\rangle.
\end{equation}

A very similar argument works to extend the meromorphic continuation in $s$ of the contour
integral past the next apparent poles at $\Re s  = \frac{1}{2}$ from the Gamma functions,
leading to a meromorphic continuation in the region $-\frac{1}{2}< \Re s < \frac{1}{2}$
given by
\begin{align}\label{eq:fullcontinuation}
  &\frac{(4\pi)^k}{4\pi i}\int_{(0)}G(s, w, z)\mathcal{Z}(s,w,z)\langle
\mathcal{V}_{f,\overline{g} }, E(\cdot, \tfrac{1}{2} - \overline{z})\rangle dz \\
  &\quad+ (4\pi)^k \Res_{z = \frac{3}{2} - s - w}  G(s, w, z) \mathcal{Z}(s,w,z) \langle
  \mathcal{V}_{f,\overline{g} }, E(\cdot, \tfrac{1}{2} - \overline{z})\rangle
  \label{line:firstresidual} \\
  &\quad+ (4\pi)^k \Res_{z = \frac{1}{2} - s} G(s, w, z) \mathcal{Z}(s,w,z) \langle
\mathcal{V}_{f,\overline{g} }, E(\cdot, \tfrac{1}{2} -\overline{z})\rangle.
\label{line:secondresidual}
\end{align}

We iterate this argument, as in Section~4 of~\cite[p. 481-483]{HoffsteinHulse13}.
Somewhat more specifically, when $\Re(s)$ approaches a negative half-integer,
$\frac{1}{2}-n$, we can shift the line of integration for $z$ right past the pole due to
$G(s,z)$ at $z=s-\frac{1}{2}+n$, move $s$ left past the line $\Re(s)=\frac{1}{2}-n$ and
then shift the line of integration for $z$ left, back to zero and over the pole at
$z=\frac{1}{2}-s-n$.
This gives meromorphic continuation of~\eqref{line:2spectralexp} arbitrarily far to the
left, accumulating an additional pair of residual terms each time $\Re s$ passes a
half-integer, and of a similar form to the first residual term coming from $G(s,z)$,
appearing in~\eqref{line:secondresidual}.

\index{residual terms}
We now specialize to $w = 0$.
It is advantageous to codify some terminology for these additional terms appearing in the
meromorphic continuation of the integral, i.e.\ terms like~\eqref{line:firstresidual}
and~\eqref{line:secondresidual}.
We call these terms \emph{residual} terms, as they come from residues in the $z$ variable;
these are distinct from residues in $s$, as these residual terms are functions in $s$.
We also introduce a notation for these residual terms.
Substituting $w = 0$ into~\eqref{line:firstresidual}, we get the residual term
\begin{equation}\label{eq:residual_0}
  \rho_{\frac{3}{2}}(s) = \frac{(4\pi)^k \zeta(2s - 2)\Gamma(2s - 2)}{\Gamma(s) \Gamma(s +
  k - 1) \zeta^*(2s - 2)} \langle \mathcal{V}_{f, \overline{g}}, \overline{E(\cdot, 2 -
s)} \rangle.
\end{equation}
This residual term is distinguished as the only residual term appearing as a residue of
the zeta functions.

The remaining residual terms all come from residues of Gamma functions, and all have a
similar form.
For each $m \geq 1$, there is a residual term $\rho_{\frac{3}{2} - m}(s)$ appearing for
$\Re s < \tfrac{3}{2} - m$, appearing from an apparent pole of Gamma functions in $G(s,z)$
at $z = \tfrac{3}{2} - m - s$, given by
\begin{equation}\label{eq:residual_m}
  \begin{split}
    \rho_{\frac{3}{2} - m}(s) = &\frac{(-1)^{m-1}(4\pi)^k \zeta(1 - m) \zeta(2s + m - 2)
    \Gamma(2s + m - 2)}{\Gamma(m) \Gamma(s) \Gamma(s + k - 1) \zeta^*(4 - 2s - 2m)} \times
    \\
    &\quad\langle \mathcal{V}_{f, \overline{g}}, \overline{E(\cdot, s + m - 1)}\rangle.
  \end{split}
\end{equation}

We summarize these computations with the following proposition.
\begin{proposition}
  The continuous spectrum component of $Z(s, 0, f\times \overline{g})$, as given
  by~\eqref{line:2spectralexp} in Proposition~\ref{prop:spectralexpansionfull}, has
  meromorphic continuation to the complex plane.
  Further, the meromorphic continuation can be written explicitly as
  \begin{equation}
    \frac{(4\pi)^k}{4\pi i} \int_{(0)} G(s,z) \mathcal{Z}(s,0,z) \langle \mathcal{V}_{f,
    \overline{g}}, \overline{E(\cdot, \tfrac{1}{2} - z)}\rangle dz + \sum_{\mathclap{0
    \leq m < \frac{3}{2} - \Re s}} \rho_{\frac{3}{2} - m}(s),
  \end{equation}
  where each residual term $\rho_{\frac{3}{2} - m}(s)$ is given by~\eqref{eq:residual_0}
  (in the case that $m = 0$) or~\eqref{eq:residual_m} (when $m \geq 1$), and
  $\rho_{\frac{3}{2} - m}$ appears only when $\Re s < \tfrac{3}{2} - m$.
\end{proposition}

\subsection{Polar Analysis of $Z(s, 0, f\times \overline{g})$}

Comparing the meromorphic continuations of the discrete spectrum component and continuous
spectrum component of $Z(s, 0, f\times \overline{g})$ reveals that the rightmost pole of
$Z(s, 0, f\times \overline{g})$ occurs in the first residual term,
$\rho_{\frac{3}{2}}(s)$, appearing in~\eqref{eq:residual_0}.
The pole occurs at $s = 1$, caused by the Eisenstein series.
By comparison, the discrete spectrum component is analytic for $\Re s > -\tfrac{1}{2} \pm
it_j$, and the rest of the continuous spectrum is analytic for $\Re s > \tfrac{1}{2}$.

The residue at this pole is given by
\begin{equation}\label{eq:Zresidue}
  \Res_{s = 1} \rho_{\frac{3}{2}}(s) = \Res_{s = 1} \frac{(4\pi)^k \zeta(2s - 2)\Gamma(2s
  - 2)}{\zeta^*(2s-2)\Gamma(s)\Gamma(s + k - 1)}\langle \mathcal{V}_{f,\overline{g} },
E(\cdot, 2-\overline{s})\rangle.
\end{equation}
Expand $\zeta^*(2s-2) = \zeta(2s-2) \Gamma(s-1) \pi^{1-s}$ in the denominator, cancel the
two zeta functions, and use the Gamma duplication identity
\begin{equation}
  \frac{\Gamma(2z)}{\Gamma(z)} = \frac{\Gamma(z + \tfrac{1}{2}) 2^{2z - 1}}{\sqrt \pi}
\end{equation}
with $z = s-1$ to rewrite the residue as
\begin{align}
  \Res_{s = 1} \rho_{\frac{3}{2}}(s) &= \Res_{s = 1} \frac{\Gamma(s - \tfrac{1}{2})
  2^{2s-3}}{\sqrt \pi} \frac{\pi^{s-1}}{\Gamma(s)} \frac{(4\pi)^k}{\Gamma(s + k - 1)}
  \langle \mathcal{V}_{f, \overline{g}}, E(\cdot, 2 - \overline{s}) \rangle \\
  &= -\frac{(4\pi)^k}{\Gamma(k)} \Res_{s = 1} \langle f \overline{g} \Im(\cdot)^k,
E(\cdot, \overline{s})\rangle.
\end{align}
Through the relationship between the Eisenstein series and the
Rankin--Selberg $L$-function (cf.~\S\ref{ssec:rankinselberg_lfunction}), 
we can rewrite this as
\begin{equation}\label{eq:first_residual_pole_1}
  \Res_{s = 1} \rho_{\frac{3}{2}}(s) = -\Res_{s = 1} \frac{L(s, f\times
  \overline{g})}{\zeta(2)}.
\end{equation}

The next pole of $Z(s, 0, f\times \overline{g})$ also comes from the first residual term
$\rho_{\frac{3}{2}}(s)$, occurring at $s = \tfrac{1}{2}$ from the Gamma function in the
numerator.
Similar to the computation of the residue at $s = 1$, we expand $\zeta^*$, cancel the two
zeta functions, and apply the Gamma duplication identity to recognize the residue as
\begin{align}
  \Res_{s = \frac{1}{2}} \rho_{\frac{3}{2}}(s) &= \Res_{s = \frac{1}{2}} \frac{(4\pi)^{s +
  k - 1}}{2\sqrt \pi} \frac{\Gamma(s - \tfrac{1}{2})}{\Gamma(s) \Gamma(s + k - 1)} \langle
  \mathcal{V}_{f, \overline{g}}, E(\cdot, 2 - \overline{s}) \rangle \\
  &= \frac{1}{2 \pi }\frac{(4\pi)^{k - \frac{1}{2}}}{\Gamma(k - \frac{1}{2})}\langle
\mathcal{V}_{f,\overline{g} }, E(\cdot, \tfrac{3}{2}) \rangle.
\end{align}
We rewrite this as a special value of the Rankin--Selberg $L$-function, 
\begin{equation}\label{eq:first_residual_pole_half}
  \Res_{s = \frac{1}{2}} \rho_{\frac{3}{2}}(s) = \frac{1}{2 \pi}\frac{(k -
  \frac{1}{2})}{4\pi} \frac{(4\pi)^{k + \frac{1}{2}}}{\Gamma(k + \frac{1}{2})}\langle
\mathcal{V}_{f,\overline{g} }, E(\cdot, \tfrac{3}{2}) \rangle = \frac{(k -
\frac{1}{2})}{4\pi^2} \frac{L(\frac{3}{2}, f\times \overline{g} )}{\zeta(3)}.
\end{equation}

Returning to the rest of the meromorphic continuation of the continuous spectrum, let us
examine the second residual term $\rho_{\frac{3}{2} - 1}(s) = \rho_{\frac{1}{2}}(s)$,
which appears as part of the meromorphic continuation only for $\Re s < \tfrac{1}{2}$.
We simplify the expression in~\eqref{eq:residual_m}, with $m = 1$,
\begin{equation}
  \rho_{\frac{1}{2}}(s) = \frac{(4\pi)^k \zeta(0)}{\Gamma(s + k - 1)} \frac{\zeta(2s
  -1)}{\zeta^*(2 - 2s)} \frac{\Gamma(2s - 1)}{\Gamma(s)} \langle
\mathcal{V}_{f,\overline{g}}, E(\cdot, \overline{s})\rangle.
\end{equation}
By using the functional equation for $\zeta^*(2 - 2s)$, the Gamma duplication formula, and
recognizing the Eisenstein series inner product as a sum of
two Rankin--Selberg $L$-functions, this simplifies to 
\begin{equation}\label{eq:secondresidualsimple}
  \rho_{\frac{1}{2}}(s) = -\frac{L(s, f\times \overline{g})}{\zeta(2s)}.
\end{equation}
We now recognize that $\rho_{\frac{1}{2}}(s)$ has poles at zeroes of $\zeta(2s)$.

\begin{remark}
  At first glance, it would appear that $\rho_{\frac{1}{2}}(s)$ also has a pole at $s=1$,
  coming from the pole of the Rankin--Selberg convolution, 
  and that this term therefore contributes a pole to $Z(s, 0, f\times \overline{g})$ at
  $s=1$.
  However, the term $\rho_{\frac{1}{2}}(s)$ does not appear as part of the meromorphic
  continuation of $Z(s, 0, f\times \overline{g})$ except when $\Re s < \frac{1}{2}$, so
  this term does not contribute a pole at $s = 1$.
\end{remark}

More generally, for each $m \geq 1$, the residual term $\rho_{\frac{3}{2} - m}(s)$, which
appears for $\Re s < \frac{3}{2} - m$, also contributes poles.
As in~\eqref{eq:secondresidualsimple}, the Eisenstein series in $\rho_{\frac{3}{2} -
m}(s)$ introduces poles at $s = \frac{\gamma}{2} - m + 1$ for each nontrivial zero
$\gamma$ of $\zeta(s)$, in addition to potential poles at negative integers and
half-integers from the Gamma function in the numerator.

Poles appearing in the discrete spectrum component do not exhibit the same properties of
cancellation, aside from those noted in Lemmas~\ref{lem:Litj_equals_zero}
and~\ref{lem:oddorthogonaltoeven}.
Analysis of the Gamma functions in the discrete component,
in~\eqref{line:1spectralexp}, shows that there are potential simple poles at $s =
\frac{1}{2} \pm it_j - n$ for $n \in \mathbb{Z}_{\geq 0}$.
The Lemmas~\ref{lem:Litj_equals_zero} and~\ref{lem:oddorthogonaltoeven} show that those
poles occurring at $s = \frac{1}{2} \pm it_j - n$ with $n$ even are cancelled by trivial
zeroes.
Together, these indicate that there are potential poles at $s = \frac{1}{2} \pm it_j - n$
for each odd, positive integer $n$.

\subsection{Analytic Behavior of $W(s; f, \overline{g})$}
\index{Ws@$W(s; f, f)$}

Recall that
\begin{equation}
  W(s; f, \overline{g}) = \frac{L(s, f\times \overline{g})}{\zeta(2s)} + Z(s, 0, f\times \overline{g}).
\end{equation}
As noted in~\eqref{eq:first_residual_pole_1}, the leading pole of $L(s, f\times
\overline{g}) \zeta(2s)^{-1}$ perfectly cancels with the leading pole of $Z(s, 0, f\times
\overline{g})$.
Therefore $W(s, f, \overline{g})$ is analytic for $\Re s > \tfrac{1}{2}$ and has a pole at
$s = \tfrac{1}{2}$, identified in~\eqref{eq:first_residual_pole_half}.

Further, the second residual term, $\rho_{\frac{1}{2}}(s)$, was shown to be exactly $-L(s,
f\times \overline{g}) \zeta(2s)^{-1}$ in~\eqref{eq:secondresidualsimple}, and appears for
$\Re s < \frac{1}{2}$.
Therefore the Rankin--Selberg $L$-function 
$L(s, f\times \overline{g})\zeta(2s)^{-1}$ perfectly cancels with $\rho_{\frac{1}{2}}(s)$
for $\Re s < \frac{1}{2}$.
For $\Re s < \frac{1}{2}$, the analytic behavior of $W(s; f, \overline{g})$ is determined
entirely by the analytic behavior of $Z(s, 0, f\times \overline{g})$ (and omitting
$\rho_{\frac{1}{2}}(s)$).

Therefore $W(s; f, \overline{g})$ has meromorphic continuation to $\mathbb{C}$.
For $\Re s > - \tfrac{1}{2}$, the only possible poles of $W(s; f, \overline{g})$ are at $s
= \tfrac{1}{2}$, coming from~\eqref{eq:first_residual_pole_half}, and those at $s = -
\tfrac{1}{2} \pm it_j$ coming from exceptional eigenvalues of the discrete spectrum.
(There are no exceptional eigenvalues on $\SL(2, \mathbb{Z})$).
Collecting the analytic data, we have proved the following.

\begin{theorem}\label{thm:Wsfgmero}
  Let $f,g$ be two holomorphic cusp forms on $\SL(2, \mathbb{Z})$.
  Maintaining the same notation as above, the function $W(s; f,\overline{g})$ has a
  meromorphic continuation to $\mathbb{C}$ given by the Rankin--Selberg 
  $L$-function~\eqref{eq:Lsfgbar_equals_eisenstein}
  and spectral decomposition in Proposition~\ref{prop:spectralexpansionfull},
  with potential poles at $s$ with $\Re s \leq \tfrac{1}{2}$ and $s \in
  \mathbb{Z}\cup(\mathbb{Z} + \tfrac{1}{2})\cup\mathfrak{S}\cup\mathfrak{Z}$, where
  $\mathfrak{Z}$ denotes the set of shifted zeta-zeroes $\{-1 + \frac{\gamma}{2} - n: n
  \in \mathbb{Z}_{\geq 0}\}$, and $\mathfrak{S}$ denotes the set of shifted discrete types
  $\{-\tfrac{1}{2} \pm it_j - n: n \in \mathbb{Z}_{\geq 0}, n \; \text{odd}\; \}$.

  The leading pole is at $s = \frac{1}{2}$ and
    \begin{equation}
      \Res_{s = \frac{1}{2}} W(s; f, \overline{g}) = \frac{(k - \frac{1}{2})}{4\pi^2}
      \frac{L(\tfrac{3}{2}, f\times \overline{g})}{\zeta(3)}.
    \end{equation}
\end{theorem}

\subsection{Complete Meromorphic Continuation of $D(s, S_f\times \overline{S_{g}})$}

With Theorem~\ref{thm:Wsfgmero} and the decomposition from
Proposition~\ref{prop:SfSg_decomposition}, we can quickly give the meromorphic
continuation of the Dirichlet series $D(s, S_f \times \overline{S_g})$.
In particular, by Proposition~\ref{prop:SfSg_decomposition}, we know that
\begin{equation}
  D(s, S_f \times \overline{S_g}) = W(s; f, \overline{g}) + \frac{1}{2\pi i}
  \int_{(\gamma)} W(s-z; f, \overline{g}) \zeta(z)\frac{\Gamma(z) \Gamma(s - z + k -
  1)}{\Gamma(s + k - 1)} \; dz,
\end{equation}
where initially $\Re s$ is large and $\gamma \in (1, \Re(s) - 1)$.
Notice that $W(s; f, \overline{g})$ can also be written as a single Dirichlet series as
\begin{align}
  W(s; f, \overline{g}) &= \frac{L(s, f\times \overline{g})}{\zeta(2s)} + Z(s, 0, f\times
  \overline{g}) \\
  &= \sum_{n,h \geq 1} \frac{a(n) \overline{b(n)} + a(n)\overline{b(n-h)} +
  a(n-h)\overline{b(n)}}{n^{s + k - 1}} \\
  &= \sum_{\substack{n \geq 1 \\ h \geq 0}} \frac{a(n)\overline{b(n-h)} +
  a(n-h)\overline{b(n)} - a(n)\overline{b(n)}}{n^{s + k - 1}} \\
  &= \sum_{n \geq 1} \frac{a(n) \overline{S_g(n)} + S_f(n) \overline{b(n)}}{n^{s + k - 1}}
  =: \sum_{n \geq 1} \frac{w(n)}{n^{s + k - 1}}.
\end{align}
We denote the $n$th coefficient of this Dirichlet series of $w(n)$.
As $a(n) \ll n^{\frac{k-1}{2} + \epsilon}$ and $S_f(n) \ll n^{\frac{k-1}{2} +
\frac{1}{3}}$, we know that $w(n) \ll n^{k-1 + \frac{1}{3} + \epsilon}$.
Thus $W(s; f, \overline{g})$ converges absolutely for $\Re s > \frac{4}{3}$.

Consider $D(s, S_f \times \overline{S_g})$ for $\Re s > 4$ and $\gamma = 2$ initially, so
that both $W(s; f, \overline{g})$ and $W(s-z; f, \overline{g})$ are absolutely convergent.
Shifting the line of $z$ integration to $-M - \frac{1}{2}$ for some positive integer $M$
passes several poles occurring when $z = 1$ (from $\zeta(z)$) or $z = -j$ with $j \in
\mathbb{Z}_{\geq 0}$ (from $\Gamma(z)$).
Notice that in this region, $W(s-z; f, \overline{g})$ converges absolutely, $\zeta(z)$ has
at most polynomial growth in vertical strips, and the Gamma functions have exponential
decay for any fixed $s$.
By Cauchy's Theorem, we have
\begin{align}
  &D(s, S_f \times \overline{S_g}) =\\
  &= W(s; f, \overline{g}) + \sum_{-M \leq j \leq 1} \Res_{z = j} W(s-z; f, \overline{g})
  \zeta(z)\frac{\Gamma(z) \Gamma(s - z + k - 1)}{\Gamma(s + k - 1)} \\
  &\quad + \frac{1}{2\pi i} \int_{(-M - \frac{1}{2})} W(s-z; f, \overline{g})
  \zeta(z)\frac{\Gamma(z) \Gamma(s - z + k - 1)}{\Gamma(s + k - 1)} \; dz \\
  &= W(s; f, \overline{g}) + \frac{W(s-1; f, \overline{g})}{s+k-2} + \sum_{j = 0}^{M}
  \frac{(-1)^j}{j!} W(s + j; f, \overline{g}) \zeta(-j) \frac{\Gamma(s + j + k -
  1)}{\Gamma(s + k - 1)} \\
  &\quad + \frac{1}{2\pi i} \int_{(-M - \frac{1}{2})} W(s-z; f, \overline{g})
  \zeta(z)\frac{\Gamma(z) \Gamma(s - z + k - 1)}{\Gamma(s + k - 1)} \; dz.
\end{align}
Each of the residues gives an expression containing $W(s; f, \overline{g})$ with clear
meromorphic continuation to the plane.
The remaining shifted integral contains $W(s-z; f, \overline{g})$ in its integrand, with
$\Re z = -M-\tfrac{1}{2}$.
Therefore $\Re s-z = \Re s + M + \tfrac{1}{2}$, and so $W(s-z; f, \overline{g})$ is
absolutely convergent for $\Re s > -M + \tfrac{5}{2}$.
As $\zeta(z)$ has only polynomial growth and $\Gamma(z)\Gamma(s-z+k-1)$ has exponential
decay in vertical strips, we see that the integral represents an analytic function of $s$
for $\Re s > -M + \tfrac{5}{2}$.
Therefore the entire right hand side has meromorphic continuation to the region $\Re s >
-M + \tfrac{5}{2}$.
As $M$ is arbitrary, we have proved the following, which we record as a corollary to
Theorem~\ref{thm:Wsfgmero}.

\begin{corollary}\label{cor:DsSfSg_has_meromorphic}
  The Dirichlet series $D(s, S_f \times \overline{S_g})$ has meromorphic continuation to
  the entire complex plane.
\end{corollary}

\begin{remark}
  Very similar work gives the meromorphic continuation for $D(s, S_f \times
  \overline{S_g})$, mainly replacing $\overline{g}$ with $T_{-1}g$ in the above
  formulation.
  This distinction only matters at higher levels when $f$ and $g$ have nontrivial
  nebentypus, and the spectral expansion is modified accordingly.
\end{remark}

Analysis of the exact nature of the poles of $D(s, S_f \times \overline{S_g})$ can be
performed directly on this presentation of the meromorphic continuation.
In many cases, the leading behavior of integral transforms on $D(s, S_f \times
\overline{S_g})$ will come from $W(s-1; f, \overline{g})/(s+k-2)$, as this term contains
the largest negative shift in $s$.
For later reference, it will be useful to view $D(s, S_f \times \overline{S_g})$ in this
form for clear arithmetic application.
We codify this in the following lemma.
\begin{lemma}\label{lem:D_is_lots_of_W}
  \begin{equation}
    \begin{split}
    D(s, S_f \times \overline{S_g}) &= W(s; f, \overline{g})
      + \frac{W(s-1; f, \overline{g})}{s+k-2} \\
    &\quad + \sum_{j = 0}^M \frac{(-1)^j}{j!} W(s+j; f, \overline{g}) \zeta(-j)
      \frac{\Gamma(s + j + k - 1)}{\Gamma(s + k - 1)} \\
    &\quad + \frac{1}{2\pi i} \int_{(-M - \frac{1}{2})} W(s-z; f, \overline{g})
      \zeta(z)\frac{\Gamma(z) \Gamma(s - z + k - 1)}{\Gamma(s + k - 1)} \; dz.
    \end{split}
  \end{equation}
\end{lemma}

\section{Second-Moment Analysis}\label{sec:secondmoment}

It is now necessary to estimate the growth of $D(s, S_f \times \overline{S_g})$ and to use
the analytic properties described above to study the sizes of sums of coefficients of cusp
forms.
It will be necessary to understand the size of growth of $D(s, S_f \times
\overline{S_g})$, but it is relatively straightforward to see that $D(s, S_f \times
\overline{S_g})$ has polynomial growth in vertical strips.

\begin{lemma}\label{lem:DsSfSg_poly_growth}
  For $\sigma < \Re s < \sigma'$ and $s$ uniformly away from poles, there exists some $A$
  such that
  \begin{equation}
    D(s, S_f \times \overline{S_g}) \ll \lvert \Im s \rvert^A.
  \end{equation}
  Therefore $D(s, S_f \times \overline{S_g})$ is of polynomial growth in vertical strips.
\end{lemma}

\begin{proof}
  From Lemma~\ref{lem:D_is_lots_of_W} it is only necessary to study the growth properties
  of $W(s; f, \overline{g})$ and the Mellin-Barnes integral transform of $W(s; f,
  \overline{g})$.

  We first handle $W(s; f, \overline{g})$.
  The diagonal component of $W(s; f, \overline{g})$ is just the Rankin--Selberg 
  $L$-function $L(s, f\times \overline{g}) \zeta(2s)^{-1}$, which has polynomial growth in
  vertical strips as a consequence of the Phragm\'{e}n-Lindel\"{o}f convexity principle
  and the functional equation.

  As noted in Remark~\ref{rem:extraremark}, the discrete spectrum and integral term in the
  continuous spectrum each have polynomial growth in vertical strips.
  It remains to consider the possible contribution from the residual terms
  $\rho_{\frac{3}{2}}(s)$ and $\rho_{\frac{3}{2} - m}(s)$.
  These each consist of a product of zeta functions, Gamma functions,
  and Rankin--Selberg 
  $L$-functions, and a quick analysis through Stirling's approximation shows that the
  exponential contributions from the Gamma functions all perfectly cancel.
  Therefore these are also of polynomial growth.

  We now handle the Mellin-Barnes transform of $W(s; f, \overline{g})$.
  We actually prove a slightly more general result.

  Let $F(s)$ be a function of polynomial growth in $\lvert \Im s \rvert$ in vertical
  strips containing $\sigma$.
  Then the function
  \begin{equation}
    \frac{1}{2\pi i} \int_{(\sigma)}F(s-z)\zeta(z) \frac{\Gamma(z) \Gamma(s -
    z)}{\Gamma(s)} dz
  \end{equation}
  has at most polynomial growth in $\lvert \Im s \rvert$.
  Indeed, through Stirling's Approximation, the integrand is bounded by
  \begin{equation}
    \lvert \Im s \rvert^A \lvert \Im (s-z) \rvert^B \lvert \Im z \rvert^C
    \exp\bigg(-\frac{\pi}{2}\Big(\lvert \Im z \rvert + \lvert \Im(s-z) \rvert - \lvert \Im
    s \rvert \Big)\bigg).
  \end{equation}
  Therefore, for $\lvert \Im z \rvert > \lvert \Im s \rvert$, the integrand has
  exponential decay and converges rapidly.
  Thus the integral is essentially of an integrand of polynomial growth along an interval
  of length $2 \lvert \Im s \rvert$, leading to an overall polynomial bound in $\lvert \Im
  s \rvert$.
\end{proof}

This is already sufficient for many applications.
Consider the integral transform
\begin{equation}\label{eq:smooth_cutoff_ref}
  \frac{1}{2\pi i} \int_{(\sigma)} D(s, S_f \times \overline{S_g}) X^s \Gamma(s) ds =
  \sum_{n \geq 1} \frac{S_f(n) \overline{S_g(n)}}{n^{k-1}} e^{-n/X},
\end{equation}
as described in \S\ref{sec:cutoff_integrals}.
Initially take $\sigma$ large enough to be in the domain of absolute convergence of $D(s,
S_f \times \overline{S_g})$, say $\sigma \geq 4$.

Through Lemma~\ref{lem:D_is_lots_of_W}, we rewrite~\eqref{eq:smooth_cutoff_ref} as
\begin{equation}\label{eq:smooth_cutoff_in_W}
  \begin{split}
    &\frac{1}{2\pi i} \int_{(4)} \bigg(\frac{1}{2} W(s; f, \overline{g}) + \frac{W(s-1; f,
\overline{g})}{s + k - 2} \bigg) X^s \Gamma(s) ds \\
    &\quad+ \frac{1}{(2\pi i)^2} \int_{(4)} \int_{(-1 + \epsilon)} W(s-z; f, \overline{g})
\zeta(z) \frac{\Gamma(z) \Gamma(s - z + k - 1)}{\Gamma(s + k - 1)}dz X^s \Gamma(s) ds.
  \end{split}
\end{equation}
From the proof and statement of Theorem~\ref{thm:Wsfgmero}, we see that $W(s; f,
\overline{g})$ is analytic in $\Re s > -\frac{1}{2}$ except for poles at $s = \frac{1}{2}$
and at $s = -\frac{1}{2} \pm it_j$.
Therefore when we shift lines of $s$-integration in~\eqref{eq:smooth_cutoff_in_W} to
$\frac{1}{2} + 2\epsilon$ passes a pole at $s = \frac{3}{2}$ from $W(s-1; f,
\overline{g})$, and otherwise no poles.

\begin{remark}
  For general level, we shift lines of $s$-integration to $\frac{1}{2} + \theta +
  2\epsilon$, where $\theta < \frac{7}{64}$ is the best-known progress towards Selberg's
  Eigenvalue Conjecture, as noted above.
\end{remark}

By Lemma~\ref{lem:DsSfSg_poly_growth}, this shift is justified and the resulting integral
converges absolutely.
Therefore
\begin{equation}
  \sum_{n \geq 1} \frac{S_f(n)\overline{S_g(n)}}{n^{k-1}} e^{-n/X} = C X^{\frac{3}{2}} +
  O_{f,g,\epsilon}(X^{\frac{-1}{2}+\theta+\epsilon})
\end{equation}
We can evaluate the residue $C$ using Theorem~\ref{thm:Wsfgmero}.
Note that the same analysis holds on $D(s, S_f \times S_g)$ as well.
In total, we have proved the following theorem.

\begin{theorem}\label{thm:second_moment_Sf}
  Suppose $f$ and $g$ are weight $k$ holomorphic cusp forms on $\SL(2, \mathbb{Z})$.
  For any $\epsilon > 0$,
  \begin{align}
    \sum_{n \geq 1} \frac{S_f(n) \overline{S_g(n)}}{n^{k-1}} e^{-n/X} &= C X^{\frac{3}{2}}
    + O_{f,g,\epsilon} (X^{\frac{1}{2} + \epsilon}) \\
    \sum_{n \geq 1} \frac{S_f(n) S_g(n)}{n^{k-1}} e^{-n/X} &= C' X^{\frac{3}{2}} +
    O_{f,g,\epsilon} (X^{\frac{1}{2} + \epsilon})
  \end{align}
  where
  \begin{equation}
    C = \frac{\Gamma(\frac{3}{2})}{4\pi^2} \frac{L(\frac{3}{2}, f \times
    \overline{g})}{\zeta(3)},
    \qquad C' = \frac{\Gamma(\frac{3}{2})}{4\pi^2} \frac{L(\frac{3}{2}, f \times
    g)}{\zeta(3)}.
  \end{equation}
\end{theorem}

As an immediate corollary, we have the following smoothed analogue of the Classical
Conjecture.
\begin{corollary}
  \begin{equation}
    \sum_{n \geq 1} \frac{\lvert S_f(n) \rvert^2}{n^{k-1}} e^{-n/X} = C X^{\frac{3}{2}} +
    O_{f,\epsilon} (X^{\frac{1}{2} + \epsilon}),
  \end{equation}
  where $C$ is the special value of $L(\frac{3}{2}, f \times \overline{f})
  \Gamma(\frac{3}{2}) (\zeta(3) 4\pi^2)^{-1}$ as above.
\end{corollary}

\section{A General Cancellation Principle}\label{sec:higherlevel}

While the techniques and methodology employed so far should work for general weight and
level, it is not immediately obvious that that the miraculous cancellation that occurs in
the level $1$ case should always occur.
In particular, it is not clear that the continuous spectrum of $Z(s, 0,f \times
\overline{g})$ will always perfectly cancel both the leading pole and potentially
infinitely many poles from the zeta zeroes of $L(s, f\times g)\zeta(2s)^{-1}$.

In the case when $f=g$ are the same cusp form, we can compare our methodology with the
results of Chandrasekharan and Narasimhan to show that the leading polar cancellation does
always occur.
A more detailed analysis using the same methodology as the rest of this chapter would
likely be able to show general cancellation.

\begin{remark}
  In Section~6 of the soon-to-be-published paper~\cite{hkldw}, my collaborators and I
  explicitly show that this cancellation continues to hold for cusp forms $f$ and $g$ on
  $\Gamma_0(N)$ when $N$ is square-free.
  This is much stronger than what is showed in the rest of this section concerning general
  cancellation between the diagonal and off-diagonal sums corresponding to $f \times
  \overline{f}$.
\end{remark}

Suppose $f(z) = \sum a(n)e(nz)$ is a cusp form on $\Gamma_0(N)$ and of weight $k \in
\mathbb{Z}\cup(\mathbb{Z} + \frac{1}{2})$ with $k > 2$.
Theorem~1 of~\cite{chandrasekharan1964mean} gives that
\begin{equation}\label{eq:CN_compare}
  \frac{1}{X} \sum_{n \leq X} \frac{\lvert S_f(n) \rvert^2}{n^{k-1}}  = C X^{\frac{1}{2}}
  + O(\log^2 X).
\end{equation}

Performing the decomposition from Proposition~\ref{prop:SfSg_decomposition} leads us to
again study $Z(s, 0, f\times \overline{f})$ and $W(s; f,\overline{f})$.
The Rankin--Selberg convolution 
$L(s, f\times f)/\zeta(2s)$ has a pole at $s = 1$.
This pole must cancel with poles from $Z(s, 0, f\times \overline{f})$, as otherwise the
methodology of this chapter contradicts~\eqref{eq:CN_compare}.
Stated differently, we must have that the leading contribution of the diagonal term
cancels perfectly with a leading contribution from the off-diagonal,
\begin{equation*}
  \Res_{s = 1}\sum_{n \geq 1}\frac{\lvert a(n) \rvert^2}{n^{s + k - 1}} = - \Res_{s =
  1}\sum_{n,h \geq 1} \frac{a(n) \overline{a(n-h)} + \overline{a(n)}a(n-h)}{n^{s + k -
  1}}.
\end{equation*}
We investigate this cancellation further by sketching the arguments of
\S\ref{sec:analyticbehavior} and \S\ref{sec:secondmoment} in greater generality.

The spectral decomposition corresponding to Proposition~\ref{prop:spectralexpansionfull}
is more complicated since we must now use the Selberg Poincar\'{e} series on $\Gamma_0(N)$
\begin{equation}
  P_h(z,s) := \sum_{\gamma \in \Gamma_\infty \backslash \Gamma_0(N)} \Im(\gamma z)^s
  e(h\gamma \cdot z).
\end{equation}
The spectral decomposition of $P_h$ will involve Eisenstein series associated to each cusp
$\mathfrak{a}$ of $\Gamma_0(N)$.
These Eisenstein series have expansions
\begin{equation}
  E_\mathfrak{a} (z,w) = \delta_\mathfrak{a} y^w + \varphi_\mathfrak{a}(0,w) y^{1-w} +
  \sum_{m \neq 0} \varphi_\mathfrak{a} (m,w) W_w(\lvert m\rvert z),
\end{equation}
where $\delta_\mathfrak{a} = 1$ if $\mathfrak{a} = \infty$ and is $0$ otherwise,
\begin{align}
  \varphi(0, w) &= \sqrt \pi \frac{\Gamma(w - \frac{1}{2})}{\Gamma(w)} \sum_c
  c^{-2w}S_\mathfrak{a}(0,0;c) \\
  \varphi(m, w) &= \frac{\pi^w}{\Gamma(w)} \lvert m \rvert^{w-1} \sum_c c^{-2w}
  S_\mathfrak{a}(0, m; c)
\end{align}
are generalized Whittaker-Fourier coefficients,
\begin{equation}
  W_w(z) = 2\sqrt y K_{w - \frac{1}{2}}(2\pi y) e(x)
\end{equation}
is a Whittaker function, $K_\nu(z)$ is a $K$-Bessel function, and
\begin{equation}
  S_\mathfrak{a}(m,n; c) = \sum_{\left(\begin{smallmatrix} a&\cdot \\ c&d
  \end{smallmatrix}\right) \in \Gamma_\infty \backslash \sigma_\alpha^{-1} \Gamma_0(N) /
  \Gamma_\infty} e\left( m \frac{d}{c} + n \frac{a}{c}\right)
\end{equation}
is a Kloosterman sum associated to double cosets of $\Gamma_0(N)$ with
\begin{equation}
  \Gamma_\infty = \left \langle \begin{pmatrix} 1&n \\ &1 \end{pmatrix} : n \in \mathbb{Z}
\right\rangle \subset \SL_2(\mathbb{Z}).
\end{equation}
This expansion is given in Theorem 3.4 of~\cite{iwaniec2002spectral}.

Letting $\mu_j$ be an orthonormal basis of the residual and cuspidal spaces, we may expand
$P_h(z,s)$ by the Spectral Theorem (as presented in Theorem~15.5
of~\cite{IwaniecKowalski04}) to get
\begin{align}
  P_h(z,s) &= \sum_j \langle P_h(\cdot, s), \mu_j \rangle \mu_j(z) \\
           &\quad + \sum_\mathfrak{a} \frac{1}{4\pi} \int_\mathbb{R} \langle P_h(\cdot,
           s), E_\mathfrak{a}(\cdot, \tfrac{1}{2} + it)\rangle E_\mathfrak{a}(z,
           \tfrac{1}{2} + it) \ dt. \label{line:PoincareLevelNContinuous}
\end{align}
This is more complicated than the $\SL_2(\mathbb{Z})$ spectral expansion
in~\eqref{eq:Pspectral} for two major reasons: we are summing over cusps and the
Kloosterman sums within the Eisenstein series are trickier to handle.
Continuing as before, we try to understand the shifted convolution sum
\begin{equation}
  Z(s,w,f\times f) = \frac{(4\pi)^{s + k - 1}}{\Gamma(s + k - 1)}\sum_{h \geq
  1}\frac{\langle \lvert f \rvert^2 \Im(\cdot)^k, P_h\rangle}{h^w}
\end{equation}
by substituting the spectral expansion for $P_h(z,s)$ and producing a meromorphic
continuation.

The analysis of the discrete spectrum is almost exactly the same: it is analytic for $\Re
s > -\frac{1}{2} + \theta$.
The only new facet is understanding the continuous spectrum component corresponding
to~\eqref{line:PoincareLevelNContinuous}.
We expect that the continuous spectrum of $Z(s, 0, f\times \overline{f})$ has leading
poles that perfectly cancel the leading pole of $L(s, f\times \overline{f})
\zeta(2s)^{-1}$.

Using analogous methods to those in Section~\ref{sec:analyticbehavior}, we compute the
continuous spectrum of $Z(s, 0, f\times \overline{f})$ to get
\begin{align}
  &\sum_{h \geq 1}\frac{(4\pi)^{s + k - 1}}{\Gamma(s + k - 1)}\sum_{\mathfrak{a}}
  \frac{1}{4\pi i} \int_{(\frac{1}{2})} \langle P_h(\cdot, s), E_\mathfrak{a}(\cdot,
  t)\rangle \langle \lvert f \rvert^2 \Im(\cdot)^k, \overline{E_\mathfrak{a}(\cdot,
  t)}\rangle \, dt \nonumber \\
  \begin{split}
  & = \frac{(4\pi)^k}{\Gamma(s + k - 1)\Gamma(s)}\sum_{\mathfrak{a}} \int_{-\infty}^\infty
    \left( \sum_{h,c \geq 1} \frac{S_\mathfrak{a} (0, h; c)}{h^{s + it} c^{1 -
    2it}}\frac{\pi^{\frac{1}{2} - it}}{\Gamma(\frac{1}{2} - it)}\right) \times \\
  &\quad \times \Gamma(s - \tfrac{1}{2} + it)\Gamma(s - \tfrac{1}{2} - it) \langle \lvert
    f \rvert^2 \Im(\cdot)^k, \overline{E_\mathfrak{a}(\cdot, \tfrac{1}{2} + it)} \rangle \,
    dt.
  \end{split}
\end{align}
We've placed parentheses around the arithmetic part, including the Kloosterman sums and
factors for completing a zeta function that appears within the Kloosterman sums.

The arithmetic part of the Eisenstein series are classically-studied $L$-functions, and
each satisfies analogous analytic properties to the function denoted $\mathcal{Z}(s,0,z)$
in \S\ref{sec:analyticbehavior}.
We summarize the results of this section with the following theorem.

\begin{theorem}\label{thm:general_weight_level_comparison}
  Let $f$ be a weight $k > 2$ cusp form on $\Gamma_0(N)$. Then
  \begin{align}\label{eq:general_level_general_comparison}
    \Res_{s = 1} \sum_{n \geq 1} \frac{\lvert a(n) \rvert^2}{n^{s + k - 1}} = - \Res_{s =
    1} \sum_{n,h \geq 1} \frac{a(n)\overline{a(n-h)}}{n^{s + k - 1}}- \Res_{s = 1}
    \sum_{n,h \geq 1} \frac{\overline{a(n)} a(n-h)}{n^{s + k - 1}},
  \end{align}
  or equivalently
  \begin{align}\label{eq:general_level_kloosterman}
    -\frac{1}{2}\Res_{s = 1} &\frac{L(s, f\times f)}{\zeta(2s)} =\\
    &= \Res_{s = 1}\sum_{\mathfrak{a}} \frac{1}{4\pi} \int_\mathbb{R} \langle P_h(\cdot,
s), E_\mathfrak{a}(\cdot, \tfrac{1}{2} + it)\rangle \langle \lvert f \rvert^2
\Im(\cdot)^k, \overline{E_\mathfrak{a}(\cdot, \tfrac{1}{2} + it)}\rangle \, dt \\
    \begin{split}
      &= \Res_{s = 1} \frac{(4\pi)^k}{\Gamma(s + k - 1)\Gamma(s)}\sum_{\mathfrak{a}}
      \int_{-\infty}^\infty  \left( \sum_{h,c \geq 1} \frac{S_\mathfrak{a} (0, h; c)}{h^{s
      + it} c^{1 - 2it}}\frac{\pi^{\frac{1}{2} - it}}{\Gamma(\frac{1}{2} - it)}\right) \\
      &\quad \times \Gamma(s - \tfrac{1}{2} + it)\Gamma(s - \tfrac{1}{2} - it) \langle
      \lvert f \rvert^2 \Im(\cdot)^k, \overline{E_\mathfrak{a}(\cdot, \tfrac{1}{2} + it)}
      \rangle \, dt.
    \end{split}
  \end{align}
\end{theorem}


\clearpage{\pagestyle{empty}\cleardoublepage}

\chapter{Applications of Dirichlet Series of Sums of Coefficients}\label{c:sums_apps}

In this chapter, we highlight some applications of the Dirichlet series $D(s, S_f \times
\overline{S_g})$ and $D(s, S_f \times S_g)$.
We begin with completed applications, summarizing the results and methodology.
Towards the end of the chapter, we highlight ongoing and future work.

\vspace{-1in}
\subsection*{\center{Reacknowledgements to my Collaborators}}

\begin{quote}
  Each of the applications in this chapter were carried out (or are being carried out)
  with Alex Walker, Chan Ieong Kuan, and Tom Hulse, each of whom are my academic brothers,
  collaborators, and friends.
\end{quote}

\section{Applications}

The Dirichlet series $D(s, S_f \times \overline{S_g})$ and $D(s, S_f \times S_g)$ present
a new avenue for investigating the behavior of $S_f$, $S_f(n)S_g(n)$, and related objects.
As $S_f$ is analogous to the error term in the Gauss Circle Problem (cf.
Chapter~\ref{c:introduction}), it is perhaps most natural to ask about the sizes of
$S_f(n)$.

\subsection*{Long Sums}

One result of this form was presented in \S\ref{c:sums}, in
Theorem~\ref{thm:second_moment_Sf} giving smoothed averages for $S_f(n)\overline{S_g(n)}$
and, if $f = g$, smoothed averages for $\lvert S_f(n) \rvert^2$.
This theorem was proved completely for $f$ and $g$ on level $1$.

In~\cite{hkldw}, my collaborators and I analyze further smoothed asymptotics for
$S_f(n)\overline{S_g(n)}$ for forms $f$ and $g$ on $\Gamma_0(N)$ for $N$ squarefree,
proving analogous results to Theorem~\ref{thm:second_moment_Sf} for forms even of
half-integral weight.
When $f = g$, this is a generalized smoothed analogue to the result of
Cram\'er~\cite{cramer1922} giving that
\begin{equation}
  \frac{1}{X} \int_0^X \Big \lvert \sum_{n \leq t} r_2(n) - \pi t \Big \rvert^2 dt = c
  X^{1/2} + O(X^{\frac{1}{4} + \epsilon}),
\end{equation}
giving evidence towards a generalized Circle problem.

It is very interesting (and very new) that if $f \neq g$, then $S_f S_g$ still appears to
satisfy a generalized Circle problem.
Recalling Chandrasekharan and Narasimhan's result that $S_f(X) =
\Omega_{\pm}(X^{\frac{k-1}{2} + \frac{1}{4}})$, it's apparent that both $S_f(n)$ and
$S_g(n)$ oscillate in size between $\pm n^{\frac{k-1}{2} + \frac{1}{4}}$.
From first principles alone, it seems possible that $S_f(n)$ might be small or negative
with $S_g(n)$ is large and positive, leading to large cancellation in average sums $\sum
S_f(n) S_g(n)$.
But Theorem~\ref{thm:second_moment_Sf} indicates that there is not large cancellation of
this sort.
Indeed, Theorem~\ref{thm:second_moment_Sf} roughly indicates that the partial sums of
Fourier coefficients of $f(z)$ correlate about as well with the partial sums of Fourier
coefficients of $g(z)$ as with itself, up to the constant $L(\frac{3}{2}, f \times g)$ of
proportionality.

By mimicking the techniques of Chapter~\ref{c:sums}, it is possible to apply different
integral transforms to $D(s, S_f \times \overline{S_g})$ or $D(s, S_f \times S_g)$, either
to get different long-average estimates, or estimates of a different variety.

\subsection*{Short-Interval Averages}

Now restrict attention to $f$ a full-integer weight cusp form on $\SL(2, \mathbb{Z})$, and
suppose that $f = g$.
In~\cite{hkldwShort}, my collaborators and I analyze short-interval estimates of the type
\begin{equation}\label{eq:short-interval-estimate}
  \frac{1}{X^{2/3}(\log X)^{1/6}} \sum_{\lvert n-X \rvert < X^{2/3} (\log X)^{1/6}} \lvert
  S_f(n) \rvert^2 \ll X^{k-1 + \frac{1}{2}},
\end{equation}
which says essentially that the Classical Conjecture holds \emph{on average} over short
intervals of width $X^{\frac{2}{3}} (\log X)^{\frac{1}{6}}$ around $X$.
This is qualitatively a  much stronger result than the long-interval estimate, and is a
vast improvement over the previous best result of this type due to
Jarnik~\cite{jutila1987lectures},
\begin{equation}
  \frac{1}{X^{\frac{3}{4} + \epsilon}} \sum_{\lvert n - X \rvert < X^{\frac{3}{4} +
  \epsilon}} \lvert S_f(n) \rvert^2 \ll X^{k - 1 + \frac{1}{2}}.
\end{equation}

It is interesting to note that it is possible to get bounds for individual sums $S_f(n)$
from short-interval estimates.

\begin{proposition}
  Suppose that
  \begin{equation}
    \frac{1}{X^w} \sum_{\lvert n - X \rvert < X^w} \lvert S_f(n) \rvert^2 \ll X^{k - 1 +
    \frac{1}{2}}
  \end{equation}
  for $w \geq \frac{1}{4}$.
  Then also
  \begin{equation}
    S_f(X) \ll X^{\frac{k-1}{2} + \frac{1}{4} + (\frac{w}{3} - \frac{1}{12})}.
  \end{equation}
\end{proposition}

\begin{proof}
  We only sketch the proof.
  For each individual Fourier coefficient $a(n)$, we have Deligne's bound $a(n) \ll
  n^{\frac{k-1}{2}}$.
  Suppose there is an $X$ such that $S_f(X) \gg X^{\frac{k-1}{2} + \frac{1}{4} + \alpha}$.
  Then $S_f(X + \ell) \gg X^{\frac{k-1}{2} + \frac{1}{4} + \alpha}$ for $\ell <
  X^{\frac{1}{4} + \alpha - \epsilon}$ for any $\epsilon > 0$, as it takes approximately
  $X^{\frac{1}{4} + \alpha}$ coefficients $a(n)$ to combine together to cancel $S_f(X)$.
  But then
  \begin{equation}
    \frac{1}{X^w} \sum_{\lvert n - X \rvert < X^w} \lvert S_f(X) \rvert^2 \gg
    \frac{1}{X^w} X^{k-1 + \frac{1}{2} + 2\alpha} X^{\min(w, \frac{1}{4} + \alpha)}.
  \end{equation}
  The $X^{\min(w, \frac{1}{4} + \alpha)}$ term comes from the width of the interval where
  each $S_f(X + \ell)$ is approximately $X^{\frac{k-1}{2} + \frac{1}{4} + \alpha}$.
  Comparing exponents of $X$ with
  \begin{equation}
    \frac{1}{X^w} \sum_{\lvert n - X \rvert < X^w} \lvert S_f(n) \rvert^2 \ll X^{k - 1 +
    \frac{1}{2}}
  \end{equation}
  shows that
  \begin{equation}
    2\alpha + \min(w, \tfrac{1}{4} + \alpha) \leq w,
  \end{equation}
  from which either $\alpha = 0$ or $\alpha < \tfrac{1}{3}( w - \tfrac{1}{4})$.
\end{proof}

\begin{corollary}
  A short-interval estimate of the type~\eqref{eq:short-interval-estimate}, with interval
  $\lvert n-X \rvert \leq X^{\frac{1}{4} + \epsilon}$, would prove the Classical
  Conjecture.
\end{corollary}

The short-interval estimate~\eqref{eq:short-interval-estimate} only produces the
individual estimate $S_f(X) \ll X^{\frac{k-1}{2} + \frac{1}{4} + \frac{5}{36}}$, which is
$1/18$ worse than the current best-known individual bound $S_f(X) \ll X^{\frac{k-1}{2} +
\frac{1}{3}}$.
On the other hand, it is by far the strongest short-interval average estimate.
(Note that the individual bound $S_f(X) \ll X^{\frac{k-1}{2} + \frac{1}{3}}$ doesn't give
a Classical Conjecture on average type result for any interval length).

\begin{remark}
  It is possible to use the short-interval estimate~\eqref{eq:short-interval-estimate},
  along with an argument used in Chapter~\ref{c:hyperboloid} which makes use of multiple
  integral transforms in combination, to recover the Hafner-Ivi\'c tyle bound $S_f(n) \ll
  n^{\frac{k-1}{2} + \frac{1}{3}}$.
\end{remark}

To prove the estimate~\eqref{eq:short-interval-estimate}, one builds upon the meromorphic
information of $D(s, S_f \times \overline{S_f})$ and applies an integral transform
\begin{equation}
  \frac{1}{2\pi i} \int_{(4)} D(s, S_f \times \overline{S_f}) \exp{\Big(\frac{\pi
  s^2}{y^2} \Big)} \frac{X^s}{y} ds
\end{equation}
to understand (essentially) a sum over the interval $\lvert n - X \rvert < X/y$.
For more details and analysis, refer to~\cite{hkldwShort}.

\subsection*{Sign-Changes of Sums of Coefficients of Cusp Forms}

As noted in Chapter~\ref{c:motivations}, the original guiding question that led to the
investigation of $D(s, S_f \times \overline{S_g})$ was concerning the sign changes within
the sequence $\{S_f(n)\}_{n \in \mathbb{N}}$.
In~\cite{hkldwSigns}, my collaborators and I succeeded in answering our original guiding
question.\footnote{The story behind the scope of this paper is a bit interesting, as it
  was written at the same time as~\cite{hkldwShort}.
  Tom Hulse had learned of a set of criteria guaranteeing sign changes from a paper of Ram
  Murty, and he first thought of how to generalize the criteria to sequences of real and
  complex coefficients.
  I had thought it was possible to further generalize towards smoothed sums, but we failed
  in this regard.
  This morphed into our generalization, stating how to translate from analytic properties
  of Dirichlet series directly into sign-change results of the coefficients.
We each focused on the parts that interested us most: I was interested in $S_f(n)$ and
$S_f^\nu(n)$ (defined below), Hulse was interested in real and complex coefficients, and
Ieong Kuan was interested in $\GL(3)$.
In the end, each aspect strengthened the overall paper and led to a nice unified
description of somewhat disjoint sign-change results.}

In this paper, we proved a veritable cornucopia of results concerning the sign changes of
coefficients and sums of coefficients of cusp forms on $\GL(2)$ and $\GL(3)$.
Here, I emphasize what I focused on, and what follows most naturally from the
considerations of Chapter~\ref{c:sums}.

Let $f$ be a weight $0$ Maass form or a holomorphic cusp form of full or half-integer
weight $k$ on a congruence subgroup $\Gamma \subseteq \SL(2, \mathbb{Z})$, possibly with
nontrivial nebentypus.\index{sign change results}
Write $f$ as either
\begin{equation}
  f(z) = \sum_{n \neq 0} A_f(n) \sqrt{y} K_{it_j} (2\pi \lvert y \rvert n) e^{2\pi i n x}
\end{equation}
if $f$ is a Maass form with eigenvalue $\lambda_j = \frac{1}{4} + t_j^2$, or
\begin{equation}
  f(z) = \sum_{n \geq 1} A_f(n)n^{\frac{k-1}{2}} e^{2 \pi i n z}
\end{equation}
if $f$ is a holomorphic cusp form.
We write the coefficients of $f$ as $a_f(n) = A_f(n)n^{\kappa(f)}$ with
\begin{equation}
  \kappa(f) = \begin{cases}
    \frac{k-1}{2} & \text{if } $f$ \; \text{is a holomorphic cusp form}, \\
    0 & \text{if } $f$ \; \text{is a Maass form},
  \end{cases}
\end{equation}
Then $n^{\kappa(f)}$ conjecturally normalizes the coefficients $A_f(n)$ correctly
depending on whether $f$ is a holomorphic cusp form or a Maass form.
It is expected that $A_f(n) \ll n^\epsilon$, but in general it is only known that
\begin{equation}
  A_f(n) \ll n^{\alpha(f) + \epsilon}
\end{equation}
where
\begin{equation}
  \alpha(f) = \begin{cases}
    0 & \text{if f is full-integral weight holomorphic}, \\
    \frac{3}{16} & \text{if f is half-integral weight holomorphic}, \\
    \frac{7}{64} & \text{if f is a Maass form}.
  \end{cases}
\end{equation}
Define the partial sums of normalized coefficients $S_f^\nu(n)$ as
\index{S@$S_f^\nu(n)$}
\begin{equation}
  S_f^\nu(n) = \sum_{m \leq n} \frac{a_f(m)}{m^\nu}
\end{equation}
Then by studying the Dirichlet series $D(s, S_f^\nu, \overline{S_g^\nu})$ and $D(s,
S_f^\nu)$ in~\cite{hkldwSigns}, we proved the following theorem.

\begin{theorem}
  Let $f$ be a weight $0$ Maass form or holomorphic cusp form as described above.
  Suppose that $0 \leq \nu < \kappa(f) + \frac{1}{6} - \frac{2 \alpha(f)}{3}$.
  If there is a coefficient $a_f(n)$ such that $\Re a_f(n) \neq 0$ (resp. $\Im a_f(n) \neq
  0$), then the sequence $\{\Re S_f^\nu(n)\}_{n \in \mathbb{N}}$ (resp. $\{\Im
S^\nu_f(n)\}_{n \in \mathbb{N}}$) has at least one sign change for some $n \in [X, X +
X^{r(\nu)}]$ for $X \gg 1$, where
  \begin{equation}
    r(\nu) = \begin{cases}
      \frac{2}{3} + \frac{2\alpha(f)}{3} + \epsilon & \text{if } \nu < \kappa(f) +
      \frac{\alpha(f)}{3} - \frac{1}{6}, \\
      \frac{2}{3} + \frac{2\alpha(f)}{3} + \Delta + \epsilon & \text{if } \nu = \kappa(f)
      + \frac{\alpha(f)}{3} - \frac{1}{6} + \Delta, 0 \leq \Delta < \frac{1}{3} - \frac{2
      \alpha(f)}{3}.
    \end{cases}
  \end{equation}
\end{theorem}

In other words, we showed high regularity of the sign changes of sums of normalized
coefficients, depending on the amount of normalization.
As should be expected, higher amounts of normalization lead to fewer guaranteed sign
changes.

However, we show that it is possible to take $\nu$ slightly larger than $\kappa(f)$, so
that the individual coefficients $a_f(n)/n^\nu$ are each decaying in size.
For example, for full-integer weight holomorphic cusp forms, we can take
\begin{equation}
  \nu = \frac{k-1}{2} + \frac{1}{6} - \epsilon
\end{equation}
and guarantee at least one sign change in $\{S_f^\nu(n)\}_{n \in \mathbb{N}}$ for some $n$
in $[X, 2X]$ for sufficiently large $X$.
Yet for this normalization, we have
\begin{equation}
  S_f^\nu(n) = \sum_{m \leq n} \frac{a_f(m)}{m^{\frac{k-1}{2} + \frac{1}{6} - \epsilon}},
\end{equation}
so that the coefficients are \emph{decaying} and look approximately like $n^{-1/6}$.
It is a remarkable fact that the coefficients are arranged in such a way that there are
still infinitely many sign regularly-spaced sign changes even though they are
\emph{over-normalized}.

This suggests a certain regularity of the sign changes of individual coefficients
$a_f(n)$, but it is challenging to describe the exact nature of this regularity.

\section{Directions for Further Investigation: Non-Cusp Forms}

In the investigations carried out thus far, we have taken $f$ to be a cusp form.
But one can attempt to perform the same argument on sums of coefficients of noncuspidal
automorphic forms.

One particular example would be to consider sums of the form
\begin{equation}
  S_{\theta^k}(n) = \sum_{m \leq n} r_k(m),
\end{equation}
where $r_k(m)$ is the number of ways of representing $m$ as a sum of $k$ squares.
This is equivalent to the Gauss $k$-dimensional sphere problem, which asks how many
integer lattice points are contained in $B_k(\sqrt n)$, the $k$-dimensional sphere of
radius $\sqrt{n}$ centered at the origin?
A (very good) first approximation is that there are approximately
$\text{Vol} B_k(\sqrt n)$ points within the sphere, so the question is really to
understand the size of the discrepancy
\begin{equation}
  P_k(n) := S_{\theta^k}(n) - \text{Vol} B_k(\sqrt n).
\end{equation}

My collaborators and I have been focusing our attention on this problem.
In the recently submitted paper~\cite{HulseGaussSphere}, we proved that
$D(s, S_{\theta^k} \times S_{\theta^k})$ and $D(s, P_k \times P_k)$ have meromorphic
continuation to the complex plane for $k \geq 3$.
Using these continuations, we were able to prove a smooth estimate of a similar flavor as
in Theorem~\ref{thm:second_moment_Sf}.

\begin{theorem}

  For $k \geq 3$ and any $\epsilon > 0$,
  \begin{equation}
    \begin{split}
      \sum_{n \geq 1} P_k(n)^2 e^{-n/X} &= \delta_{[k = 3]}C' X^2 \log X + C X^{k-1} +
      \delta_{[k = 4]} C'' X^{\frac{5}{2}} + O(X^{k - 2 + \epsilon}),
    \end{split}
  \end{equation}
  where $\delta_{[k = n]}$ is $1$ if $k = n$ and is $0$ otherwise.
  Similarly,
  \begin{equation}
    \int_0^\infty P_k(t)^2 e^{-t/X} dt = \delta_{[k = 3]}D' X^2 \log X + D X^{k-1} +
    \delta_{[k = 4]} D'' X^{\frac{5}{2}} + O(X^{k - 2 + \epsilon}).
  \end{equation}

\end{theorem}

These two statements can be thought of as discrete and continuous Laplace transforms of
the mean square error in the Gauss $k$-dimensional Sphere problem.
We are also able to prove results concerning sharp sums and integrals.

\begin{theorem}

  For each $k \geq 3$, there exists $\lambda > 0$ such that
  \begin{equation}
    \sum_{n \leq X} P_k(n)^2 = \delta_{[k = 3]} C' X^{k-1} \log X + C X^{k-1} +
    O_{\lambda}(X^{k - 1 - \lambda}).
  \end{equation}
  Similarly, we also have
  \begin{equation}
    \int_0^X (P_k(x))^2 dx = \delta_{[k = 3]} D' X^{k-1} \log X + D X^{k-1} +
    O_{\lambda}(X^{k-1 - \lambda}).
  \end{equation}

\end{theorem}

In the dimension $3$ case, this is the first known polynomial savings on the error term,
and represents the first major improvement over a result from Jarnik in
1940~\cite{Jarnik40}, which achieved only $\sqrt{\log X}$ savings.

In~\cite{HulseGaussSphere}, we do not prove what $\lambda$ is.
In forthcoming work, we will consider the size of $\lambda$.
We are also working on extending the techniques and results to the classical Gauss circle
problem, when $k = 2$.

\begin{remark}
  There are limitations to this technique.
  We can only consider forms $f$ for which we understand the shifted convolution sum
  coming from $f \times f$ sufficiently well.
  So we are not capable of understanding sums of coefficients of Maass forms at this time,
  since shifted convolution sums of Maass forms with Maass forms remain mysterious.
\end{remark}


\clearpage{\pagestyle{empty}\cleardoublepage}

\chapter{On Lattice Points on Hyperboloids}\label{c:hyperboloid}

\section{Introduction}\label{sec:hyperboloid_introduction}
\index{hyperboloid, one-sheeted}
\index{H@$\mathcal{H}_{d,h}$}
\index{N@$N_{d,h}(R)$}

A one-sheeted $d$-dimensional hyperboloid $\mathcal{H}_{d,h}$ is a surface satisfying the
equation
\begin{equation}
  X_1^2 + \cdots + X_{d-1}^2 = X_d^2 + h
\end{equation}
for some fixed positive integer $h$.
In this chapter, we investigate the number of integer points lying on the hyperboloid
$\mathcal{H}_{d,h}$.
In particular, we investigate the asymptotics for the number $N_{d,h}(R)$ of integer
points $\bm{m} = (m_1, \ldots, m_d) \in \mathbb{Z}^d$ lying on $\mathcal{H}_{d,h}$ and
within the ball $\| \bm{m} \|^2 \leq R$ for large $R$.
Stated differently, if $B(\sqrt{R})$ is the ball of radius $\sqrt R$ in $\mathbb{R}^d$,
centered at the origin, then
\begin{equation}
  N_{d,h}(R) = \# \big(\mathbb{Z}^d \cap \mathcal{H}_{d,h} \cap B(\sqrt{R})\big).
\end{equation}

Heuristically, one should expect to be capable of determining the leading term asymptotic
using the circle method on hyperboloids $\mathcal{H}_{d,h}$ of sufficiently high
dimension.
More recently, Oh and Shah~\cite{ohshah2014} used ergodic methods to study the
three-dimensional hyperboloid $\mathcal{H}_{3,h}$ when $h$ is a positive square.
They proved the following theorem.
\begin{theorem}{Oh and Shah}
  Suppose that $h$ is a square.
  On $\mathcal{H}_{3,h}$, as $X \to \infty$,
  \begin{equation}
    N_{d,h}(X) = c X^{\frac{1}{2}}\log X + O \big( X^{\frac{1}{2}} (\log X)^{\frac{3}{4}}
    \big)
  \end{equation}
  for some constant $c > 0$.
\end{theorem}

In this chapter, we sharpen and extend this theorem to any dimension $d \geq 3$ and any
integral $h \geq 1$.
We also prove a smoothed analogue, including smaller-order growth terms.
The primary result is the following theorem.

\begin{theorem}\label{theorem:hyperboloid_intro_sharp}
  Let $d \geq 3$ and $h \geq 1$ be integers.
  Let $N_{d,h}(R)$ denote the number of integer points $\bm{m}$ on the hyperboloid
  $\mathcal{H}_{d,h}$ with $\| \bm{m} \|^2 \leq R$.
  Then for any $\epsilon > 0$,
  \begin{align}
    N_{d,h}(R) &= \delta_{[d = 3]} \delta_{[h = a^2]} C'_3 R^{\frac{1}{2}} \log R + C_d
    R^{\frac{d}{2} - 1} + O(R^{\frac{d}{2} - 1 - \lambda(d) + \epsilon}).
  \end{align}
  Here the Kronecker $\delta$ expressions indicate that the first term only occurs if $d =
  3$ and if $h$ is a square, and $\lambda(d) > 0$ is a constant depending only on the
  dimension $d$.
\end{theorem}

This Theorem is presented in greater detail as Theorem~\ref{thm:hyp:sharp_theorem_full},
including the description of $\lambda(d)$.
When $d = 3$, the power savings $\lambda(d)$ is exactly $\frac{1}{44}$.
As $d$ gets larger, $\lambda(d)$ grows and limits towards $\frac{1}{6}$.
Note that there is an error term with polynomial savings, which is a significant
improvement over previous results.
As a corollary, one recovers the Theorem of Oh and Shah.

In addition to the sharp estimate of Theorem~\ref{theorem:hyperboloid_intro_sharp}, we
consider smoothed approximations to $N_{d,h}(R)$.
In~\eqref{eq:hyp:points_equals_sum}, we show that $N_{d,h}(R) = \sum_{2m^2+h \leq R}
r_{d-1}(m^2 + h)$.
Then sums of the form
\begin{equation}
  \sum_{m \in \mathbb{Z}} r_{d-1}(m^2 + h) e^{-\frac{2m^2 + h}{R}}
\end{equation}
count the number of points $\bm{m}$ on $\mathcal{H}_{d,h}$ with exponential decay in $\|
\bm{m} \|$ once $\| \bm{m} \|^2 \geq R$.
This smoothed sum should be thought of as giving a smooth approximation to $N_{d,h}(R)$.
Through the methodology of this chapter, we prove the following smooth estimate.

\begin{theorem}\label{theorem:hyperboloid_intro_smooth}
  Let $d \geq 3$ and $h \geq 1$ be integers.
  Then for each $h$ and $d$, there exist constants $C'$ and $C_m$ such that for any
  $\epsilon > 0$,
  \begin{align}
    &\sum_{m \in \mathbb{Z}} r_{d-1}(m^2 + h) e^{-(2m^2 + h)/X} \\
    &\qquad = \delta_{[d=3]} \delta_{[h = a^2]} C' X^{\frac{1}{2}} \log X + \sum_{0 \leq m
    < \lceil \frac{d}{2} - 1 \rceil} C_{m} X^{\frac{d}{2} - 1 -\frac{m}{2}}
    + O(X^{\frac{d}{4} - \frac{1}{2} + \epsilon}).
  \end{align}
  Here, $\delta_{[\text{condition}]}$ is a Kronecker $\delta$ and evaluates to $1$ if the
  condition is true and $0$ otherwise.
\end{theorem}

See Theorem~\ref{thm:hyperboloid:smooth_full} in \S\ref{sec:hyp:proof_main_theorems} for a
more complete statement.
This Theorem suggests that for dimensions greater than $4$, there may be secondary main
terms with lower power contributions.

Theorems~\ref{theorem:hyperboloid_intro_sharp} and~\ref{theorem:hyperboloid_intro_smooth}
can be thought of as average order estimates of the function $r_{d-1}(m^2+h)$.
In particular, for $2m^2 + h \leq R$, the average value of $r_{d-1}(m^2+h)$ is about
$R^{\frac{d-1}{2}-1}$.
In the process of proving~\ref{theorem:hyperboloid_intro_sharp}, we also prove
that this average order estimate holds on short-intervals, i.e.\ intervals
around $R$ of length much less than $R$.

\begin{theorem}\label{theorem:hyperboloid_intro_short}
  Let $k \geq \frac{1}{2}$ be a full or half-integer.
  Then for each dimension $d$, there is a constant $\lambda(d) > 0$ such that
  \begin{equation}
    \sum_{\lvert 2m^2 + h - X \rvert < X^{1 + \epsilon - \lambda(d)}} r_{d-1}(m^2 + h) \ll
    X^{\frac{d}{2} - 1 + \epsilon - \lambda(d)}.
  \end{equation}
  The constant $\lambda(d)$ is the same constant as in
  Theorem~\ref{theorem:hyperboloid_intro_sharp}.
\end{theorem}

This Theorem can be roughly interpreted to count the number of lattice points $\bm{m}$ on
$\mathcal{H}_{d,h}$ with $\| \bm{m} \|$ very near $X$, or equivalently counting the number
of lattice points within a sphere of radius slightly larger than $\sqrt X$ and outside of
a sphere of radius a slightly smaller than $\sqrt X$.
This Theorem can be compared to the main theorem in short intervals
in~\cite{hkldwShort}, as described in Chapter~\ref{c:sums_apps}.

\begin{remark}
  To make heuristic sense of Theorem~\ref{theorem:hyperboloid_intro_short}, note that
  there are on the order of $X^{\frac{1}{2} + \epsilon - \lambda(d)}$ integers $m$ such
  that $\lvert 2m^2 + h - X \rvert < X^{1 + \epsilon - \lambda(d)}$.
  Therefore, if each of these values of $r_{d-1}(2m^2 + h)$ is approximately the size we
  expect, $X^{\frac{d-1}{2} - 1}$, then the total size should be
  \begin{equation}
    X^{\frac{d-1}{2} - 1} \cdot X^{\frac{1}{2} + \epsilon - \lambda(d)} = X^{\frac{d}{2} -
    1 + \epsilon - \lambda(d)},
  \end{equation}
  which is exactly what is shown in Theorem~\ref{theorem:hyperboloid_intro_short}.
\end{remark}

\subsection*{Overview of Methodology}

In order to count points on hyperboloids, let $d = 2k + 2$ (where $k$ may be a
half-integer).
Then in
\begin{equation}
  X_1^2 + \cdots + X_{2k+1}^2 = X_{2k+2}^2 + h,
\end{equation}
notice that for a point $\bm{X}$ on the hyperboloid,
\begin{equation}
  (X_1^2 + \cdots + X_{2k+1}^2) + X_{2k+2}^2 \leq R \iff 2X_{2k+2}^2 + h \leq R.
\end{equation}
It suffices to consider those points on the hyperboloid with $2X_{2k+2}^2 + h \leq R$.
Recall the notation that $r_d(n)$ is the number of representations of $n$ as a sum of
$d$ squares.
Then, breaking the hyperboloid into each possible value of $X_{2k+2}^2 + h$ and summing
across the number of representations as sums of squares, we have that
\begin{equation}\label{eq:hyp:points_equals_sum}
  N_{d,h}(R) = \sum_{2X_{2k+2}^2 + h \leq R} r_{2k+1}(X_{2k+2}^2 + h) = \sum_{2m^2 + h
  \leq R} r_{2k+1}(m^2 + h).
\end{equation}
We will find the number of points on the hyperboloid by estimating this last sum.

Consider the automorphic function
\index{V@$\widetilde{V}(z)$}
\begin{equation}
  \widetilde{V}(z) = \theta^{2k+1}(z) \overline{\theta(z)}y^{\frac{k+1}{2}},
\end{equation}
where
\index{theta@$\theta(z)$}
\begin{equation}
  \theta(z) = \sum_{n \in \mathbb{Z}} e^{2\pi i n^2 z}
\end{equation}
is the classical Jacobi theta function.
Heuristically, the $h$th Fourier coefficient of $\widetilde{V}(z)$ is a weighted
version of the sum $\sum_m r_{2k+1}(m^2+h)$, and so proper analysis of the the $h$th
Fourier coefficient of $\widetilde{V}(z)$ will give an estimate for $N_{d,h}(R)$.

More completely, let $P_h(z,s)$ denote a Poincar\'e series that isolates the $h$th Fourier
coefficient.
Then we will have that
\begin{equation}
  \frac{(2\pi)^{s + \frac{k-1}{2}}}{\Gamma(s + \frac{k-1}{2})} \langle P_h(\cdot, s),
  \widetilde{V} \rangle = \sum_{m \in \mathbb{Z}}
  \frac{r_{2k+1}(m^2+h)}{(2m^2+h)^{s+\frac{k-1}{2}}}.
\end{equation}
In order to understand the meromorphic properties of this Dirichlet series, we
will use a spectral expansion of the Poincar\'e series and understand each of the terms in
the spectral decomposition.
As $\widetilde{V}(z)$ is not square integrable, it is necessary to modify
$\widetilde{V}(z)$ by cancelling out the growth.
We do that in the next section by subtracting carefully chosen Eisenstein series.

Once the meromorphic properties of this Dirichlet series are understood, it only remains
to perform some classical cutoff integral transforms.
In Section~\ref{sec:hyp:proof_main_theorems}, we apply three Mellin integral
transforms described in Chapter~\ref{c:background} and perform classical
integral analysis in order to prove our main theorems of this chapter.

\section{Altering $\widetilde{V}$ to be Square-Integrable}
\index{V@$\widetilde{V}(z)$}

From the transformation laws of $\theta(z)$, we see that $\widetilde{V}$ satisfies the
transformation law
\begin{equation}\label{eq:V_transformation_law}
  \widetilde{V}(\gamma z) = \frac{\varepsilon_d^{-2k} \kron{c}{d}^{2k} (cz+d)^k}{\lvert cz+d
  \rvert^k} \widetilde{V}(z)
\end{equation}
for $\gamma = \Big( \begin{smallmatrix} a&b\\c&d \end{smallmatrix} \Big) \in \Gamma_0(4)$,
and where
\begin{equation}
  \varepsilon_d = \begin{cases}
    1 & d \equiv 1 \pmod 4 \\
    i & d \equiv 3 \pmod 4
  \end{cases}
\end{equation}
is the sign of the $d$th Gauss sum.
Therefore when $k$ is an integer, $\widetilde{V}$ is a modular form of full-integral
weight $k$ of nebentypus $\chi(\cdot) = \kron{-1}{\cdot}^k$ on $\Gamma_0(4)$.
When $k$ is a half-integer, $\widetilde{V}$ is a modular form of half-integral weight $k$
on $\Gamma_0(4)$ with a normalized theta multiplier system as described
in~\eqref{eq:V_transformation_law}.

Under the action of $\Gamma_0(4)$, the quotient $\Gamma_0(4)\backslash\mathcal{H}$
has three cusps: at $0, \frac{1}{2}$, and $\infty$.
We use $E_\mathfrak{a}^k(z,w)$ to denote the Eisenstein series of weight $k$ associated to
the cusp $\mathfrak{a}$, as detailed extensively in \S\ref{sec:Eisen_summary}.
We will soon see that $\widetilde{V}$ is non-cuspidal, and we will analyze the behavior of
$\widetilde{V}$ at each of the cusps.
In doing so, we will prove the following.

\begin{proposition}
  For $k \geq 1$, define $V(z)$ as
  \begin{equation}
    V(z) := \widetilde{V}(z) - E_\infty^k(z, \tfrac{k+1}{2}) - E_0^k(z, \tfrac{k+1}{2}).
  \end{equation}
  Then $V(z)$ is in $L^2(\Gamma_0(4)\backslash \mathcal{H}, k)$.

  In the case when $k = \frac{1}{2}$, we define
  \begin{equation}
    V(z) := \widetilde{V}(z) - \const_{w = \frac{3}{4}} E_\infty^k(z, w) - \const_{w =
    \frac{3}{4}} E_0^k(z,w),
  \end{equation}
  where $\const_{w=c} f(w)$ refers to the constant term in the Laurent expansion of $f(w)$
  expanded at $w=c$.
  Then $V(z)$ is in $L^2(\Gamma_0(4)\backslash \mathcal{H}, \frac{1}{2})$.
\end{proposition}

\begin{proof}

Writing $\widetilde{V}$ directly as
\begin{equation}
  \widetilde{V}(z) = \sum_{m_1, \ldots, m_{2k+2} \in \mathbb{Z}} y^{\frac{k+1}{2}} e^{2\pi
  i x(m_1^2 + \cdots m_{2k+1}^2 - m_{2k+2}^2)} e^{-2\pi y(m_1^2 + \cdots + m_{2k+2}^2)}
\end{equation}
shows that all terms have significant exponential decay in $y$, except when $m_1 = \cdots
= m_{2k+2} = 0$, in which case there is the term $y^{\frac{k+1}{2}}$.
Correspondingly, at the $\infty$ cusp, $\widetilde{V}(z)$ grows like $y^{\frac{k+1}{2}}$.
However the function $\widetilde{V}(z) - y^{\frac{k+1}{2}}$ has exponential decay as $y
\to \infty$.

At the $0$ cusp, we use $\sigma_0 = \Big(\begin{smallmatrix} 0&-\frac{1}{2} \\ 2&0
\end{smallmatrix}\Big)$, a matrix in $\SL(2, \mathbb{R})$ taking $0$ to $\infty$, and
directly compute
\begin{equation}
  \widetilde{V}\big|_{\sigma_0}(z) = \theta^{2k+1}\Big(\frac{-1}{4z}\Big)
  \overline{\theta\Big(\frac{-1}{4z}\Big)} \Im^{\frac{k+1}{2}} \Big(\frac{-1}{4z}\Big)
  \frac{\lvert -2iz \rvert^k}{(-2iz)^k}
\end{equation}

At the $\tfrac{1}{2}$ cusp, $\widetilde{V}$ has exponential decay because each $\theta(z)$
factor has exponential decay there.

Thus $\widetilde{V}$ grows like $y^{\frac{k+1}{2}}$ at the $\infty$ and $0$ cusps, and has
exponential decay at the $\frac{1}{2}$ cusp.
To cancel and better understand these growth terms, we subtract spectral Eisenstein series
associated to the cusps $0$ and $\infty$ with spectral parameter chosen so that the
leading growth of the Eisenstein series perfectly cancels the growth of $\widetilde{V}$.
We will use the properties of the full and half-integral weight Eisenstein series
associated to the cusp $\mathfrak{a}$, $E_\mathfrak{a}^k(z,w)$, as described more fully in
\S\ref{sec:Eisen_summary}.
In particular, it is shown in \S\ref{sec:Eisen_summary} that the constant terms in the
Fourier series of the Eisenstein series $E_\mathfrak{a}^k(z,w)$, expanded at the cusp
$\mathfrak{a}$, is of the shape\index{E@$E_\mathfrak{a}^k(z,s)$}
\begin{equation}
  y^w + c(w)y^{1-w}
\end{equation}
for a constant $c(w)$ depending on $w$.
Therefore, specializing the parameter $w = \frac{k+1}{2}$, the leading term from the
constant term of each Eisenstein series perfectly cancels the growth of $\widetilde{V}$ at
each cusp.
Further, each Eisenstein series is small at each cusp other than its associated cusp, so
for instance $E^k_\infty(z, \tfrac{k+1}{2})$ cancels the $y^{\frac{k+1}{2}}$ at the
$\infty$ cusp and is otherwise small at each other cusp (see~\cite{Iwaniec97} for more).

However, when $k$ is half-integral weight, the Eisenstein series
$E^k_\mathfrak{a}(z,w)$ has a pole at $w = \frac{3}{4}$.
When $k = \frac{1}{2}$, corresponding to the dimension $3$ hyperboloid,
the two Eisenstein series $E^{\frac{1}{2}}_\infty(z,w)$ and
$E_0^{\frac{1}{2}}(z,w)$ each have poles at
$w = \frac{k+1}{2} = \frac{3}{4}$, and so we cannot subtract
them from $\widetilde{V}$ directly.
Referring again to \S\ref{sec:Eisen_summary}, it is clear that the constant term of the
Laurent expansion at $w = \frac{3}{4}$ of each Eisenstein series contains the leading
growth terms $y^{\frac{3}{4}}$.
Since the constant term in the Laurent expansion is also modular, we conclude the $k =
\frac{1}{2}$ case.
\end{proof}

\section{Analytic Behavior}

Let $P_h^k(z,s)$ denote the weight $k$ Poincar\'e series
\begin{equation}
  P_h^k(z,s) = \sum_{\gamma \in \Gamma_\infty \backslash \Gamma_0(4)} \Im(\gamma z)^s
  e^{2\pi i h \gamma z} J(\gamma, z)^{-2k}
\end{equation}
where
\begin{equation}
  J(\gamma, z) = \frac{j(\gamma, z)}{\lvert j(\gamma, z) \rvert}
\end{equation}
and $j(\gamma, z) = \theta(\gamma z)/\theta(z) = \varepsilon^{-1} \kron{c}{d}
(cz+d)^{\frac{1}{2}}$, exactly as for the Eisenstein series defined in
Chapter~\ref{c:background}.

Our basic strategy is to understand the Petersson inner product $\langle P_h^k(\cdot, s),
V(z) \rangle$ in two different ways.
On the one hand, we will compute it directly, giving a Dirichlet series $D_h^k(s)$ with
coefficients $r_{2k+1}(m^2 + h)$.
On the other hand, we will take a spectral expansion of $P_h^k$ and understand the
meromorphic properties of each part of the spectral expansion.

\subsection{Direct Expansion}\label{ssec:direct_expansion}

We first understand $\langle P_h^k(\cdot, s), V \rangle$ directly, using the method of
unfolding:
\begin{align}
  \langle P_h^k(\cdot, s), V \rangle &= \int \int_{\Gamma_0(4) \backslash \mathcal{H}}
  P_h(z,s) \overline{V(z)} \frac{dx dy}{y^2} \\
  &= \int_0^\infty \int_0^1 y^{s-1} e^{2\pi i h z} \overline{V(z)} dx \frac{dy}{y}.
\end{align}
Initially, we consider the case when $k > \frac{1}{2}$.
We'll consider the three dimensional case, when $k = \frac{1}{2}$, afterwards.

\subsubsection*{Dimension $\geq 4$}

Writing $V = \widetilde{V} - E_\infty^k(z, \frac{k+1}{2}) - E_0^k(z, \frac{k+1}{2})$, we
compute
\begin{align}
  &\langle P_h^k(\cdot, s), E_\infty^k(z, \tfrac{k+1}{2}) + E_0^k(z, \tfrac{k+1}{2})
\rangle \\
  &\quad = \frac{\overline{\rho_\infty^k(h, \frac{k+1}{2}) + \rho_0^k(h,
  \frac{k+1}{2})}}{(4\pi h)^{s-1}} \frac{\Gamma(s + \frac{k}{2} - \frac{1}{2}) \Gamma(s -
\frac{k}{2} - \frac{1}{2})}{\Gamma(s - \frac{k}{2})}.
\end{align}
Expanding $\widetilde{V}$, we can compute the remaining $x$ integral as
\begin{align}
  &\int_0^1 \overline{\widetilde{V}(z)} e^{2\pi i h x} dx = \int_0^1
  \overline{\theta^{2k+1}(z)}\theta(z)y^{\frac{k+1}{2}} e^{2\pi i h x} dx \\
  &\qquad = \sum_{m_1, \ldots, m_{2k+2}} y^{\frac{k+1}{2}} e^{2\pi y(m_1^2 + \cdots +
m_{2k+2}^2)} \int_0^1 e^{-2\pi i x(m_1^2 + \cdots + m_{2k+1}^2 - m_{2k+2}^2 - h)} dx \\
  &\qquad = \quad y^{\frac{k+1}{2}}
  \sum_{\mathclap{\substack{\bm{m} \in \mathbb{Z}^{2k+2} \\
  m_1^2 + \cdots + m_{2k+1}^2 = m_{2k+2}^2 + h}}} e^{-2\pi y(m_1^2 + \cdots m_{2k+2}^2)} =
  \quad y^{\frac{k+1}{2}} \sum_{\mathclap{\substack{\bm{m} \in \mathbb{Z}^{2k+2} \\
  m_1^2 + \cdots + m_{2k+1}^2 = m_{2k+2}^2 + h}}} e^{-2\pi y(2m_{2k+2}^2 + h)} \\
  &\qquad = y^{\frac{k+1}{2}} \sum_{m \in \mathbb{Z}} r_{2k+1}(m^2 + h)
  e^{-2\pi y (2m^2 + h)}.
\end{align}
To go from the penultimate line to the last line, we write $m = m_{2k+2}$ and count the
number of representations of $m^2 + h$.
We compute the remaining $y$ integral
\begin{equation}
  \int_0^\infty y^{s + \frac{k-1}{2}} \sum_{m \in \mathbb{Z}} r_{2k+1}(m^2 + h) e^{-2\pi
  y(2m^2 + h)} \frac{dy}{y} = \sum_{m \in \mathbb{Z}} \frac{r_{2k+1}(m^2 + h)}{(2m^2 +
  h)^{s + \frac{k-1}{2}}} \frac{\Gamma(s + \frac{k-1}{2})}{(2\pi)^{s + \frac{k-1}{2}}}.
\end{equation}

Define\index{D@$D_h^k(s)$}
\begin{equation}\label{eq:hyp:Dh_def}
  D^k_h(s) := \frac{(2\pi)^{s + \frac{k-1}{2}}}{\Gamma(s + \frac{k-1}{2})} \langle
  P_h^k(\cdot, s), V\rangle.
\end{equation}
Then our computation above shows that for $k \geq 1$, $D^k_h(s)$ can be written as
\begin{equation}\label{eq:basic_expansion_final}
  \sum_{m \in \mathbb{Z}}\frac{r_{2k+1}(m^2 + h)}{(2m^2 + h)^{s + \frac{k-1}{2}}} -
  \mathfrak{E}_h^k(s)
\end{equation}
where\index{E@$\mathfrak{E}_h^k(s)$}
\begin{equation}
  \mathfrak{E}_h^k(s) =  \frac{(2\pi)^{\frac{k+1}{2}} (\overline{\rho_\infty^k(h,
  \frac{k+1}{2}) + \rho_0^k(h, \frac{k+1}{2})})}{(2 h)^{s-1}} \frac{\Gamma(s - \frac{k}{2}
- \frac{1}{2})}{\Gamma(s - \frac{k}{2})}.
\end{equation}
Notice that $\mathfrak{E}_h^k(s)$ has poles at $s = \frac{k+1}{2} - m$ for $m \in
\mathbb{Z}_{\geq 0}$, coming from the Gamma function in the
numerator, and clear meromorphic continuation.\footnote{$\mathfrak{E}$ is an E in an old
  German font, which many mathematicians would pronounce as ``fraktur E'' or ``mathfrak
E.'' We use it because those terms come from Eisenstein series.}

After applying Stirling's approximation to estimate the Gamma functions, we have the
following proposition.

\begin{proposition}\label{prop:nonspectral_analytic_props_large_dim}
  With the notation above and with $k \geq 1$, we have
  \begin{equation}
    D^k_h(s) := \frac{(2\pi)^{s + \frac{k-1}{2}}}{\Gamma(s + \frac{k-1}{2})} \langle
    P_h^k(\cdot, s), V\rangle = \sum_{m \in \mathbb{Z}}\frac{r_{2k+1}(m^2 + h)}{(2m^2 +
  h)^{s + \frac{k-1}{2}}} - \mathfrak{E}_h^k(s).
  \end{equation}
  The function $\mathfrak{E}_h^k(s)$ is analytic for $\Re s > \frac{1}{2}$ except for
  simple poles at $s = \frac{k+1}{2} - m$ for $m \in \mathbb{Z}_{\geq 0}$, and has
  meromorphic continuation to the plane.
  For $s$ away from poles with $\Re s > \frac{1}{2}$, we have the bound
  \begin{equation}
    \mathfrak{E}_h^k(s) \ll (1 + \lvert s \rvert)^{-\frac{1}{2}}.
  \end{equation}
\end{proposition}

\subsubsection*{Dimension $3$}

We proceed analogously, and write
$V = \widetilde{V} - \const_{w = \frac{3}{4}} E_\infty^{\frac{1}{2}}(z, w) -
\const_{w = \frac{3}{4}} E_0^{\frac{1}{2}} (z, w)$, initially for $\Re s \gg 1$.
We now need to compute
\begin{equation}
  \langle P_h^{\frac{1}{2}} (\cdot, s), \const_{w = \frac{3}{4}}
  E_\mathfrak{a}^{\frac{1}{2}} (\cdot, w) \rangle
  =
  \const_{w = \frac{3}{4}} \; \langle P_h^{\frac{1}{2}} (\cdot, s),
  E_\mathfrak{a}^{\frac{1}{2}} (\cdot, w) \rangle.
\end{equation}
Computing the expression for the inner product
$\langle P_h^{\frac{1}{2}} , E_\mathfrak{a}^{\frac{1}{2}}  \rangle$
for $\Re s \gg 1$ directly and then taking the constant term in $w$ gives that
\begin{equation}\label{eq:dim_3_constant_laurent}
  \const_{w = \frac{3}{4}} \;
  \langle P_h^{\frac{1}{2}} (\cdot, s), E_\mathfrak{a}^{\frac{1}{2}} (\cdot, w) \rangle
  =
  \const_{w = \frac{3}{4}}
  \frac{\overline{\rho_{\mathfrak{a}}^{\frac{1}{2}} (h,w)}}{(4\pi h)^{s-1}}
  \frac{\Gamma(s+\overline{w}-1)\Gamma(s-\overline{w})}{\Gamma(s - \frac{1}{4})}.
\end{equation}
We must now make sense of this constant term.

In \S\ref{sec:Eisen_summary}, it is shown that $\rho_\mathfrak{a}^{\frac{1}{2}}(h,w)$
has a simple pole at $w = \frac{3}{4}$ if and only if $h$ is a positive square,
and otherwise is analytic.
Thus the constant term~\eqref{eq:dim_3_constant_laurent} manifests in three ways:
\begin{enumerate}
  \item The constant terms of $\rho_{\mathfrak{a}}^{\frac{1}{2}}(h,w)$,
    $\Gamma(s + \overline{w} - 1)$, and $\Gamma(s - \overline{w})$
  \item The residue term of $\rho_{\mathfrak{a}}^{\frac{1}{2}}(h,w)$,
    the constant term of $\Gamma(s + \overline{w} - 1)$,
    and the linear term of $\Gamma(s - \overline{w})$
  \item The residue term of $\rho_{\mathfrak{a}}^{\frac{1}{2}}(h,w)$,
    the linear term of $\Gamma(s + \overline{w} - 1)$,
    and the constant term of $\Gamma(s - \overline{w})$.
\end{enumerate}
Together, these mean that~\eqref{eq:dim_3_constant_laurent} can be written as
\begin{equation} \label{eq:dim_3_laurent_simplify_I}
  \begin{split}
    &\const_{w = \frac{3}{4}} \; \langle P_h^{\frac{1}{2}} (\cdot, s),
    E_\mathfrak{a}^{\frac{1}{2}} (\cdot, w) \rangle = \\
    &\quad = \frac{\const_{w = \frac{3}{4}}
    \rho_\mathfrak{a}^{\frac{1}{2}}(h,w)
    \Gamma(s - \frac{1}{4})\Gamma(s - \frac{3}{4})}
    {(4\pi h)^{s-1} \Gamma(s - \frac{1}{4})}
    \\
    &\qquad + \frac{\Res_{w = \frac{3}{4}}
    \rho_{\mathfrak{a}}^{\frac{1}{2}}(h,w)}{(4\pi h)^{s-1}}
    \bigg( \frac{\Gamma'(s - \frac{1}{4}) \Gamma(s - \frac{3}{4})}
      {\Gamma(s - \frac{1}{4})}
      +
    \frac{\Gamma'(s - \frac{3}{4}) \Gamma(s - \frac{1}{4})}{\Gamma(s - \frac{1}{4})}\bigg)
    \\
    &\quad =\frac{\const_{w = \frac{3}{4}} \rho_\mathfrak{a}^{\frac{1}{2}}(h,w)
    \Gamma(s - \frac{3}{4})}{(4\pi h)^{s-1}}
    \\
    &\qquad + \frac{\Res_{w = \frac{3}{4}}
    \rho_{\mathfrak{a}}^{\frac{1}{2}}(h,w)}{(4\pi h)^{s-1}}
    \bigg( \frac{\Gamma'(s - \frac{1}{4}) \Gamma(s - \tfrac{3}{4})}
      {\Gamma(s - \frac{1}{4})}
      +
    \Gamma'(s - \tfrac{3}{4})\bigg).
  \end{split}
\end{equation}
This expression has clear meromorphic continuation to the plane, and the poles and
analytic behavior can be determined from the individual Gamma functions.
This expression has a simple pole at $s = \frac{3}{4}$, and when
$\rho_\mathfrak{a}^{\frac{1}{2}}(s,w)$ has a pole at $w = \frac{3}{4}$, this
expression has a double pole in $s$ at $s = \frac{3}{4}$ coming
from $\Gamma'(s - \frac{3}{4})$.
These are the only poles in this expression when $\Re s > \frac{1}{2}$.

As in the case when $k \geq 1$, we define
\begin{equation}\label{eq:hyp:Dh_def_half}
  D_h^{\frac{1}{2}}(s) :=
  \frac{(2\pi)^{s - \frac{1}{4}}}{\Gamma(s - \frac{1}{4})}
  \langle P_h^{\frac{1}{2}}(\cdot, s), V \rangle.
\end{equation}
At each cusp $\mathfrak{a}$, we also define
\begin{equation}
  \mathfrak{E}_{h, \mathfrak{a}}^{\frac{1}{2}}(s)
  :=
  \frac{(2\pi)^{s - \frac{1}{4}}}{\Gamma(s - \frac{1}{4})} \const_{w = \frac{3}{4}}
  \langle P_h^{\frac{1}{2}}(\cdot, s), E_\mathfrak{a}^{\frac{1}{2}}(\cdot, w) \rangle.
\end{equation}
Notice that this is $(2\pi)^{s-\frac{1}{4}} \Gamma(s - \frac{1}{4})^{-1}$ times
the expression in~\eqref{eq:dim_3_laurent_simplify_I}.
Finally, define
\begin{equation}
  \mathfrak{E}_h^{\frac{1}{2}}(s)
  :=
  \mathfrak{E}_{h, \infty}^{\frac{1}{2}}(s) + \mathfrak{E}_{h, 0}^{\frac{1}{2}}(s).
\end{equation}
Then when $k \geq 1$, we have that
\begin{equation}
  D^{\frac{1}{2}}_h(s) = \frac{(2\pi)^{s - \frac{1}{4}}}{\Gamma(s - \frac{1}{4})}
  \langle P_h^{\frac{1}{2}}(\cdot, s), V\rangle
  =
  \sum_{m \in \mathbb{Z}}
  \frac{r_{3}(m^2 + h)}{(2m^2 + h)^{s - \frac{1}{4}}}
  -
  \mathfrak{E}_h^{\frac{1}{2}}(s).
\end{equation}
Although the intermediate steps are different, this final notation agrees with the
notation for $k \geq 1$.

To roughly understand the growth of $\Gamma'(s)$, it suffices to use Cauchy's
Integral Formula by examining (for $\lvert \Im s \rvert \gg 1$)
\begin{equation}
  \Gamma'(s) = \frac{1}{2\pi i} \int_{\mathcal{B}_1(s)} \frac{\Gamma(z)}{(z - s)^2}dz \ll
  \max_{0 \leq \theta \leq 2\pi} \lvert \Gamma(s + e^{i \theta}) \rvert,
\end{equation}
where $\mathcal{B}_1(s)$ is the circle of radius $1$ around $s$.
It is then straightforward to get rough bounds on $\Gamma'(s)$ through
Stirling's Approximation.

\begin{remark}
  It is possible to get much stronger bounds, but it will turn out that this
  rough bound suffices.
\end{remark}

We gather the relevant details from this section into the following proposition.

\begin{proposition}\label{prop:nonspectral_analytic_props_dim_3}
  With the notation above,
  \begin{equation}
    \begin{split}
      D^{\frac{1}{2}}_h(s) = \frac{(2\pi)^{s - \frac{1}{4}}}{\Gamma(s - \frac{1}{4})}
      \langle P_h^{\frac{1}{2}}(\cdot, s), V\rangle
      =
      \sum_{m \in \mathbb{Z}}
      \frac{r_{3}(m^2 + h)}{(2m^2 + h)^{s - \frac{1}{4}}}
      -
      \mathfrak{E}_h^{\frac{1}{2}}(s).
    \end{split}
  \end{equation}
  Further, $\mathfrak{E}_h^{\frac{1}{2}}(s)$ has meromorphic continuation to the plane,
  and is analytic for $\Re s > \frac{1}{2}$ except for a pole at $s = \frac{3}{4}$.
  If $h$ is a square, this is a double pole.
  If $h$ is not a square, then this is a simple pole.

  For $s$ away from the pole at $\tfrac{3}{4}$, $\Re s > \frac{1}{2}$, we have the bound
  \begin{equation}
    \mathfrak{E}_h^{\frac{1}{2}}(s) \ll (1 + \lvert s \rvert)^{\frac{1}{2}}.
  \end{equation}
\end{proposition}

\subsection{Spectral Expansion}
\index{mu@$\mu_j$}

Fix an orthonormal basis of Maass forms of weight $k$ for $L^2(\Gamma_0(4)\backslash
\mathcal{H}, k)$.
This basis consists of Maass forms $\{\mu_j(z)\}$ of types $\frac{1}{2} + it_j$ and
corresponding eigenvalues $\frac{1}{4} + t_j^2$, each with expansion
\begin{equation}
  \mu_j(z) = \sum_{n \neq 0} \rho_j(n) W_{\frac{n}{\lvert n \rvert} \frac{k}{2},
  it_j}(4\pi \lvert n \rvert y) e^{2\pi i n x},
\end{equation}
as well as a finite number of Maass forms $\mu_{j,\ell}(z)$ with eigenvalues
$\frac{\ell}{2}(1 - \frac{\ell}{2})$ with $1 \leq \ell \leq k$ for $\ell$ a (possibly
half-integer) satisfying $\ell \equiv k \bmod 2$, each with expansion
\begin{equation}
  \mu_{j, \ell}(z) = \sum_{n \neq 0} \rho_{j,\ell}(n) W_{\frac{n}{\lvert n \rvert}
  \frac{k}{2}, \frac{\ell-1}{2}}(4\pi \lvert n \rvert y) e^{2\pi i n x}.
\end{equation}
Note that for each $\ell$, there are finitely many such Maass forms, as these come from
holomorphic cusp forms of weight $\ell$.
These Maass forms contribute the so-called \emph{bottom of the spectrum}, as described
in~\cite[Chapter 3]{GoldfeldHundleyI}.

Then $P_h^k$ has a Selberg Spectral decomposition~\cite{IwaniecKowalski04,
Goldfeld2006automorphic} of the form
\begin{align}
  P_h^k(z,s) &= \sum_j \langle P_h^k(\cdot, s), \mu_j \rangle \mu_j(z) + \sum_{\frac{1}{2}
  \leq \ell \leq k} \sum_j \langle P_h^k(\cdot, s), \mu_{j, \ell}\rangle \mu_{j, \ell}(z)
  \label{line:1:Ph_discrete} \\
  &\quad + \sum_{\mathfrak{a}} \langle P_h^k(\cdot, s), R^k_{\mathfrak{a}} \rangle
  R^k_{\mathfrak{a}}(z) \label{line:2:Ph_residual} \\
  &\quad + \frac{1}{4\pi i} \sum_{\mathfrak{a}} \int_{(\frac{1}{2})} \langle P_h^k(\cdot,
  s), E_{\mathfrak{a}}^k (\cdot, u)\rangle E_{\mathfrak{a}}^k (z, u) du.
  \label{line:3:Ph_continuous}
\end{align}
In this expansion, line~\eqref{line:1:Ph_discrete} is the \emph{discrete part of the
spectrum}, line~\eqref{line:2:Ph_residual} is the \emph{residual part of the spectrum},
and line~\eqref{line:3:Ph_continuous} is the \emph{continuous part of the spectrum}.
\index{spectrum!discrete part}
\index{spectrum!residual part}
\index{spectrum!continuous part}
The sums over $\mathfrak{a}$ are sums over the three cusps of $\Gamma_0(4)$.
Note that the residual part of the spectrum exists only when $k$ is a half-integer, in
which case
\begin{equation}\label{eq:hyp:residual_spec_def}
  R_{\mathfrak{a}}^k(z) = \Res_{w = \frac{3}{4}} E_{\mathfrak{a}}^k(z, w).
\end{equation}

Each of the inner products against $P_h^k$ can be directly evaluated.
These computations are very similar to those computations in \S\ref{sec:Eisen_summary} and
\S\ref{ssec:direct_expansion}, and we omit them.
We collect these together in the following lemma.

\begin{lemma}\label{lem:inner_product_list}
  Maintaining the notation above, we have
  \begin{align}
    \langle P_h^k(\cdot, s), \mu_j \rangle &= \frac{\overline{\rho_j(h)}}{(4\pi h)^{s-1}}
    \frac{\Gamma(s - \tfrac{1}{2} + it_j) \Gamma(s - \tfrac{1}{2} - it_j)}{\Gamma(s -
    \frac{k}{2})} \\
    \langle P_h^k(\cdot, s), \mu_{j,\ell} \rangle &=
    \frac{\overline{\rho_{j,\ell}(h)}}{(4\pi h)^{s-1}} \frac{\Gamma(s + \frac{\ell}{2} -
    1) \Gamma(s - \frac{\ell}{2})}{\Gamma(s - \frac{k}{2})} \\
    \langle P_h(\cdot, s), R^k_{\mathfrak{a}} \rangle &= \frac{\Res_{w =
    \frac{3}{4}}\overline{\rho_{\mathfrak{a}}^k(h,w)}}{(4\pi h)^{s - 1}}  \frac{\Gamma(s -
    \frac{1}{4}) \Gamma(s - \frac{3}{4})}{\Gamma(s - \frac{k}{2})} \\
    \langle P_h(\cdot, s), E_{\mathfrak{a}}^k(\cdot, \tfrac{1}{2} + it)\rangle &=
    \frac{\overline{\rho_\mathfrak{a}(h, \tfrac{1}{2} + it)}}{(4\pi h)^{s-1}}
    \frac{\Gamma(s - \tfrac{1}{2} + it) \Gamma(s - \tfrac{1}{2} - it)}{\Gamma(s -
    \frac{k}{2})}.
  \end{align}
  Here, $\rho_j(h)$ is the $h$th coefficient of $\mu_j$, $\rho_{j, \ell}(h)$ is the $h$th
  coefficient of $\mu_{j,\ell}$, $\rho_{\mathfrak{a}}(h, \tfrac{1}{2} + it)$ is the
  $h$th coefficient of $E_{\mathfrak{a}}^k(z, \tfrac{1}{2} + it)$, and $R_\mathfrak{a}^k$
  is as in~\eqref{eq:hyp:residual_spec_def}.
\end{lemma}

\subsection{Meromorphic Continuation of $\langle P_h^k, V \rangle$}

In order to provide a meromorphic continuation for $D_h^k(s)$ (defined
in~\eqref{eq:hyp:Dh_def} and~\eqref{eq:hyp:Dh_def_half}), we provide a meromorphic
continuation for the expression coming from the spectral decomposition of $P_h^k(z,s)$.
Inserting the spectral decomposition of $P_h^k(z,s)$ into $\langle P_h^k(\cdot, s), V
\rangle$, we get
\begin{align}
  \langle P_h^k(\cdot, s), V \rangle &= \sum_j \langle P_h^k(\cdot, s), \mu_j \rangle
  \langle \mu_j, V \rangle + \sum_{\frac{1}{2} \leq \ell \leq k} \sum_j \langle
  P_h^k(\cdot, s), \mu_{j,\ell}\rangle \langle \mu_{j,\ell}, V \rangle
  \label{line:Ph_spectral_discrete} \\
  &\quad + \sum_{\mathfrak{a}} \langle P_h^k(\cdot, s), R^k_\mathfrak{a}\rangle \langle
  R^k_\mathfrak{a}, V \rangle \label{line:Ph_spectral_residual} \\
  &\quad + \frac{1}{4\pi i} \sum_{\mathfrak{a}} \int_{(\frac{1}{2})} \langle P_h^k(\cdot,
  s), E^k_{\mathfrak{a}}(\cdot, u) \rangle \langle E^k_{\mathfrak{a}}(\cdot, u), V \rangle
  du, \label{line:Ph_spectral_continuous}
\end{align}
where we have again separated the expressions into separate lines for the discrete
spectrum, the residual spectrum, and the continuous spectrum.
To study the meromorphic continuation, we provide separate meromorphic continuations for
the discrete spectrum, the residual spectrum, and the continuous spectrum in turn.

\subsubsection*{Discrete Spectrum}
\index{spectrum!discrete part}

Consider the discrete spectrum appearing in line~\eqref{line:Ph_spectral_discrete}, which
we rewrite for convenience:
\begin{equation}
  \langle P_h^k(\cdot, s), V \rangle = \sum_j \langle P_h^k(\cdot, s), \mu_j \rangle
  \langle \mu_j, V \rangle + \sum_{\frac{1}{2} \leq \ell \leq k} \sum_j \langle
  P_h^k(\cdot, s), \mu_{j,\ell}\rangle \langle \mu_{j,\ell}, V \rangle.
\end{equation}
Analysis of the discrete spectrum naturally breaks into two categories: analysis of the
finitely many Maass forms $\{\mu_{j, \ell}\}_{j, \ell}$ coming from holomorphic cusp
forms of weight $\ell$, and analysis of the infinitely many Maass forms $\{\mu_j\}_{j}$
with corresponding types $\tfrac{1}{2} = it_j$.
Taking inspiration from~\cite[\S3.10]{GoldfeldHundleyI}, we refer to the Maass
forms $\{\mu_{j, \ell}\}_{j, \ell}$ as the \emph{old discrete spectrum}, and
\index{discrete spectrum!old}
\index{discrete spectrum!new}
the Maass forms $\{\mu_j\}_j$ as the \emph{new discrete spectrum}.
We will prove the following proposition in this section.

\begin{proposition}\label{prop:hyperboloid_discrete_props}
  Write $s = \sigma + it$.
  The discrete spectrum component has meromorphic continuation to the plane, and
  \begin{enumerate}

    \item For $\Re s > \frac{1}{2}$, the new discrete spectrum is analytic and satisfies
      the bound
      \begin{equation}
        \sum_j \langle P_h^k(\cdot, s), \mu_j\rangle \langle \mu_j, V \rangle \ll_{h,
        \epsilon} (1 + \lvert t \rvert)^{\frac{7k}{2} + \frac{15}{2} + \epsilon}
        e^{-\frac{\pi}{2} \lvert t \rvert}. \qquad (\tfrac{1}{2} < \Re s < 1)
      \end{equation}
      The new discrete spectrum has a line of poles on $\Re s = \frac{1}{2}$.

    \item For $\Re s > \frac{k}{2} - 1$, the old discrete spectrum is analytic and
      satisfies the bound
      \begin{equation}
        \sum_{\frac{1}{2} \leq \ell \leq k} \sum_j \langle P_h^k(\cdot, s),
        \mu_{j,\ell}\rangle \langle \mu_{j,\ell}, V \rangle \ll (1 + \lvert t
      \rvert)^{\frac{k}{2} - \frac{3}{2} + \sigma} e^{-\frac{\pi}{2}\lvert t \rvert}.
      \end{equation}
      In the region $\Re s > 0$, the old discrete spectrum has simple poles at $s =
      \frac{k}{2} - 1 - m$ for $m \in \mathbb{Z}_{\geq 0}$.

  \end{enumerate}
\end{proposition}

In order to prove this proposition, we prove a sequence of lemmata.
We first bound the infinite sum in the new discrete spectrum.

\begin{lemma}\label{lem:discrete_newspectrum_innerproduct_bound}
  With $\mu_j$ coming from the new discrete spectrum, and with the same notation as above,
  we have
  \begin{equation}
    \sum_{T \leq \lvert t_j \rvert \leq 2T} \rho_j(h) \langle \mu_j, V \rangle = \sum_{T
    \leq \lvert t_j \rvert \leq 2T} \rho_j(h) \langle \mu_j, \theta^{2k+1}
    \overline{\theta} y^{\frac{k+1}{2}} \rangle \ll T^{3k + 8 + \epsilon}.
  \end{equation}
\end{lemma}

\begin{proof}
  First note that $E^k_\mathfrak{a}(z, \tfrac{k+1}{2})$ is orthogonal to the cusp form
  $\mu_j$, which gives the first equality.
  We recognize $\theta(z) y^{1/4}$ as a constant times $\Res_{u = \frac{3}{4}}
  E_\infty^{\frac{1}{2}}(z,u)$.
  Performing this on $\overline{\theta}$ transforms each inner product into
  \begin{equation}
    \Res_{u = \frac{3}{4}} \langle \mu_j, \theta^{2k+1}
    \overline{E_\infty^{\frac{1}{2}}(\cdot,u)} y^{\frac{2k+1}{4}}\rangle.
  \end{equation}
  Using the standard unfolding argument on the Eisenstein series (which uses similar
  methodology to the computations of $\langle P, V \rangle$ in
  \S\ref{ssec:direct_expansion}), we see that this is equal to
  \begin{equation}
    \Res_{u = \frac{3}{4}}
    \sum_{n \geq 1} \frac{r_{2k+1}(n)\rho_j(n)}{(4\pi n)^{u + \frac{k}{2} - \frac{3}{4}}}
    \frac{%
      \Gamma(u + \frac{k}{2} - \frac{3}{4} + \frac{1}{2} + it_j)
      \Gamma(u + \frac{k}{2} - \frac{3}{4} + \frac{1}{2} - it_j)
    }
    {\Gamma(u + \frac{1}{4})}.
  \end{equation}
  We bound the size of this residue by first proving a bound in $u$ using
  the Phragm\'{e}n-Lindel\"{o}f principle, and then using Cauchy's Residue Theorem to
  bound the sum.

  From the average estimate $r_d(n) \approx n^{\frac{d}{2} - 1}$, one can show that
  the summation converges trivially absolutely for $\Re u > \frac{k}{2} + \tfrac{5}{4}$.
  Applying Phragm\'{e}-Lindel\"{o}f and using Stirling's approximation then shows that
  \begin{equation}
    \rho_j(h) \langle \mu_j, \theta^{2k+1} \overline{\theta} y^{\frac{k+1}{2}} \rangle
    \ll
    \rho_j(h) (1 + \lvert t_j \rvert)^{3k+6+\epsilon} e^{-\pi \lvert t_j \rvert}.
  \end{equation}

  \begin{remark}
    Note that as a residue in $u$ is being taken, the relevant bound from the $t_j$
    contribution, which is entirely determined by the factors
    \begin{equation}
      \rho_j(1)
      \Gamma(u + \tfrac{k}{2} - \tfrac{1}{4} + it_j)
      \Gamma(u + \tfrac{k}{2} - \tfrac{1}{4} - it_j).
    \end{equation}
    Heuristically, the final bound can be attained just from using the Phragm\'{e}n-Lindel\"{o}f
    Convexity principle on this piece.
  \end{remark}

  As noted in~\cite{hulseCountingSquare}, it is possible to understand bound $\rho_j(h)$
  on average over $j$.
  Through the standard Kuznetsof Trace Formula and arguing as in~\cite[Section
  4]{HoffsteinHulse13}, one can show that $\rho_j(h) \ll_{\epsilon, h}
  e^{\frac{\pi}{2}\lvert t_j \rvert}(1 + \lvert t_j \rvert)^\epsilon$ (on average over
  $j$) for full integral weight $k$.
  By using Proskurin's generalization of the Kuznetsof Trace Formula, described
  in~\cite{Duke88}, one can show the same for half integral weight $k$.
  (This is closely related to an argument in~\cite{hulseCountingSquare}).

  \begin{remark}
    For full-integral weight Maass forms, Goldfeld, Hoffstein, and
    Lieman~\cite{goldfeld1994appendix} showed that $\rho_j(h) \ll h^\theta \log(1 + \lvert
    t_j \rvert)  e^{-\frac{\pi}{2} \lvert t_j \rvert}$ for each \emph{individual} $t_j$,
    where $\theta$ denotes the best known progress towards the non-achimedian Ramanujan
    conjecture.
    This is a superior bound than even the on-average bound given by the Kuznetsof Trace
    Formula, but it doesn't immediately generalize to half-integral weight Maass forms.
  \end{remark}

The on-average bounds then give that
  \begin{equation}
    \langle \mu_j, \theta^{2k+1} \overline{\theta} y^{\frac{k+1}{2}} \rangle \ll (1 +
    \lvert t_j \rvert)^{3k + 6 + \epsilon} e^{-\frac{\pi}{2} \lvert t_j \rvert}
  \end{equation}
  on average.
  Recalling that there are $O(T^2 \log T)$ Maass forms with $T \leq \lvert t_j \rvert \leq
  2T$, summing these terms together gives
  \begin{equation}
    \sum_{T \leq \lvert t_j \rvert \leq 2T} \rho_j(h) \langle \mu_j, \theta^{2k+1}
    \overline{\theta} y^{\frac{k+1}{2}} \rangle \ll (1 + \lvert t_j \rvert)^{3k + 8 +
    \epsilon}.
  \end{equation}
  This concludes the proof.
\end{proof}

It is also necessary to note that Selberg's Eigenvalue Conjecture is known for weight $k$
Maass forms on $\Gamma_0(4)$.

\begin{lemma}
  Suppose $\mu_j$ is a Maass form appearing in the old discrete spectrum described above.
  Denote the type of $\mu_j$ by $\frac{1}{2} + it_j$, so that $\mu_j$ is an eigenfunction
  of the Laplacian with eigenvalue $\lambda = \frac{1}{4} + t_j^2$.
  Then $\lambda > \frac{1}{4}$.
  That is, there are no exceptional eigenvalues.\index{Selberg Eigenvalue Conjecture}
\end{lemma}

\begin{proof}
  This is an argument given by Gergely Harcos at the Alfr\'ed R\'enyi Institute of
  Mathematics in an answer at MathOverflow~\cite{GHfromMO}.
  We sketch the argument here.
  A half-integral weight Maass form with eigenvalue $(1 - t^2)/4$ has a Shimura lift to an
  integral weight Maass form on $\Gamma_0(1)$ with eigenvalue $\frac{1}{4} - t^2$.
  As Selberg's Eigenvalue Conjecture is known in this case (see~\cite{Blomer2013}), this
  completes the proof.
\end{proof}

Now fix $s = \sigma + it$ with $\frac{1}{2} < \sigma < 1$.
We write $\lvert t_j \rvert \sim T$ to mean $T \leq \lvert t_j \rvert \leq 2T$ for the
rest of this section.
Then
\begin{equation}
  \sum_{\lvert t_j \rvert \sim T} \langle P_h^k(\cdot, s), \mu_j \rangle \langle \mu_j, V
  \rangle = \sum_{\lvert t_j \rvert \sim T} \frac{\Gamma(s - \frac{1}{2} + it_j) \Gamma(s
- \frac{1}{2} - it_j)}{(4\pi h)^{s - 1} \Gamma(s - \frac{k}{2})} \rho_j(h) \langle \mu_j,
V \rangle.
\end{equation}
In this expression it is clear that the rightmost poles in $s$ are at $s = \frac{1}{2} \pm
it_j$, which occur on the line $\Re s = \frac{1}{2}$.
By Stirling's Approximation, this is asymptotically
\begin{equation}\label{eq:discrete_newspectrum_stepI}
  \sum_{\lvert t_j \rvert \sim T} \frac{(1 + \lvert t + t_j \rvert)^{\sigma - 1} (1 +
  \lvert t - t_j \rvert)^{\sigma - 1}}{(4\pi h)^{\sigma - 1 + it}(1 + \lvert t
\rvert)^{\sigma - \frac{1}{2} - \frac{k}{2}}} e^{-\frac{\pi}{2}(\lvert t + t_j \rvert +
\lvert t-t_j \rvert - \lvert t \rvert)} \rho_j(h) \langle \mu_j, V \rangle.
\end{equation}
Examination of the exponential contribution shows that there is large exponential decay in
$t_j$ when $\lvert t_j \rvert > \lvert t \rvert$, so we only need to investigate the
convergence when $\lvert t_j \rvert \leq \lvert t \rvert$.
Then~\eqref{eq:discrete_newspectrum_stepI} is bounded by
\begin{equation}
  O_h\Big( \sum_{\lvert t_j \rvert \sim T} (1 + \lvert t \rvert)^{\frac{k}{2} -
  \frac{1}{2}} e^{-\frac{\pi}{2} \lvert t \rvert} \rho_j(h) \langle \mu_j, V \rangle
\Big),
\end{equation}
which, by Lemma~\ref{lem:discrete_newspectrum_innerproduct_bound}, is bounded by
\begin{equation}
  O_{h, \epsilon} \Big( (1 + \lvert t \rvert)^{\frac{7}{2}k + \frac{15}{2} + \epsilon}
  e^{-\frac{\pi}{2}\lvert t \rvert}\Big).
\end{equation}
Summing dyadically gives the first part of the proposition.

For the second part of the proposition, recall the inner product from
Lemma~\ref{lem:inner_product_list}
\begin{equation}
  \langle P_h^k(\cdot, s), \mu_{j, \ell} \rangle =
  \frac{\overline{\rho_{j,\ell}(h)}}{(4\pi h)}^{s-1} \frac{\Gamma(s + \frac{\ell}{2} - 1)
  \Gamma(s - \frac{\ell}{2})}{\Gamma(s - \frac{k}{2})}.
\end{equation}
As the sum over $j$ and $\ell$ are each finite for the old discrete spectrum, the analytic
properties in $s$ can be read directly from the inner products $\langle P_h^k, \mu_{j,
\ell} \rangle$.
The leading poles all come from the Gamma function $\Gamma(s - \frac{\ell}{2})$ in the
numerator.
Note that the first apparent pole is at $s = \frac{k}{2}$ is cancelled by the Gamma
function in the denominator, but there are poles at $s = \frac{k}{2} - 1 - m$ for $m \in
\mathbb{Z}_{\geq 0}$.

\subsubsection*{Residual Spectrum}
\index{spectrum!residual part}

Consider the residual spectrum appearing in line~\eqref{line:Ph_spectral_residual}, which
we rewrite for convenience:
\begin{equation}
  \sum_{\mathfrak{a}} \langle P_h^k(\cdot, s), R^k_{\mathfrak{a}} \rangle \langle
  R^k_{\mathfrak{a}}, V \rangle.
\end{equation}
The residual spectrum only occurs when $k$ is a half-integer.
For each cusp, $\langle R^k_{\mathfrak{a}}, V \rangle$ evaluates to some constant and
doesn't affect the analysis in $s$.
Referring to Lemma~\ref{lem:inner_product_list}, we see that
\begin{equation}
  \langle P_h^k(\cdot, s), R_{\mathfrak{a}} \rangle = \frac{\Res_{w = \frac{3}{4}}
  \overline{\rho_\mathfrak{a}^k(h,w)}}{(4\pi h)^{s-1}} \frac{\Gamma(s - \frac{1}{4})
  \Gamma(s - \frac{3}{4})}{\Gamma(s - \frac{k}{2})}.
\end{equation}
As described in~\eqref{eq:Eisenstein_halfweight_generalshape_coeffs}, the residue in
$\rho_\mathfrak{a}^k(h,w)$ comes from a potential pole in
\begin{equation}
  L\Big(2w - \tfrac{1}{2}, \big( \tfrac{h (-1)^{k - \frac{1}{2}}}{\cdot}\big) \Big)
\end{equation}
at $w = \frac{3}{4}$.
This $L$-function has a pole if and only if the character is trivial, which occurs if and
only if $h$ is a square, and $k = \frac{1}{2} + 2m$ for some $m \in \mathbb{Z}_{\geq 0}$.

Therefore, if $h$ is not a square or if $k$ is not of the form $\frac{1}{2} + 2m$, then
the residual spectrum vanishes.
If $h$ is a square and $k = \frac{1}{2} + 2m$, then the residue is nonzero, and the
analytic properties of the residual spectrum can be read directly from
\begin{equation}\label{eq:hyperboloid:residual_analytic_descriptor}
  \frac{\Gamma(s - \frac{1}{4}) \Gamma(s - \frac{3}{4})}{(4\pi h)^{s-1} \Gamma(s -
  \frac{k}{2})}.
\end{equation}
We codify this in a proposition.

\begin{proposition}\label{prop:hyperboloid_residual_props}
  The residual spectrum in line~\eqref{line:Ph_spectral_residual} vanishes unless $h$ is a
  square and $k = \frac{1}{2} + 2m$ for some $m \in \mathbb{Z}_{\geq 0}$, in which case
  the residual spectrum has meromorphic continuation to the plane and is analytic for $\Re
  s > 0$, except possibly for poles at $s = \frac{3}{4}$ and $s = \frac{1}{4}$.
\end{proposition}

\subsubsection*{Continuous Spectrum}
\index{spectrum!continuous part}

Consider the continuous spectrum appearing in line~\eqref{line:Ph_spectral_continuous},
which we rewrite for convenience:
\begin{equation}
  \frac{1}{4\pi i} \sum_{\mathfrak{a}} \int_{(\frac{1}{2})} \langle P_h^k(\cdot, s),
  E^k_{\mathfrak{a}}(\cdot, u) \rangle \langle E^k_{\mathfrak{a}}(\cdot, u), V \rangle du.
\end{equation}
Referring to Lemma~\ref{lem:inner_product_list}, we see that
\begin{equation}
  \langle P_h(\cdot, s), E_{\mathfrak{a}}^k(\cdot, \tfrac{1}{2} + it)\rangle =
  \frac{\overline{\rho_\mathfrak{a}(h, \tfrac{1}{2} + it)}}{(4\pi h)^{s-1}} \frac{\Gamma(s
  - \tfrac{1}{2} + it) \Gamma(s - \tfrac{1}{2} - it)}{\Gamma(s - \frac{k}{2})}.
\end{equation}
The Gamma functions can be approximated through Stirling's Approximation.
Each $\rho_\mathfrak{a}(h, \tfrac{1}{2} + it)$ can understood through the
Phragm\'en-Lindel\"of Principle to satisfy the bound $\rho_\mathfrak{a}(h, \frac{1}{2} +
it) \ll (1 + \lvert t \rvert)^{\frac{1}{4} - \frac{k}{2} + \epsilon}$.

The first apparent poles in $s$ can be read from the Gamma functions, and it is clear that
there are no poles in $s$ for $\Re s > \frac{1}{2}$.

The other inner products, $\langle E_\mathfrak{a}^k(\cdot, u), V \rangle$, can be
understood through Zagier normalization~\cite{ZagierRankinSelberg}.
In particular, since $V = c y^{\frac{1-k}{2}} + O(y^{-N})$ for some constant $c$ and any
$N \geq 0$ as $y \to \infty$ (and more generally, at each cusp), Zagier normalization
allows us to identify
\begin{equation}\label{eq:continuous_stepI}
  \langle E_\infty^k(\cdot, u), V \rangle = \int_0^\infty \widetilde{V}_0(y) y^{u-1}
  \frac{dy}{y} = \frac{\Gamma(s + \frac{k+1}{2} - 1)}{(4\pi)^{u + \frac{k+1}{2} - 1}}
  \sum_{n \geq 1} \frac{r_{2k+1}(n)r_1(n)}{n^{u + \frac{k+1}{2} - 1}}
\end{equation}
for $\frac{1-k}{2} < \Re u < \frac{k+1}{2}$ and give meromorphic continuation to the plane.
Here, $\widetilde{V}_0$ is the $0$th Fourier coefficient of $\widetilde{V} = \theta^{2k+1}
\overline{\theta}$, which was first defined in \S\ref{sec:hyperboloid_introduction}.
Notice that the region of identification includes $\Re u = \frac{1}{2}$, as is necessary
for this application.
Similar expressions exist at the cusps $0$ and $\frac{1}{2}$.

\begin{remark}
  Zagier normalization also gives that the Dirichlet series at the right
  in~\eqref{eq:continuous_stepI} has a potential pole at $u = \frac{k+1}{2}$, which agrees
  with on-average estimates.
  The function $r_1(n)$ is essentially a square indicator function, so the Dirichlet
  series can be rewritten as
  \begin{equation}
    2\sum_{n \geq 1} \frac{r_{2k+1}(n^2)}{n^{2(u + \frac{k+1}{2} - 1)}},
  \end{equation}
  in which it is straightforward to use the Gaussian heuristic to confirm the pole from
  Zagier normalization.
\end{remark}

Using the functional equation of the Eisenstein series to give the functional equation of
the Dirichlet series in~\eqref{eq:continuous_stepI} and applying the Phragm\'en-Lindel\"of
Convexity Principle guarantees that
\begin{equation}
  \langle E_\mathfrak{a}^k(\cdot, \tfrac{1}{2} + it), V \rangle \ll (1 + \lvert t
  \rvert)^{2k - 1 + \epsilon}.
\end{equation}
Denote $\Re s = \sigma$, and suppose $\tfrac{1}{2} < \sigma < 1$.
Applying Stirling's Approximation and the bounds above, we can now estimate
\begin{align}
  &\frac{1}{2\pi i} \int_{(\frac{1}{2})} \langle P_h^k(\cdot, s),
E^k_{\mathfrak{a}}(\cdot, u) \rangle \langle E^k_{\mathfrak{a}}(\cdot, u), V \rangle du \\
  &\quad \ll \int_{-\infty}^\infty
  \frac{(1 + \lvert t \rvert)^{(\frac{1}{4} - \frac{k}{2} + \epsilon) + (2k - 1 +
  \epsilon)}}{(4\pi h)^{\sigma - 1}}
  \frac{(1 + \lvert s + t \rvert)^{\sigma - 1} (1 + \lvert s - t \rvert)^{\sigma - 1}}{(1
  + \lvert s \rvert)^{\sigma - \frac{1}{2} - \frac{k}{2}}}
  e^{-\frac{\pi}{2}(\lvert s - t \rvert + \lvert s + t \rvert - \lvert s \rvert)}dt.
\end{align}
When $\lvert t \rvert > \lvert s \rvert$, there is significant exponential decay in $t$,
effectively cutting off the integral to the interval $\lvert t \rvert \leq \lvert s
\rvert$.
Within this interval, the integral can be bounded by
\begin{equation}
  \int_{-\lvert s \rvert}^{\lvert s \rvert} (1 + \lvert t \rvert)^{\frac{3}{2}k -
  \frac{3}{4} + \epsilon} (1 + \lvert s \rvert)^{\frac{k}{2} - \frac{1}{2}}
  e^{-\frac{\pi}{2}\lvert s \rvert} dt \ll (1 + \lvert s \rvert)^{2k - \frac{1}{4} +
  \epsilon} e^{-\frac{\pi}{2}\lvert s \rvert}.
\end{equation}

\begin{proposition}\label{prop:hyperboloid_continuous_props}
  The continuous spectrum in line~\eqref{line:Ph_spectral_continuous} has meromorphic
  continuation to the plane and is analytic for $\Re s > \frac{1}{2}$ and has apparent
  poles at $\Re s = \frac{1}{2}$.
  For $\frac{1}{2} < \Re s < 1$, the continuous spectrum satisfies the bound
  \begin{equation}
    \frac{1}{4\pi i} \sum_{\mathfrak{a}} \int_{(\frac{1}{2})} \langle P_h^k(\cdot, s),
    E^k_{\mathfrak{a}}(\cdot, u) \rangle \langle E^k_{\mathfrak{a}}(\cdot, u), V \rangle
    du \ll (1 + \lvert s \rvert)^{2k - \frac{1}{4} + \epsilon} e^{-\frac{\pi}{2}\lvert s
    \rvert}.
  \end{equation}
\end{proposition}

\subsection{Analytic Behavior of $D_h^k(s)$}

We are now ready to describe the analytic behavior of $D_h^k(s)$ for each $k \geq
\frac{1}{2}$, for $\Re s > \frac{1}{2}$.
Recall that $D_h^k(s)$ is defined in~\eqref{eq:hyp:Dh_def} and~\eqref{eq:hyp:Dh_def_half}
as
\begin{equation}
  D_h^k(s) = \frac{(2\pi)^{s + \frac{k-1}{2}}}{\Gamma(s + \frac{k-1}{2})} \langle
  P_h^{\frac{1}{2}}(\cdot, s), V \rangle.
\end{equation}
In Propositions~\ref{prop:nonspectral_analytic_props_large_dim} (for $k \geq 1$)
and~\ref{prop:nonspectral_analytic_props_dim_3} (for $k = \frac{1}{2}$), it was shown that
\begin{equation}\label{eq:hyperboloid_analytic_middlestep_I}
  D_h^k(s) = \sum_{m \in \mathbb{Z}} \frac{r_{2k+1}(m^2 + h)}{(2m^2 + h)^{s +
  \frac{k-1}{2}}} - \mathfrak{E}_h^k(s)
\end{equation}
and the analytic properties of $\mathfrak{E}_h^k(s)$ are described.
On the other hand, through the Spectral Expansion of $P_h^k$, we also have an expression
for $\langle P_h^k(\cdot, s), V \rangle$, given
in~\eqref{line:Ph_spectral_discrete}--\eqref{line:Ph_spectral_continuous}.
Multiplying by $(2\pi)^{s + \frac{k-1}{2}} \Gamma(s + \frac{k-1}{2})^{-1}$ and
rearranging~\eqref{eq:hyperboloid_analytic_middlestep_I}, we have
\begin{equation}\label{eq:hyperboloid_dirichlet_decomposition}
  \begin{split}
    &\sum_{m \in \mathbb{Z}} \frac{r_{2k+1}(m^2 + h)}{(2m^2 + h)^{s + \frac{k-1}{2}}} =
    \mathfrak{E}_h^k(s) \\
    &\qquad + \frac{(2\pi)^{s + \frac{k-1}{2}}}{\Gamma(s + \frac{k-1}{2})} \bigg( \sum_j
  \langle P_h^k, \mu_j \rangle \langle \mu_j, V \rangle + \sum_{\ell,j} \langle P_h^k,
\mu_{j,\ell} \rangle \langle \mu_{j, \ell}, V \rangle \bigg) \\
    &\qquad + \frac{(2\pi)^{s + \frac{k-1}{2}}}{\Gamma(s + \frac{k-1}{2})}
\sum_\mathfrak{a} \langle P_h^k, R_\mathfrak{a}\rangle \langle R_\mathfrak{a}, V \rangle
\\
    &\qquad + \frac{(2\pi)^{s + \frac{k-1}{2}}}{\Gamma(s + \frac{k-1}{2})} \frac{1}{4\pi
i} \sum_{\mathfrak{a}} \int_{(\frac{1}{2})} \langle P_h^k, E_\mathfrak{a}^k\rangle \langle
E_\mathfrak{a}^k, V\rangle,
  \end{split}
\end{equation}
where the lines are separated into the Eisenstein correction factors in $V$, the discrete
spectrum, the residual spectrum, and the continuous spectrum, respectively.
The analytic properties of the discrete, residual, and continuous spectra are described in
Propositions~\ref{prop:hyperboloid_discrete_props},~\ref{prop:hyperboloid_residual_props},
and~\ref{prop:hyperboloid_continuous_props}, respectively.
Assembling these propositions together, we have proved the following theorem.

\begin{theorem}\label{thm:hyperboloid:mero_summary}
  Let $h \geq 1$ be an integer and $k \geq \frac{1}{2}$ be either an integer or a
  half-integer.
  The Dirichlet series
  \begin{equation}
    \sum_{m \in \mathbb{Z}} \frac{r_{2k+1}(m^2 + h)}{(2m^2 + h)^{s + \frac{k-1}{2}}}
  \end{equation}
  has meromorphic continuation to the plane, and is analytic for $\Re s > \frac{1}{2}$
  except for
  \begin{itemize}
    \item simple poles at $s = \frac{k}{2} - 1 - m$ for $m \in \mathbb{Z}_{\geq 0}$,
      coming from the discrete spectrum,
    \item simple poles at $s = \frac{k+1}{2} - m$ for $m \in \mathbb{Z}_{\geq 0}$, coming
      from the Eisenstein correction factors $\mathfrak{E}_h^k(s)$, and
    \item a double pole at $s = \frac{3}{4}$ when $k = \frac{1}{2}$ and $h$ is a square,
      also coming from $\mathfrak{E}_h^k(s)$.
  \end{itemize}
\end{theorem}

\section{Integral Analysis}\label{sec:hyp:integral_analysis}

We are now ready to perform the main integral analysis on $\langle P_h^k(\cdot, s),
V\rangle$.
In this section, we handle the $k > \frac{1}{2}$ case.
We will examine
\begin{equation}
  \frac{1}{2\pi i} \int_{(\sigma)} \sum_{n \in \mathbb{Z}} \frac{r_{2k+1}(n^2 + h)}{(2n^2 +
  h)^{s + \frac{k-1}{2}}} X^{s+\frac{k-1}{2}} V_Y(s) ds,
\end{equation}
which is closely related to studying
\begin{equation}
  \frac{1}{2\pi i} \int_{(\sigma)} \langle P_h^k(\cdot, s), V \rangle \frac{(2\pi)^{s +
  \frac{k-1}{2}}}{\Gamma(s + \frac{k-1}{2})} X^{s+\frac{k-1}{2}} V_Y(s) ds.
\end{equation}
We will use three integral kernels $V_Y(s)$ described in Chapter~\ref{c:background}.
We denote the Mellin transform of $V_Y(s)$ by $v_Y(x)$ when appropriate.

From the decomposition in~\eqref{eq:hyperboloid_dirichlet_decomposition}, it will suffice
to consider the integral transforms applied to $\mathfrak{E}_h^k$, the discrete spectrum,
the residual spectrum, and the continuous spectrum separately.
In each integral, we will shift the line of integration to $\frac{1}{2} + \epsilon$ for a
small $\epsilon > 0$ and analyze the poles and residues.

\subsection{Integral Analysis of $\mathfrak{E}_h^k$}

We first study
\begin{equation}
  \frac{1}{2\pi i} \int_{(\sigma)} \mathfrak{E}_h^k(s) X^{s + \frac{k-1}{2}} V_Y(s) ds.
\end{equation}
From Theorem~\ref{thm:hyperboloid:mero_summary} and
Propositions~\ref{prop:nonspectral_analytic_props_large_dim}
and~\ref{prop:nonspectral_analytic_props_dim_3}, we recognize that $\mathfrak{E}_h^k$ has
poles at $s = \frac{k+1}{2} - m$ for $m \in \mathbb{Z}_{\geq 0}$.
All of these poles are simple, except when $k = \frac{1}{2}$ and $h$ is a square, in which
case the leading pole at $s = \frac{k+1}{2} = \frac{3}{4}$ is a double pole.
As $\mathfrak{E}_h^k$ is of moderate growth for $\Re s > \frac{1}{2}$ and each kernel
$V_Y(s)$ is of rapid decay, we may shift the line of integration to $\frac{1}{2} +
\epsilon$, and by Cauchy's Theorem we have
\begin{equation}\label{eq:hyp:integral_MT}
  \begin{split}
    &\frac{1}{2\pi i} \int_{(\sigma)} \mathfrak{E}_h^k(s) X^{s + \frac{k-1}{2}} V_Y(s) ds \\
    &\quad = \frac{1}{2\pi i} \int_{(\frac{1}{2} + \epsilon)} \mathfrak{E}_h^k(s) X^{s +
    \frac{k-1}{2}} V_Y(s) ds + \sum_{0 \leq m < \frac{k}{2}} R_{k-m,h}^k X^{k - m}
    V_Y(\tfrac{k+1}{2} - m) \\
    &\qquad + \delta_{[k=\frac{1}{2}]} \delta_{[h = a^2]} \bigg( R'_h X^{\frac{1}{2}} \log X
  V_Y(\tfrac{3}{4}) + R'_h X^{\frac{1}{2}} V'_Y(\tfrac{3}{4}) \bigg)
  \end{split}
\end{equation}
where
\begin{equation}
  R_{k-m,h}^k = \Res_{s = \frac{k+1}{2} - m} \mathfrak{E}_h^k(s)
\end{equation}
and where $R'_h$ is the coefficient of $(s - \tfrac{3}{4})^{-2}$ in the Laurent expansion
of $\mathfrak{E}_h^{\frac{1}{2}}(s)$.
The Kronecker $\delta$ symbol is used here to mean
\begin{equation}
  \delta_{[k = \frac{1}{2}]} \delta_{[h = a^2]} = \begin{cases}
    1 & \text{if } k = \frac{1}{2} \; \text{and } h \; \text{is a square}, \\
    0 & \text{otherwise}.
  \end{cases}
\end{equation}

To estimate the shifted integral, recall from
Propositions~\ref{prop:nonspectral_analytic_props_large_dim}
and~\ref{prop:nonspectral_analytic_props_dim_3} that for $\Re s > \frac{1}{2}$, we know
the bound $\mathfrak{E}_h^k(s) \ll (1 + \lvert s \rvert)^{\frac{1}{2}}$.
Therefore as long as $V_Y(s) \ll (1 + \lvert s \rvert)^{-\frac{3}{2} - \epsilon}$, the
integral converges absolutely.
With respect to the three integral transforms, this means that
\begin{align}
  \frac{1}{2\pi i} &\int_{(\frac{1}{2} + \epsilon)} \mathfrak{E}_h^k(s) X^{s +
  \frac{k-1}{2}} \Gamma(s + \tfrac{k-1}{2}) ds \ll X^{\frac{k}{2} + \epsilon}
  \label{eq:hyperboloid:nonspectral_gamma_integral_bound} \\
  \frac{1}{2\pi i} &\int_{(\frac{1}{2} + \epsilon)} \mathfrak{E}_h^k(s) X^{s +
  \frac{k-1}{2}} \frac{e^{\pi s^2 / y^2}}{y} ds  \ll X^{\frac{k}{2} + \epsilon}
  Y^{\frac{1}{2}} \label{eq:hyperboloid:nonspectral_concentrating_integral_bound}\\
  \frac{1}{2\pi i} &\int_{(\frac{1}{2} + \epsilon)} \mathfrak{E}_h^k(s) X^{s +
  \frac{k-1}{2}} \Phi_Y(s) ds  \ll X^{\frac{k}{2} + \epsilon} Y^{\frac{1}{2}+\epsilon}.
  \label{eq:hyperboloid:nonspectral_compact_integral_bound}
\end{align}

\subsection{Integral analysis of the discrete spectrum}

We now study the integral of the discrete spectrum.
To condense notation, we introduce the notation
\begin{equation}
  \discrete_h^k(s) := \frac{(2\pi)^{s + \frac{k-1}{2}}}{\Gamma(s + \frac{k-1}{2})} \bigg(
  \sum_j \langle P_h^k, \mu_j \rangle \langle \mu_j, V \rangle + \sum_{\ell,j} \langle
P_h^k, \mu_{j,\ell} \rangle \langle \mu_{j, \ell}, V \rangle \bigg).
\end{equation}
Then the integral of the discrete spectrum can be written as
\begin{equation}\label{eq:integral_analysis_discrete_I}
  \frac{1}{2\pi i} \int_{(\sigma)} \discrete_h^k(s) X^{s + \frac{k-1}{2}} V_Y(s) ds.
\end{equation}
From Theorem~\ref{thm:hyperboloid:mero_summary} we recognize that the integrand has simple
poles at $s = \frac{k}{2} - m$ for $m \in \mathbb{Z}_{\geq 0}$, coming from the finitely
many terms of the ``old'' discrete spectrum.
Therefore, shifting the line of integration to $\frac{1}{2} + \epsilon$ and applying
Cauchy's Theorem shows that~\eqref{eq:integral_analysis_discrete_I} is equal to
\begin{equation}
  \sum_{0 \leq m \leq \frac{k-1}{2}} R^k_{k - \frac{1}{2} - m, h} X^{k - \frac{1}{2} - m}
  V(\tfrac{k+1}{2} - m - \tfrac{1}{2}) + \frac{1}{2\pi i} \int_{(\frac{1}{2} + \epsilon)}
  \discrete_h^k(s) X^{s + \frac{k-1}{2}} V_Y(s) ds,
\end{equation}
where $R^k_{k - \frac{1}{2} - m}$ are the collected residues of the old discrete spectrum
at $s = \frac{k}{2} - m$.

To estimate the shifted integral, note that from applying Stirling's Approximation to
$\Gamma(s+\frac{k-1}{2})^{-1}$ and using the approximations from
Proposition~\ref{prop:hyperboloid_discrete_props} that $\discrete_h^k(s) \ll (1 + \lvert s
\rvert)^{3k + \frac{17}{2} +\epsilon}$.
Therefore as long as $V_Y(s) \ll (1 + \lvert s \rvert)^{-3k - \frac{19}{2} - 2\epsilon}$,
the integral converges absolutely.
The three integral transforms then give
\begin{align}
  &\frac{1}{2\pi i} \int_{(\frac{1}{2} + \epsilon)} \discrete_h^k(s) X^{s + \frac{k-1}{2}}
  \Gamma(s + \frac{k-1}{2})  ds \ll X^{\frac{k}{2} + \epsilon}
  \label{eq:hyperboloid:discrete_gamma_integral_bound}\\
  &\frac{1}{2\pi i} \int_{(\frac{1}{2} + \epsilon)} \discrete_h^k(s) X^{s + \frac{k-1}{2}}
  \frac{e^\pi s^2/Y^2}{Y} ds \ll X^{\frac{k}{2} + \epsilon} Y^{3k + \frac{17}{2} + \epsilon}
  \label{eq:hyperboloid:discrete_concentrating_integral_bound} \\
  &\frac{1}{2\pi i} \int_{(\frac{1}{2} + \epsilon)} \discrete_h^k(s) X^{s + \frac{k-1}{2}}
  \Phi_Y(s) ds \ll X^{\frac{k}{2}+ \epsilon} Y^{3k + \frac{17}{2} + 2\epsilon}.
  \label{eq:hyperboloid:discrete_compact_integral_bound}
\end{align}

\subsection{Integral analysis of the residual spectrum}

We now study the integral of the residual spectrum,
\begin{equation}\label{eq:hyperboloid:integral_analysis_residual_I}
  \frac{1}{2\pi i} \int_{(\sigma)} \frac{(2\pi)^{s + \frac{k-1}{2}}}{\Gamma(s +
    \frac{k-1}{2})} \sum_\mathfrak{a} \langle P_h^k, R^k_\mathfrak{a}\rangle \langle
  R^k_\mathfrak{a}, V \rangle X^{s + \frac{k-1}{2}} V_Y(s) ds.
\end{equation}
From Theorem~\ref{thm:hyperboloid:mero_summary},
Proposition~\ref{prop:hyperboloid_residual_props}, and the analytic properties of the
residual spectrum as described in~\eqref{eq:hyperboloid:residual_analytic_descriptor}, the
residual spectrum is analytic for $\Re s > \frac{1}{2} + \epsilon$ and is bounded by $O((1
+ \lvert s \rvert)^{\frac{1}{2}})$.
By Cauchy's Theorem the integral~\eqref{eq:hyperboloid:integral_analysis_residual_I} is
equal to the shifted integral
\begin{equation}
  \frac{1}{2\pi i} \int_{(\frac{1}{2} + \epsilon)} \frac{(2\pi)^{s +
  \frac{k-1}{2}}}{\Gamma(s + \frac{k-1}{2})} \sum_\mathfrak{a} \langle P_h^k,
  R_\mathfrak{a}\rangle \langle R^k_\mathfrak{a}, V \rangle X^{s + \frac{k-1}{2}} V_Y(s) ds.
\end{equation}
As long as $V_Y(s) \ll (1 + \lvert s \rvert)^{-\frac{3}{2} + \epsilon}$, this converges
absolutely.
The three integral transforms then satisfy the bounds
\begin{align}
  &\frac{1}{2\pi i} \int_{(\frac{1}{2} + \epsilon)} \frac{(2\pi)^{s +
  \frac{k-1}{2}}}{\Gamma(s + \frac{k-1}{2})} \sum_\mathfrak{a} \langle P_h^k,
    R^k_\mathfrak{a}\rangle \langle R^k_\mathfrak{a}, V \rangle X^{s + \frac{k-1}{2}} \Gamma(s
    + \tfrac{k-1}{2}) ds \ll X^{\frac{k}{2} + \epsilon}
  \label{eq:hyperboloid:residual_gamma_integral_bound} \\
  &\frac{1}{2\pi i} \int_{(\frac{1}{2} + \epsilon)} \frac{(2\pi)^{s +
  \frac{k-1}{2}}}{\Gamma(s + \frac{k-1}{2})} \sum_\mathfrak{a} \langle P_h^k,
  R^k_\mathfrak{a}\rangle \langle R^k_\mathfrak{a}, V \rangle X^{s + \frac{k-1}{2}}
  \frac{e^{\pi s^2/Y^2}}{Y} ds \ll X^{\frac{k}{2} + \epsilon} Y^{\frac{1}{2} + \epsilon}
  \label{eq:hyperboloid:residual_concentrating_integral_bound} \\
  &\frac{1}{2\pi i} \int_{(\frac{1}{2} + \epsilon)} \frac{(2\pi)^{s +
  \frac{k-1}{2}}}{\Gamma(s + \frac{k-1}{2})} \sum_\mathfrak{a} \langle P_h^k,
  R^k_\mathfrak{a}\rangle \langle R^k_\mathfrak{a}, V \rangle X^{s + \frac{k-1}{2}} \Phi_Y(s)
  ds \ll X^{\frac{k}{2} + \epsilon} Y^{\frac{1}{2} + 2\epsilon}.
  \label{eq:hyperboloid:residual_compact_integral_bound}
\end{align}

\subsection{Integral analysis of the continuous spectrum}

We now consider the last integral, the integral of the continuous spectrum:
\begin{equation}\label{eq:hyperboloid:integral_analysis_continuous_I}
  \frac{1}{2\pi i} \int_{(\sigma)} \frac{(2\pi)^{s + \frac{k-1}{2}}}{\Gamma(s +
  \frac{k-1}{2})} \frac{1}{4\pi i} \sum_{\mathfrak{a}} \int_{(\frac{1}{2})} \langle P_h^k,
  E_\mathfrak{a}^k\rangle \langle E_\mathfrak{a}^k, V\rangle X^{s+\frac{k-1}{2}} V_Y(s)
  du \, ds.
\end{equation}
Recall that we use $u$ to denote the variable within the Eisenstein series
$E_\mathfrak{a}^k(z,u)$, though we omit this from the notation.
By Theorem~\ref{thm:hyperboloid:mero_summary}, we know that the continuous spectrum is
analytic for $\Re s > \frac{1}{2}$.
By Cauchy's Theorem, the integral~\eqref{eq:hyperboloid:integral_analysis_continuous_I}
is equal to the shifted integral
\begin{equation}
  \frac{1}{2\pi i} \int_{(\frac{1}{2} + \epsilon)} \frac{(2\pi)^{s +
  \frac{k-1}{2}}}{\Gamma(s + \frac{k-1}{2})} \frac{1}{4\pi i} \sum_{\mathfrak{a}}
  \int_{(\frac{1}{2})} \langle P_h^k, E_\mathfrak{a}^k\rangle \langle E_\mathfrak{a}^k,
  V\rangle X^{s+\frac{k-1}{2}} V_Y(s) du \, ds.
\end{equation}
From Stirling's Approximation applied to $\Gamma(s + \frac{k-1}{2})^{-1}$ and the bounds
from Proposition~\ref{prop:hyperboloid_continuous_props}, as long as $V_Y(s) \ll (1 +
\lvert s \rvert)^{-\frac{3}{2}k - \frac{7}{4} - 2\epsilon}$, the shifted integral
converges absolutely.
The three integral transforms thus satisfy the bounds
\begin{align}
  \begin{split}
    &\frac{1}{2\pi i} \int_{(\frac{1}{2} + \epsilon)} \frac{(2\pi)^{s +
    \frac{k-1}{2}}}{\Gamma(s + \frac{k-1}{2})} \frac{1}{4\pi i} \\
    &\quad \times \sum_{\mathfrak{a}} \int_{(\frac{1}{2})} \langle P_h^k,
    E_\mathfrak{a}^k\rangle \langle
    E_\mathfrak{a}^k, V\rangle X^{s+\frac{k-1}{2}} \Gamma(s + \tfrac{k-1}{2}) du \, ds \ll
    X^{\frac{k}{2} + \epsilon} \label{eq:hyperboloid:continuous_gamma_integral_bound}
  \end{split}
  \\
  \begin{split}
    &\frac{1}{2\pi i} \int_{(\frac{1}{2} + \epsilon)} \frac{(2\pi)^{s +
    \frac{k-1}{2}}}{\Gamma(s + \frac{k-1}{2})} \frac{1}{4\pi i} \\
    &\quad \times \sum_{\mathfrak{a}}
    \int_{(\frac{1}{2})} \langle P_h^k, E_\mathfrak{a}^k\rangle \langle E_\mathfrak{a}^k,
    V\rangle X^{s+\frac{k-1}{2}} \frac{e^{\pi s^2 / Y^2}}{Y} du\, ds \ll X^{\frac{k}{2} +
    \epsilon} Y^{\frac{3}{2}k + \frac{3}{4} + \epsilon}
    \label{eq:hyperboloid:continuous_concentrating_integral_bound}
  \end{split}
  \\
  \begin{split}
    &\frac{1}{2\pi i} \int_{(\frac{1}{2} + \epsilon)} \frac{(2\pi)^{s +
    \frac{k-1}{2}}}{\Gamma(s + \frac{k-1}{2})} \frac{1}{4\pi i} \\
    &\quad \times \sum_{\mathfrak{a}}
    \int_{(\frac{1}{2})} \langle P_h^k, E_\mathfrak{a}^k\rangle \langle E_\mathfrak{a}^k,
    V\rangle X^{s+\frac{k-1}{2}} \Phi_Y(s) du \, ds \ll X^{\frac{k}{2} + \epsilon}
    Y^{\frac{3}{2}k + \frac{3}{4} + 2\epsilon}.
    \label{eq:hyperboloid:continuous_compact_integral_bound}
  \end{split}
\end{align}

\section{Proof of Main Theorems}\label{sec:hyp:proof_main_theorems}

Using the integral analysis from the previous section, we now prove the main theorems of
this chapter.

\subsection{Smoothed Main Theorem}

Consider the integral transform
\begin{equation}
  \frac{1}{2\pi i} \int_{(\sigma)} \sum_{m \in \mathbb{Z}} \frac{r_{2k+1}{(m^s +
  h)}}{(2m^2 + h)^{s + \frac{k-1}{2}}} X^{s + \frac{k-1}{2}} \Gamma(s + \tfrac{k-1}{2})
  ds.
\end{equation}
On the one hand, by the standard properties of the transform (in
\S\ref{sec:cutoff_integrals}), this is exactly the sum
\begin{equation}
  \sum_{m \in \mathbb{Z}} r_{2k+1}(m^2 + h) e^{-(2m^2 + h)/X}.
\end{equation}
On the other hand, by~\eqref{eq:hyperboloid_dirichlet_decomposition} and the analysis in
the previous section (and in particular the bounds from
lines~\eqref{eq:hyperboloid:nonspectral_gamma_integral_bound},
\eqref{eq:hyperboloid:discrete_gamma_integral_bound},
\eqref{eq:hyperboloid:residual_gamma_integral_bound},
and~\eqref{eq:hyperboloid:continuous_gamma_integral_bound}, as well as the residual
expression in~\eqref{eq:hyp:integral_MT}), this transform is equal to
\begin{align}
  &\sum_{0 \leq m < \frac{k}{2}} R_{k-m, h}^k \Gamma(k-m) X^{k-m} + \delta_{[k =
\frac{1}{2}]} \delta_{[h = a^2]} \bigg( R'_h X^{\frac{1}{2}} \log X \Gamma(\tfrac{1}{2}) +
R'_h \Gamma'(\tfrac{1}{2}) X^{\frac{1}{2}} \bigg) \\
  &\qquad + \sum_{0 \leq m < \frac{k-1}{2}} R_{k - \frac{1}{2} - m,h}^k \Gamma(k -
  \tfrac{1}{2} - m) X^{k - \frac{1}{2} - m} + O(X^{\frac{k}{2} + \epsilon}).
\end{align}
This proves the following theorem.
\begin{theorem}\label{thm:hyperboloid:smooth_full}
  Let $k \geq \frac{1}{2}$ be a full or half-integer.
  Then
  \begin{align}
    &\sum_{m \in \mathbb{Z}} r_{2k+1}(m^2 + h) e^{-(2m^2 + h)/X} =\\
    &\quad =\delta_{[k = \frac{1}{2}]} \delta_{[h = a^2]} \bigg( R'_h X^{\frac{1}{2}} \log
    X \Gamma(\tfrac{1}{2}) + R'_h \Gamma'(\tfrac{1}{2}) X^{\frac{1}{2}} \bigg) + \sum_{0
    \leq m < \frac{k}{2}} R_{k-m, h}^k \Gamma(k-m) X^{k-m} \\
    &\qquad + \sum_{0 \leq m < \frac{k-1}{2}} R_{k - \frac{1}{2} - m,h}^k \Gamma(k -
    \tfrac{1}{2} - m) X^{k - \frac{1}{2} - m} + O(X^{\frac{k}{2} + \epsilon}).
  \end{align}
  Here, $\delta_{[\text{condition}]}$ is a Kronecker $\delta$ and evaluates to $1$ if the
  condition is true and $0$ otherwise, and the constants $R^k_{\ell, h}$ and $R'_h$ are
  residues as defined in \S\ref{sec:hyp:integral_analysis}.
\end{theorem}

\subsection{Main Theorem in Short-Intervals}

Consider the integral transform
\begin{equation}
  \frac{1}{2\pi i} \int_{(\sigma)} \sum_{m \in \mathbb{Z}} \frac{r_{2k+1}{(m^s +
  h)}}{(2m^2 + h)^{s + \frac{k-1}{2}}} X^{s + \frac{k-1}{2}} \frac{e^{\pi s^2/Y^2}}{Y} ds.
\end{equation}
On the one hand, by the properties of this integral transform as described in
\S\ref{sec:cutoff_integrals}, this is exactly the sum
\begin{equation}
  \sum_{m \in \mathbb{Z}} r_{2k+1}(m^2 + h) \exp\bigg( -\frac{Y^2
  \log^2(X/(2m^2+h))}{4\pi} \bigg).
\end{equation}
When $2m^2 + h \in [X - X/Y, X + X/Y]$, the exponential damping term is almost constant.
But for $m$ with $\lvert 2m^2 + h - X \rvert > X/Y^{1-\epsilon}$, the exponential term
contributes significant exponential decay.
Further, by the positivity of $r_{2k+1}$, we have that
\begin{equation}
  \sum_{\lvert 2m^2 + h - X \rvert < \frac{X}{Y}} r_{2k+1}(m^2 + h) \ll \frac{1}{2\pi i}
  \int_{(\sigma)} \sum_{m \in \mathbb{Z}} \frac{r_{2k+1}{(m^s + h)}}{(2m^2 + h)^{s +
  \frac{k-1}{2}}} X^{s + \frac{k-1}{2}} \frac{e^{\pi s^2/Y^2}}{Y} ds.
\end{equation}

\begin{remark}
  Roughly speaking, this should be interpreted to mean that this integral transform
  concentrates the mass of the integral on those $m$ such that $2m^2 + h$ is within a
  short interval around $X$.
\end{remark}

And on the other hand, by~\eqref{eq:hyperboloid_dirichlet_decomposition} and the analysis
in the previous section (in particular the bounds from
lines~\eqref{eq:hyperboloid:nonspectral_concentrating_integral_bound},
\eqref{eq:hyperboloid:discrete_concentrating_integral_bound},
\eqref{eq:hyperboloid:residual_concentrating_integral_bound},
and~\eqref{eq:hyperboloid:continuous_concentrating_integral_bound}, as well as the
residual expression in~\eqref{eq:hyp:integral_MT}), this transform is
equal to
\begin{align}
  &\delta_{[k = \frac{1}{2}]} \delta_{[h = a^2]} \bigg(R'_h X^{\frac{1}{2}}\log X +R'_h
X^{\frac{1}{2}} \frac{\pi}{Y} \bigg)\frac{\exp(\frac{\pi}{4Y})}{Y} + R_k^k X^k
\frac{\exp(\frac{\pi k^2}{Y})}{Y} \\
  &\quad + O(\frac{X^{k-\frac{1}{2}}}{Y})
  + O(X^{\frac{k}{2} + \epsilon} Y^{3k + \frac{17}{2} + \epsilon}).
\end{align}
In this expression, we only kept the leading poles.
As we are only seeking to create an upper bound, we simplify the above expression into the
bound
\begin{equation}\label{eq:hyp:short_interval_step_I}
  O(\frac{X^{k+\epsilon}}{Y})
  + O(X^{\frac{k}{2}+\epsilon} Y^{3k + \frac{17}{2} + \epsilon}).
\end{equation}

We choose $Y$ to balance the expressions in~\eqref{eq:hyp:short_interval_step_I}.
The two terms are balaned when $Y = X^{1/(6 + \frac{19}{k} + \frac{\epsilon}{k})}$.
This gives the overall bound
\begin{equation}\label{eq:hyp:short_interval_bound_largek}
  O(X^{k + \epsilon - \lambda(k)}), \qquad \text{where } \lambda(k) = \frac{1}{6 +
  \frac{19}{k}}.
\end{equation}

We have now shown that
\begin{equation}
  \sum_{\lvert 2m^2 + h - X \rvert < X^{1 + \epsilon - \lambda(k)}} r_{2k+1}(m^2 + h) \ll
  X^{k + \epsilon - \lambda(k)},
\end{equation}
where $\lambda(k)$ is defined by~\eqref{eq:hyp:short_interval_bound_largek}.
This is the content of the second main theorem in this chapter.
\begin{theorem}\label{thm:hyp:concentrating_theorem_full}
  Let $k \geq \frac{1}{2}$ be a full or half-integer.
  Then
  \begin{equation}
    \sum_{\lvert 2m^2 + h - X \rvert < X^{1 + \epsilon - \lambda(k)}} r_{2k+1}(m^2 + h)
    \ll X^{k + \epsilon - \lambda(k)},
  \end{equation}
  where $\lambda(k)$ is defined as
  \begin{equation}
    \lambda(k) = \frac{1}{6 + \frac{19}{k}}.
  \end{equation}
\end{theorem}

\subsection{Sharp Main Theorem}

Finally, consider the integral transform
\begin{equation}
  \frac{1}{2\pi i} \int_{(\sigma)} \sum_{m \in \mathbb{Z}} \frac{r_{2k+1}{(m^s +
  h)}}{(2m^2 + h)^{s + \frac{k-1}{2}}} X^{s + \frac{k-1}{2}} \Phi_Y(s) ds.
\end{equation}
On the one hand, by the description of this integral transform in
\S\ref{sec:cutoff_integrals}, this is exactly the sum
\begin{equation}\label{eq:hyp:sharp_main_estimate_I}
  \sum_{\lvert 2m^2 + h \rvert \leq X} r_{2k+1}(m^2 + h) + \sum_{X \leq \lvert 2m^2+h
  \rvert \leq X+\frac{X}{Y}} r_{2k+1}(m^2+h) \phi_Y(\tfrac{2m^2+h}{X}).
\end{equation}
As $r_{2k+1}$ is always positive, we can bound the second term above by
\begin{equation}\label{eq:hyp:sharp_main_estimate_II}
  \sum_{X \leq \lvert 2m^2+h \rvert \leq X+\frac{X}{Y}} r_{2k+1}(m^2+h)
  \phi_Y(\tfrac{2m^2+h}{X}) \ll \sum_{\lvert 2m^2 + h - X \rvert \leq \frac{X}{Y}}
  r_{2k+1}(m^2 + h).
\end{equation}
Notice that this is a short-interval type estimate, exactly as considered in
Theorem~\ref{thm:hyp:concentrating_theorem_full}.

On the other hand, by~\eqref{eq:hyperboloid_dirichlet_decomposition} and the analysis in
the previous section (and in particular the bounds from
lines~\eqref{eq:hyperboloid:nonspectral_compact_integral_bound},
\eqref{eq:hyperboloid:discrete_compact_integral_bound},
\eqref{eq:hyperboloid:residual_compact_integral_bound},
and~\eqref{eq:hyperboloid:continuous_compact_integral_bound}, as well as the residual
expression in~\eqref{eq:hyp:integral_MT}), this integral
transform is equal to
\begin{align}
  &\sum_{0 \leq m < \frac{k}{2}} R_{k-m, h}^k X^{k-m} \big( \tfrac{1}{k-m} +
O(\tfrac{1}{Y}) \big) \\
  &\quad + \delta_{[k = \frac{1}{2}]} \delta_{[h = a^2]} \bigg( R'_h X^{\frac{1}{2}} \log
X \big(2 + O(\tfrac{1}{Y}) \big) + R'_h X^{\frac{1}{2}} \big( -4 + O(\tfrac{1}{Y}) \big)
\bigg) \\
  &\quad + \sum_{0 \leq m < \frac{k-1}{2}} R_{k - \frac{1}{2} - m,h}^k X^{k - \frac{1}{2}
- m} \big( \tfrac{1}{k - \frac{1}{2} - m} + O(\tfrac{1}{Y})\big) \\
  &\quad
  + O(X^{\frac{k}{2} + \epsilon} Y^{\frac{1}{2} + \epsilon})
  + O(X^{\frac{k}{2} + \epsilon} Y^{3k + \frac{17}{2} + 2\epsilon})
  + O(X^{\frac{k}{2} + \epsilon} Y^{\frac{3}{2}k + \frac{3}{4} + 2\epsilon}).
\end{align}
We note that we have used that $\Phi_Y(s) = \frac{1}{s} + O(\frac{1}{Y})$ and
$\Phi_Y'(s) = -\frac{1}{s^2} + O(\frac{1}{Y})$ from \S\ref{sec:cutoff_integrals}
to simplify the residual terms involving the weight function $\Phi_Y$.
These contribute to the error terms in this expression.

Keeping only the leading terms of growth, we rewrite this as
\begin{align}
  &\delta_{[k = \frac{1}{2}]} \delta_{[h = a^2]} \bigg( 2 R'_h X^{\frac{1}{2}} \log X -4
R'_h X^{\frac{1}{2}} \bigg) + \tfrac{1}{k} R_{k, h}^k X^{k} +  O\big(\frac{X^{\frac{1}{2}}
\log X}{Y}\big) \\
  &\quad + O\big(\frac{X^{k}}{Y}\big)
  + O(X^{\frac{k-1}{2}})
  + O(X^{\frac{k}{2} + \epsilon} Y^{3k + \frac{17}{2} + 2\epsilon}).
\end{align}
The collected error terms can be written as
\begin{equation}
  O(X^{\frac{k-1}{2}})
  + O\big(\frac{X^{k+\epsilon}}{Y}\big)
  + O(X^{\frac{k}{2} + \epsilon} Y^{3k + \frac{17}{2} + 2\epsilon})
\end{equation}
Notice that the terms including $Y$ are the exact same expression as
in~\eqref{eq:hyp:short_interval_step_I}, and so the choice of $Y$ that optimizes the error
bound is the same!
That is, we choose $Y = X^{\lambda(k)}$ where $\lambda(k)$ is defined as in
Theorem~\ref{thm:hyp:concentrating_theorem_full}.

By combining this optimal error term and choice of $Y = X^{\lambda(k)}$ with the
expression~\eqref{eq:hyp:sharp_main_estimate_I} and the
bound~\eqref{eq:hyp:sharp_main_estimate_II}, we have shown that
\begin{align}
  &\frac{1}{2\pi i} \int_{(\sigma)} \sum_{m \in \mathbb{Z}} \frac{r_{2k+1}{(m^s +
h)}}{(2m^2 + h)^{s + \frac{k-1}{2}}} X^{s + \frac{k-1}{2}} \Phi_Y(s) ds \\
  &\quad =\sum_{\lvert 2m^2 + h \rvert \leq X} r_{2k+1}(m^2 + h) +O \bigg( \sum_{\lvert
  2m^2+h - X \rvert \leq X^{1 + \epsilon - \lambda(k)}} r_{2k+1}(m^2+h)\bigg)  \\
  &\quad = \delta_{[k = \frac{1}{2}]} \delta_{[h = a^2]} \bigg( 2 R'_h X^{\frac{1}{2}}
\log X -4 R'_h X^{\frac{1}{2}} \bigg) + \tfrac{1}{k} R_{k, h}^k X^{k} + O(X^{k + \epsilon
- \lambda(k)}).
\end{align}
Rearranging, this shows that
\begin{align}
  &\sum_{\lvert 2m^2 + h \rvert \leq X} r_{2k+1}(m^2 + h) \\
  &\quad = \delta_{[k = \frac{1}{2}]} \delta_{[h = a^2]} \bigg( 2 R'_h X^{\frac{1}{2}}
\log X -4 R'_h X^{\frac{1}{2}} \bigg) + \tfrac{1}{k} R_{k, h}^k X^{k} + O(X^{k + \epsilon
- \lambda(k)}) \\
  &\quad + O \bigg( \sum_{\lvert 2m^2+h - X \rvert \leq X^{1 + \epsilon - \lambda(k)}}
  r_{2k+1}(m^2+h)\bigg).
\end{align}
By Theorem~\ref{thm:hyp:concentrating_theorem_full}, the last big Oh term is bounded by
$O(X^{k + \epsilon - \lambda(k)})$.
This concludes the proof of the following theorem.

\begin{theorem}\label{thm:hyp:sharp_theorem_full}
  Let $k \geq \frac{1}{2}$ be a full or half-integer.
  Then
  \begin{align}
    &\sum_{\lvert 2m^2 + h - X \rvert \leq X} r_{2k+1}(m^2 + h) \\
    &\quad= \delta_{[k = \frac{1}{2}]} \delta_{[h = a^2]} \bigg( 2 R'_h X^{\frac{1}{2}}
  \log X -4 R'_h X^{\frac{1}{2}} \bigg) + \tfrac{1}{k} R_{k, h}^k X^{k} + O(X^{k +
\epsilon - \lambda(k)}),
  \end{align}
  where $\lambda(k)$ is defined as
  \begin{equation}
    \lambda(k) = \frac{1}{6 + \frac{19}{k}}.
  \end{equation}
  In particular, if $k = \frac{1}{2}$ then
  \begin{align}
    &\sum_{\lvert 2m^2 + h - X \rvert \leq X} r_{2k+1}(m^2 + h) \\
    &\quad = \delta_{[h = a^2]} \bigg( 2 R'_h X^{\frac{1}{2}} \log X -4 R'_h
  X^{\frac{1}{2}} \bigg) + 2R_{1/2, h}^{1/2} X^{\frac{1}{2}} + O(X^{\frac{1}{2} -
\frac{1}{44} + \epsilon}),
  \end{align}
  and for $k \geq 1$
  \begin{equation}
    \sum_{\lvert 2m^2 + h - X \rvert \leq X} r_{2k+1}(m^2 + h)  = \tfrac{1}{k} R_{k, h}^k
    X^{k} + O(X^{k + \epsilon - \lambda(k)}).
  \end{equation}
\end{theorem}

\begin{remark}
  It is interesting that the shape of the main term is different in the dimension $3$ case
  (when $k = \frac{1}{2}$) compared to all higher dimensions.
  There is a rough heuristic argument that explains this.
  Counting solutions to $X^2 + Y^2 = Z^2 + h$ is the same as counting solutions to
  $(Z-X)(Z+X) = Y^2 - h$, which can be thought of as counting the number of factorizations
  of $Y^2 - h$ as $Y$ varies.
  The number of factorizations of $Y^2 - h$ depends heavily on whether or not $h$ is a
  square.
  If $h$ is not a square, then there are expected to be relatively few factorizations.
  If $h$ is a square, then there are expected to be logarithmically many factorizations in
  $Y$.
  Thus if $Y$ varies up to size $\sqrt X$ and there are $\log Y$ many factorizations on
  average, we should expect $X^{\frac{1}{2}} \log X$ terms.

  In higher dimensions, this factorization heuristic doesn't apply.
  And in higher dimensions, the regularity of representation of integers as sums of many
  squares should smooth away irregularities present in low dimension.
\end{remark}


\clearpage{\pagestyle{empty}\cleardoublepage}

\chapter{An Application of Counting Lattice Points on Hyperboids}\label{c:hyperboloid_apps}

In this Chapter, we highlight one particular application of counting lattice points on
one-sheeted hyperboloids.
In particular, we describe how to understand sums of the form
\begin{equation}
  \sum_{n \leq X} d (n^2 + 1),
\end{equation}
where $d(n)$ denotes the number of divisors of $n$.

\section*{Connection to the Divisor Sum $\sum d(n^2 + 1)$}

The special three dimensional hyperboloid
\begin{equation}
  X^2 + Y^2 = Z^2 + 1
\end{equation}
is closely related to the divisor sum
\begin{equation}
  \sum_{n \leq R} d(n^2 + 1),
\end{equation}
which has been heavily studied by Hooley~\cite{Hooley63}.
This relationship is visible through the following theorem.

\begin{theorem}
  Let $d_o(n)$ denote the number of positive odd divisors of $n$.
  Then
  \begin{equation}
    \#\{\text{Integer points on } X^2 + Y^2 = Z^2 + 1 \; \text{with } \lvert Z \rvert \leq
    R \} = \sum_{n \leq R} 4 d_o(n^2 + 1).
  \end{equation}
\end{theorem}

To prove this theorem, we first prove this lemma.

\begin{lemma}
  Given an integer $Z$, we first show that $Z^2 + 1$ is not divisible by any prime
  congruent to $3 \bmod 4$.
\end{lemma}

\begin{proof}
  Indeed, factorize $Z + i$ as a product of Gaussian primes
  \begin{equation}
    Z + i = \pi_1^{k_1} \cdots \pi_r^{k_r}.
  \end{equation}
  Taking norms and letting $N(\pi_i) = p_i$, we have
  \begin{equation}
    Z^2 + 1 = p_1^{k_1} \cdots p_r^{k_r}.
  \end{equation}
  A Gaussian prime $\pi_j$ satisfies exactly one of the following:
  \begin{enumerate}
    \item $\pi_j = 1 + i$ or $\pi_j = 1 - i$ (in which case $\pi_j \mid 2$),
    \item $N(\pi_j) \equiv 1 \bmod 4$, or
    \item $\pi_j$ is inert, and is a prime congruent to $3 \bmod 4$ in $\mathbb{Z}$.
  \end{enumerate}
  Therefore if $Z^2+1$ is divisible by a prime congruent to $3 \bmod 4$, then there is an
  inert prime $\pi_j$ dividing $Z+i$ and (by conjugation) also $Z-i$.
  But then $\pi_j$ divides $Z+i - (Z-i) = 2i$, which is impossible as $\pi_j$ is inert and
  thus doesn't divide $2$.
  Therefore $Z^2+1$ is not divisible by a prime congruent to $3 \bmod 4$, and this
  concludes the proof of the sublemma.
\end{proof}

Returning to the proof of the theorem, it is a classical result that the number of ways of
writing a non-square $n$ as a sum of two squares is given by
$\frac{1}{2}(e_1+1)(e_2+1)\cdots(e_r+1)$, where $e_j$ is the multiplicity of the prime
$p_j$ congruent to $3 \bmod 4$ dividing $n$.
(This is under the assumption that $n$ can be written as the sum of two squares, which is
the case we are interested in).
In this formulation, note that we consider $X^2 + Y^2$ and $(-X)^2 + Y^2$ to be the same
representation.
Therefore the number of ways of writing $Z^2 + 1$ is $\frac{1}{2}(k_1+1)(k_2+1) \cdots
(k_r+1)$, excluding $2$ and its exponent from the list.
This is exactly half the number of odd divisors of $Z^2 + 1$, which we denote as
$\frac{1}{2}d_o(Z^2 + 1)$.

As each individual representation $X^2 + Y^2 = Z^2 + 1$ comes with the eight lattice
points $(\pm X, \pm Y, \pm Z)$, we see that the number of lattice point solutions to $X^2
+ Y^2 = Z^2 + 1$ with $\lvert Z \rvert \leq R$ is given by
\begin{equation}
  8 \sum_{Z \leq R}\tfrac{1}{2}d_o(Z^2+1).
\end{equation}
This completes the proof of the theorem. \qed{}

Note that when $n$ is even, all the divisors of $n^2 + 1$ are odd and $d(n^2 + 1) =
d_o(n^2 + 1)$.
On the other hand, when $n$ is odd, then $n^2 + 1$ is divisible by $2$ exactly once, and
$d(n^2 + 1)=2d_o(n^2 + 1)$.
Therefore it is possible to convert the summation to $\sum d(n^2 + 1)$.

As a corollary to the proof of the above theorem, one can prove the following.
\begin{corollary}
  \begin{align}
    \#\{\text{Integer points on } X^2 + Y^2 = (2Z)^2 + 1 \; \text{with } \lvert Z \rvert
    \leq R \} &= \sum_{n \leq R} 4 d_o(4n^2 + 1) \\
    &= \sum_{n \leq R} 4 d(4n^2 + 1).
  \end{align}
\end{corollary}
While this hyperboloid is not studied in this thesis, the methodology still applies and
it is possible to obtain asymptotics with error term for this lattice counting problem as
well.

Further, if we denote
\begin{align}
  N_1(R) &:= \#\{\text{Integer points on } X^2 + Y^2 = Z^2 + 1 \; \text{with } \lvert Z
  \rvert \leq R \}
  \\
  N_2(R) &:= \#\{\text{Integer points on } X^2 + Y^2 = (2Z)^2 + 1 \; \text{with } \lvert Z
  \rvert \leq R \},
\end{align}
then one can now easily compute that
\begin{equation}
  \sum_{n \leq R} d(n^2 + 1) = \frac{N_1(R)}{2} - \frac{N_2(R/2)}{4}.
\end{equation}
In this way, we convert a classical, and still somewhat mysterious, divisor sum into two
lattice counting problems.

In the future, it would be a good idea to optimize the arguments in
Chapter~\ref{c:hyperboloid} in order to try to improve estimates for divisor sums.
In particular, in the integral analysis for proving the main theorems, it is possible to
further shift lines of integration and handle resulting residue terms (although one would
also need to get a deeper understanding of the underlying analytic behavior).

\begin{remark}
  Similar techniques may be applied to study
  \begin{equation}
    \sum_{n \leq R} d(n^2 + h)
  \end{equation}
  for any positive $h$, although the computations look progressively messier.
\end{remark}


\clearpage{\pagestyle{empty}\cleardoublepage}

\chapter{Conclusion and Directions for Further Investigation}\label{c:conclusion}

In the preceding, we considered two problems closely related to the classical Gauss
Circle problem.

We first considered partial sums $S_f(n) = \sum_{m \leq n} a(n)$ of Fourier coefficients
of a cusp form $f(z) = \sum a(n) e(nz)$, and showed that on average these partial sums
behave like the error term in the Gauss Circle problem.
To study these partial sums, we constructed the Dirichlet series
\begin{equation}
  D(s, S_f \times S_g) = \sum_{n \geq 1} \frac{S_f(n) S_g(n)}{n^{s + k - 1}},
\end{equation}
as well as some related Dirichlet series, and showed that these Dirichlet series have
meromorphic continuation to the plane.

The primary challenge comes from understanding the properties of the shifted convolution
sum
\begin{equation}
  \sum_{n,m \geq 1} \frac{a(n+m) a(n)}{(n+m)^s n^w}.
\end{equation}
In Chapter~\ref{c:sums}, we approached this sum through a spectral expansion of a Poincare
series appearing in an inner product against an appropriately chosen product of cuspforms.
We showed that this spectral expansion is, to a large extent, explicitly understandable.
By relating the properties of partial sums of coefficients of cusp forms to a particular
spectral expansion, we explicitly relate the arithmetic properties of the coefficients to
the very analytic properties of the spectrum of the hyperbolic Laplacian.

It is interesting to reflect on the successes and limitations of this approach.
In Chapters~\ref{c:sums} and~\ref{c:sums_apps}, we showed that we are now able to prove
many improvements to classical results.
But we were unable to improve the estimate for the size of a single partial sum.
Instead, we are only capable of matching the estimate due to Hafner and Ivi\'c that
\begin{equation}
  S_f(X) \ll X^{\frac{k-1}{2} + \frac{1}{3} + \epsilon}.
\end{equation}
From the point of view of this thesis, the obstacle to further improvement was directly
seen to be lack of understanding of the discrete spectrum (in terms of their distribution
and cancellation) and the Lindel\"{o}f Hypothesis (or rather subconvexity) type bounds for
Rankin-Selberg convolutions.
Fundamentally, the techniques in this thesis are very different from previously attempted
techniques --- thus the similarity in the final bounds heuristically seems to represent an
absolute bound on our current understanding.

In Chapter~\ref{c:sums_apps} we noted several other applications of the Dirichlet series
$D(s, S_f \times S_g)$, including several projects that have already led to published
papers~\cite{hkldw, hkldwShort, hkldwSigns}.
We have shown that the Dirichlet series $D(s, S_f \times S_g)$ present powerful tools to
study a variety of questions related to the size, shape, and behavior of sums of
coefficients of cusp forms.
Through analysis of $D(s, S_f \times S_g)$, it is possible to prove results on long
interval estimates, short interval estimates, and sign changes.
More generally, one can apply many different Mellin integral transforms to $D(s, S_f
\times S_g)$ in order to directly study different aspects.

One avenue of exploration that my collaborators and I have begun to explore it to perform
an analogous construction of $D(s, S_f \times S_f)$ in cases when $f$ is not a cusp form.
This is proving to be a very interesting direction, and we will be able to explore more
and different variants of the Gauss Circle problem.
One particular direction is discussed at the very end of Chapter~\ref{c:sums_apps}.

There is another interesting direction for further exploration, based on the techniques
and observations from Chapters~\ref{c:sums} and~\ref{c:sums_apps}.
As an individual Fourier coefficient is roughly of size $a(n) \sim n^{\frac{k-1}{2}}$ and
the partial sum appears to satisfy $S_f(n) \ll n^{\frac{k-1}{2} + \frac{1}{4}}$, there is
a large amount of cancellation among individual coefficients.

But I ask the following question: What if we consider partial sums formed from the partial
sums $S_f(n)$?
That is, what should we expect from the sizes of
\begin{equation}
  \sum_{n \leq X} S_f(n)?
\end{equation}
Further, what if we iterate this process and consider sums of sums of sums, and so on?
Initial numerical investigation suggests that there continues to be extreme amounts of
cancellation, far more than would occur by merely random chance.

In Chapter~\ref{c:hyperboloid}, we considered the question of how many points lie within
the $d$-dimensional sphere of radius $\sqrt R$ and on the one-sheeted hyperboloid
\begin{equation}
  \mathcal{H}_{h,d} = X_1^2 + \cdots + X_{d-1}^2 = X_d^2 + h,
\end{equation}
which is essentially a constrained Gauss Sphere problem.
We were able to prove improved bounds and asymptotics for this number of points.

On comparing the main techniques of Chapter~\ref{c:hyperboloid} with those of
Chapter~\ref{c:sums}, it is clear that there are many similarities.
In both, we reduce the study towards sufficient understanding of a carefully chosen
shifted convolution sum.
And in both, we understand the convolution sum by translating the properties into
properties of functions associated to the spectrum of the hyperbolic Laplacian.

The analysis and proof of the main theorem of Chapter~\ref{c:hyperboloid} is not
completely optimized.
In particular, it is possible to perform a very close and detailed analytic argument,
similar to the argument appearing in~\cite{hkldwShort}, in order to further improve the
bound on the main error term of the lattice point estimate.
It may even be possible to improve estimates for the divisor sum $\sum d(n^2 + 1)$, whose
connection with lattice points on hyperboloids is explained in
Chapter~\ref{c:hyperboloid_apps}.

More generally, the ideas and techniques of Chapter~\ref{c:hyperboloid} can be extended to
more general products of theta functions.
By replacing the Jacobi theta function with theta functions associated to different
quadratic forms, it should be possible to understand a wide variety of quadratic surfaces.


\clearpage{\pagestyle{empty}\cleardoublepage}

\chapter{Original Motivations}\label{c:motivations}

Let me describe briefly what we set out to prove and how different the final results
actually are.\footnote{Jeff Hoffstein once told me that far too few published works
  describe the difference between the original intent and the final version.
I include that here, in this very informal section.}

Chapters~\ref{c:sums} and~\ref{c:sums_apps} were inspired from a single question that Jeff
Hoffstein asked after a talk from Winfried Kohnen during Jeff's Birthday Conference.
Kohnen described new results on sign changes of cusp forms, and Jeff asked whether or not
it was possible to prove that sums of coefficients of cusp forms change sign frequently.

Very na{\"\i}vely, I thought that if I could understand this question, I would probably
use the series $D(s, S_f) = \sum_{n \geq 1} S_f(n) n^{-s}$ and $D(s, S_f \times
\overline{S_f}) = \sum_{n \geq 1} \lvert S_f(n) \rvert^s n^{-s}$.

The reason for this is simple: one way to determine that a sequence changes signs often is
to show that partial sums of squares are large while partial sums are small (indicating
lots of cancellation).
This fundamental idea was included within Kohnen's talk.
As for the use of these Dirichlet series --- it's in some sense the first thing that a
multiplicative number theorist might try.

I doodled on this question for a few talks, and shared my doodles with Tom Hulse.
Tom had two quick ideas: firstly, he knew of an exact set of conditions that guarantee
sign changes in intervals; secondly, he knew a Mellin-Barnes transform to decouple
denominators. This seemed like a short, quick project, so we set out to prove sign
changes.

We were wrong.
This was not short nor quick, and as we delved into the problem it became apparent that
there were significant obstacles in the way of understanding the spectral
analysis.
In some sense, we knew that these were understandable from the general philosophy
of~\cite{HoffsteinHulse13}.
Yet actually demonstrating the extent of cancellation required attention to different
subtleties.
And I engrossed myself into these details.

Several months later, the question had totally shifted away from sign changes, and instead
focused on the Cusp Form analogy to the Gauss Circle problem in various aspects, leading
ultimately to the current line of research.
In hindsight, it turns out that our analysis of $D(s, S_f)$ and $D(s, S_f \times
\overline{S_f})$ have had only limited success in actually proving sign changes --- but
they are great tools otherwise.

Chapters~\ref{c:hyperboloid} and~\ref{c:hyperboloid_apps} accomplished almost exactly what
was originally intended.
In~\cite{hulseCountingSquare}, Hulse, K{\i}ral, Kuan, and Lim studied a problem
inspired from the same work of Oh and Shah~\cite{ohshah2014} that studied lattice points
on hyperboloids through ergodic methods.
Shortly afterwards, a K{\i}ral, Kuan, and I began to look at lattice points on
hyperboloids.
Our initial investigations stalled, though I first learned much about the general
techniques involving shifted convolution sums from those first forays into this problem.

It is interesting to note that the combination of integral transforms that leads to the
main theorem of Chapter~\ref{c:hyperboloid} was noticed by a combination of Kuan, Walker,
and I originally while we were writing our sign changes paper~\cite{hkldwSigns}.
We then forgot about this technique until we began to struggle with our forthcoming
work on the Gauss Sphere problem, until Walker and I slowly realized that we were
vastly overcomplicating a particular difficulty.

Originally, the intention of Chapter~\ref{c:hyperboloid} was simply to prove the
meromorphic continuation of the underlying Dirichlet series and to prove the smoothed sum
result.
This is a rare case of proving exactly what I had set out to prove, and then a little
bit more.\footnote{As opposed to the normal pattern of failing in the original goal, and
  then proving something else entirely.
  As Jeff likes to say, it is important to be able to love the theorem that you can prove,
as rarely can we prove the theorems that we love.}


\clearpage{\pagestyle{empty}\cleardoublepage}


\addcontentsline{toc}{chapter}{Index}
\begin{spacing}{0.9}
  \printindex
\end{spacing}


\addcontentsline{toc}{chapter}{Bibliography}
\bibliographystyle{alpha}
\begin{spacing}{0.9}
  \bibliography{thesisbiblio}
\end{spacing}

\end{document}